\documentclass[11pt]{article}

\usepackage{fullpage}
\usepackage{amssymb}
\usepackage{amsmath}
\usepackage{amsfonts}
\usepackage{amsthm}
\usepackage{url}
\usepackage{graphicx}
\usepackage{subfig}
\usepackage{fancyhdr}
\usepackage{url}

\usepackage[utf8]{inputenc}
\usepackage[english]{babel}
\usepackage{hyperref}

\newcommand{\nc}{\newcommand}
\nc{\heading}[1]{\begin{center} \large \bf #1 \end{center}}

\newcommand{\oA}{\overline{A}}
\newcommand{\oB}{\overline{B}}
\newcommand{\oI}{\overline{I}}

\newcommand{\oR}{\overline{R}}

\newcommand{\Ex}{\mathsf{E}}
\renewcommand{\Pr}{\mathsf{P}}

\theoremstyle{plain}
\newtheorem{prop}{Proposition}

\newtheorem{corollary}{Corollary}

\usepackage{glossaries}
\makeglossaries
\newglossaryentry{latex}
{
    name=latex,
    description={Is a mark up language specially suited
    for scientific documents}
}
\newglossaryentry{maths}
{
    name=mathematics,
    description={Mathematics is what mathematicians do}
}
\begin{document}


\date{March 30, 2021}
\title{Stochastic Modeling of an Infectious Disease \\
 {Part III-B: Analysis of the Time-Nonhomogeneous BDI Process and Simulation Experiments of both BD and BDI Processes.
 } }
\author{Hisashi Kobayashi\footnote{The Sherman Fairchild University Professor of Electrical Engineering and Computer Science, Emeritus.
Email: Hisashi@Princeton.EDU, Website: \url{https://hp.hisashikobayashi.com/}, Wikipedia: \url{https://en.wikipedia.org/wiki/Hisashi_Kobayashi} 
} \\
   Dept. of  Electrical Engineering \\
   Princeton University \\
   Princeton, NJ 08544, U.S.A.}

\maketitle

\begin{abstract}
In Section 1, we revisit the partial differential equation (PDE) for the probability generating function (PGF) of the time-nonhomogeneous BDI (birth-and-death-with-immigration) process and derive a closed form solution. To the best of our knowledge, this is a new mathematical result.\footnote{See Bailey \cite{bailey:1964}, p. 115: ``In general the solution of such an equation is likely to be difficult to obtain, and we shall not undertake the discussion of any special cases here."} We state this result as Proposition \ref{prop:nonhomo-BDI}.  We state as Corollary \ref{coro:nonhomo-BDI-NBD} that the \emph{negative binomial distribution} of the time-homogeneous BDI process discussed in Part I \cite{kobayashi:2020a-arXiv} extends to the general time-nonhomogeneous case, provided that the ratio of the immigration rate $\nu(t)$ to the birth rate $\lambda(t)$ is a constant. 

In section \ref{subsec:Bartlett-Bailey}, we take up the heuristic approach discussed by Bartlett and Bailey (1964), and carry it out to completion by arriving at the solution obtained above,.

In Section \ref{sec:nonhomo-BD-simulation}, we present the results of our extensive simulation experiments of the time-nonhomogeneous BD process that was analyzed in Part III-A \cite{kobayashi:2021a} and confirm our analytic results.

In Section \ref{sec:nonhomo-BDI-simulation}, we undertake similar simulation experiments for the BDI process that is analyzed in Section \ref{sec:non-homo-BDI}.

As we discuss in Section \ref{sec:discussion}, our stochastic model now seems more promising and powerful than has been heretofore expected.   

In Appendix B, a closed form solution for the M(t)/M(t)/$\infty$\footnote{The first M(t) stands for the time-varying Poisson arrivals and the second M(t) means the exponential service time, which is also time varying.   This notation should not be confused with the function $M(t)=\int_0^t \mu(u)e^{-s(u)}\,du$ defined in (60) of \cite{kobayashi:2021a}.} queue is obtained, as a special case of this BDI process model.

\end{abstract}

\tableofcontents

\section{General Solution for the Time-Nonhomogenous BDI Process Model}\label{sec:non-homo-BDI}

In this section we will obtain closed form expressions for the PGF (probabilty generating function) and the time-dependent probability distribution for the general time-nonhomogeneous birth-death process with immigration. To the best of our knowledge, these are new results: the closed form solution has been heretofore believed to be difficult to obtain for the general time-nonhomogeneous BDI process \footnote{See Bailey \cite{bailey:1964}, p. 115.} 

\subsection{Revisit the Partial Differential Equation for the PGF}\label{subsec:Revisit-PDE}
Recall that in \cite{kobayashi:2021a} the PDE (partial differential equation) (50), we first worked with the left and middle term of the auxiliary differential equations (52), ibid.
\begin{align}
dt= - \frac{dz}{(\lambda(t)z-\mu(t))(z-1)}=\frac{dG}{\nu(t)(z-1)G(z,t)},\label{52-Part-IIIA}
\end{align}
and obtained the first solution (58), ibid.
\begin{equation}
 \fbox{
\begin{minipage}{4.5cm}
\[
\frac{e^{-s(t)}}{z-1}-L(t)=C_1.   
\] 
\end{minipage}
} \label{1st-solution-IIIA}
\end{equation}

Let us now equate the leftmost and rightmost terms of (\ref{52-Part-IIIA}):
\begin{align}
\frac{dG(z,t)}{\nu(t)G(z,t)(z-1)}=dt,
\end{align} 
which can be written as
\begin{align}
\frac{d\log G(z, t)}{dt}=\nu(t)(z-1). \label{dlogG}
\end{align}
A critical next step is to substitute  (\ref{1st-solution-IIIA}) into (\ref{dlogG}), leading to an  expression which includes $C_1$, but not the variable $z$
\begin{align}
\frac{d\log G(z, t)}{dt}=\frac{\nu(t)e^{-s(t)}}{C_1+L(t)}, \label{eqn[without=z}
\end{align}
which readily leads to the second solution:
\begin{equation}
 \fbox{
\begin{minipage}{8cm}
\[
G(z,t)\exp\left(-\int_0^t\frac{\nu(u)e^{-s(u)}}{C_1+L(u)}\,du\right)=C_2.  
\] 
\end{minipage}
} \label{voila-C_2} 
\end{equation}
This form of the second solution is unusual in the sense that $C_1$ of the first solution appears in this solution, and this might have deterred Bailey and other investigators from proceeding further.

But by writing the functional relation between $C_1$ and $C_2$ as 
\begin{align}
C_2=f(C_1), \label{C_2-C_1}
\end{align}
and setting $t=0$ in (\ref{1st-solution-IIIA}) and (\ref{voila-C_2}), and using the initial condition $G(z,0)=z^{I_0}$, we arrive at
\begin{align}
f\left(\frac{1}{z-1}\right)=G(z,0)=z^{I_0},
\end{align}
from which we find
\begin{equation}
 \fbox{
\begin{minipage}{4cm}
\[
f(y)=\left(1+\frac{1}{y}\right)^{I_0}.  
\] 
\end{minipage}
} \label{func-form-f} 
\end{equation}
It will be worth noting that this functional form is exactly the same as (69) of \cite{kobayashi:2021a} that we obtained for the BD process. 

From (\ref{1st-solution-IIIA}), (\ref{voila-C_2}), (\ref{C_2-C_1}) and (\ref{func-form-f}), we finally obtain what we have been after, which we state below as a proposition:

\begin{prop}[The representation of the PGF of the general time-nonhomogeneous BDI process]\label{prop:nonhomo-BDI}
The PGF of the BDI process with the initial value is given by
\begin{equation}
\fbox{
\begin{minipage}{8cm}
\[
G_{BDI:I_0}(z,t)=G_{BD:I_0}(z,t) G_{ID:0}(z,t),  
\] 
\end{minipage}
} \label{G_BDI}
\end{equation}
where the first term in RHS (right hand side) is equivalent to the PGF of the BD (birth-death process) with the initial value of $I_0$, but no immigration in $(0, t)$ \footnote{The time-homogeneous version of the BD process is also referred to as the FA (Feller-Arley) process. See \cite{kobayashi:2020b-arXiv} Section 3.1.}, given in (70) of \cite{kobayashi:2021a},
\begin{align}
G_{BD:I_0}(z, t)\triangleq \left(1+\frac{e^{s(t)}(z-1)}{1-e^{s(t)}L(t)(z-1)}\right)^{I_0}.
\label{G_BD}
\end{align}
and the second term in the RHS of (\ref{G_BDI}) represents the PGF contributed by the \emph{immigrants and their descendants}, which we denote $G_{ID:0}(z,t)$, with 0 indicating no immigrants at time 0:
\begin{equation}
\fbox{
\begin{minipage}{10cm}
\[
G_{ID:0}(z, t)=\exp\left(\int_0^t\frac{\nu(u)e^{s(t)-s(u)}(z-1)}{1-e^{s(t)}(L(t)-L(u))(z-1)}\,du\right),
\] 
\end{minipage}
} \label{G_Immigration}
\end{equation}

Thus, the BDI process can be expressed as  
\begin{align}
I_{BDI:I_0}(t)=I_{BD:I_0}(t)+I_{ID:0}(t). \label{BDI-two-processes}
\end{align}
where the component processes are statistically independent.  $\Box$
\end{prop}

On using the functions $\alpha(t)$ and $\beta$ defined earlier, \footnote{Recall the definitions 
\begin{align}
\alpha(t)\triangleq \frac{M(t)}{1+M(t)},~~~\mbox{and}~~~\beta(t) \triangleq \frac{L(t)}{1+M(t)},\label{def-alpha-beta}
\end{align}
with $L(t)$ and $M(t)$ being given by
\begin{align}
L(t)\triangleq \int_0^t \lambda(u) e^{-s(u)}\,du, ~~~\mbox{and}~~~M(t)\triangleq \int_0^t\mu(u)e^{-s(u)}\,du,       \label{def-L(t)-M(t)-nonhomo}
\end{align}
and
\begin{align}
s(t)\triangleq \int_0^t a(u)\,du=\int_0^t (\lambda(u)-\mu(u))\,du.  \label{def-s(t)}
\end{align}}
we find an alternative expression of (\ref{G_BD}), as we obtained in (78) of \cite{kobayashi:2021a}:
\begin{align}
G_{BD:I_0}(z, t)=\left(\frac{\alpha(t)+(1-\alpha(t)-\beta(t))z}{1-\beta(t)z}\right)^{I_0}.
\end{align}

In order to better understand the new term (\ref{G_Immigration}), we write the denominator of the integrand in (\ref{G_Immigration}) as
\begin{align}
A(u)\triangleq 1-e^{s(t)}(L(t)-L(u))(z-1),  \label{def-A(u)}
\end{align}
and its derivative w.r.t. the variable $u$ as
\begin{align}
A'(u)=\lambda(u)e^{s(t)-s(u)}(z-1). \label{A'(u)}
\end{align}
Thus, we can transform (\ref{G_Immigration}) to
\begin{align}
G_{ID:0}(z, t)&=\exp\left(\int_0^t r(u)\frac{A'(u)}{A(u)}\,du\right),~~~\mbox{where}~~~r(u)=\frac{\nu(u)}{\lambda(u)}.\label{G_I-with-r(u)}
\end{align}
On defining $B(u)=\log A(u)$, we have
\begin{align}
G_{I:0}(z,t)&=\exp\left(\int_0^t r(u)B'(u)\,du\right)=\exp\left(\left[r(u)B(u)\right]_{u=0}^t-\int_0^t r'(u)B(u)\,du\right)\nonumber\\
&= \exp\left(\left[\log A(u)^{r(u)}\right]_{u=0}^t \right)\cdot\exp\left(-\int_0^t r'(u)\log A(u)\,du\right)
\label{G_I-near-final}
\end{align}
Noting that
\begin{align}
A(t)=1,~~\mbox{and}~~~A(0)=1-e^{s(t)}L(t)(z-1),
\end{align} 
we arrive at the following expression:
\begin{equation}
\fbox{
\begin{minipage}{8.5cm}
\[
G_{ID:0}(z,t)=G_{NB(r(0),\beta(t))}G_C(z, t) ,
\]
\end{minipage}
}  \label{G_ID-our-result}
\end{equation}
where the first term
\begin{align}
G_{NB(r(0),\beta(t))}&=\left(\frac{1}{1-e^{s(t)}L(t)(z-1)}\right)^{r(0)}
= \left(\frac{1-\beta(t)}{1-\beta(t)z}\right)^{r(0)},
\end{align}
is the PGF of a NBD (negative binomial distributed) process with the parameters $(r(0), \beta(t))$.\footnote{In order to arrive at the second expression, we use the identity $L(t)-M(t)=1-e^{-s(t)}$ as shown in (61) \cite{kobayashi:2021a} .} The second term of the RHS of (\ref{G_ID-our-result}) defined by
\begin{align}
G_C(z,t)& \triangleq \exp\left(-\int_0^tr'(u) \log A(u)\,du\right),\label{C(z,t)}
\end{align}
expresses the effect of $r(t)$ not being a constant. Clearly if $r'(t)=0$, then $C(z,t)=1$ for all $z$ and $t$.

We state the above result, which we believe to be new, as a corollary to the above proposition:
\begin{corollary}[A time-nonhomogeneous BDI process is negative binomial distributed when $I_0=0$ and $r(t)=r$]\label{coro:nonhomo-BDI-NBD}

If the function $r(t)=\frac{\nu(t)}{\lambda(t)}$ is a constant $r$ for all $t$, then the BDI process with $I_0=0$ is a NBD (negative binomial distributed) process with parameters $(r, \beta(t))$.
\begin{align}
G_{BDI:0} (z,t)=\left(\frac{1-\beta(t)}{1-\beta(t) z}\right)^r,~~~\mbox{when}~~r(t)=r. \label{NBD-of-BDI}
\end{align}
\end{corollary}
\begin{proof}
If $r(t)=r$, then $r'(u)=0$ and thus, $C(t)=1$. Therefore,
\begin{align}
G_{BDI:0} (z,t)=G_{NB(r,\beta(t))}(z,t).  \label{NBD-and-L(t)}
\end{align}
\end{proof}
Note  (\ref{NBD-of-BDI}) is a generalization of (44) of \cite{kobayashi:2020a-arXiv} that holds for the time-homogeneous case. The corresponding PMF (probability mass function) takes the same form as (47), ibid.
\begin{align}
P^{(BDI:0)}_k(t)={k+r-1 \choose k}(1-\beta(t))^r \beta(t)^k, ~~~k= 0,1,2, \ldots.\label{PMF-nonhomo}
\end{align}

The assumption $r(t)$ being a constant does not imply that $\nu(t)$ and $\lambda(t)$ must be constants. In fact, it makes a perfect sense to consider the case where $\nu(t)= r \lambda(t)$: When a government asks its citizens to reduce their social activities to curtail the pandemic, the government should also tighten its security at its borders such as airports, seaports, etc.  In our simulation, we will consider the case in which $\nu(t)$ is changed in concert with $\lambda(t)$.  We fully evaluate, in \cite{kobayashi:2021bb}, the correction factor $C(z,t)$ when $\nu$ deviates from $r\lambda(t)$.

\subsection{Completing the Bartlett-Bailey Approach}
\label{subsec:Bartlett-Bailey}
In this section we take up ``Section 9.4: The effect of immigration" of Bailey \cite{bailey:1964}, pp. 115-116, and extend his heuristic approach and arrive at our closed form solution obtained in the preceding section.  It is not clear whether this heuristic approach is Bailey's own work.  In Bartlett's 3rd edition (1978), (see \cite{bartlett:1978},``Section 3-41: The effect of immigration (pp. 82-83),") he also describes this approach.  Judging from that the 1st edition of Bartlett's book goes back to 1955, it is very likely that the materials in Bailey ibid. may have been originally conceived by Bartlett. They both considered how to incorporate the effect of immigrants based on the result obtained in 1948 by Kendall on the time-nonhomegenous BD (birth-and-death) process \cite{kendall:1948a}.

Let $N_{ID}(u,t)$ be the number of the descendants of an immigrant who arrived at $u$ and are alive at $t$.  The immigrant who is still alive should be included in $N_{ID}(u,t)$. Let the $G_{ID}(z, u,t)$ be the PGF of $N_{ID}(u,t)$:
\begin{align}
G_{ID}(z,u, t)\triangleq \Ex[z^{N_{ID}(u, t)}], ~~u\leq t.\label{G_I(z,t;u)}
\end{align}

Recall that the PGF of the time-nonhomogeneous BD process, with the initial population $I_0$, is given by (70) of \cite{kobayashi:2021a}:
\begin{align}
G_{BD:I_0}(z, t)=\left(1+\frac{1}{\frac{e^{-s(t)}}{z-1}-L(t)}\right)^{I_0}. \label{PGF:BD}
\end{align}
Suppose that an immigrant arrives at t=0.  Then its descendants' PGF $G_{ID}(z,0,t)$ can be found by setting $I_0=1$ in the above formula:
\begin{align}
G_{ID}(z,0,t)=G_{BD:1}(z,t)=1+\frac{1}{\frac{e^{-s(t)}}{z-1}-L(t)}. \label{G_I(z, 0, t)}
\end{align}
If an immigrant arrives at time $u$, we just shift the above from zero to $u$, obtaining
\begin{align}
G_{ID}(z,u, t)=1+\frac{1}{\frac{e^{-s(u,t)}}{z-1}-L(u, t)}.\label{G_I(z,u,t)}
\end{align}
where
\begin{align}
s(u, t)&\triangleq \int_u^t (\lambda(\tau)-\mu(\tau))\,d\tau \label{def-s(u,t)}\\
L(u, t)&\triangleq \int_u^t \lambda(\tau)e^{-s(u,\tau)}\,d\tau, \label{def-L(u,t)}
\end{align}

Consider an infinitesimal interval $[u, u+du)$ within the interval $[0, t)$.  The probability of having an immigrant in this interval, and that of having no immigrant are given by $\nu(u)\,du$ and $1-\nu(u)\,du$, respectively.  Thus, the PGF at time $t$ due to the contribution from the $N_{ID}(u, t)$ descendants is given as
\begin{align}
\nu(u)G_{ID}(z,u, t)\,du+(1-\nu(u)du)&=1+\nu(u)(G_{ID}(z,u,t)-1)\,dr,  \label{G-I-u}
\end{align}
where the term $1-\nu(u)du$ contributes to the $z^0$ terms of the PGF.  
Let us divide the interval $[0, t)$ into $\cup_i [u_i, u_i+du)$ such that $u_1=0, u_2=du, u_3=2du, \cdots, u_i=u_{i-1}+du=(i-1)du, \cdots$.  In other words, $[0, t)$ is segmented into $N(du)=t/du$ infinitesimal intervals, which are contiguous but disjoint. Clearly, as $du\to 0$, $N(du)\to\infty$. The events (i.e, arrival or non-arrival) in different intervals are statistically independent to each other.  Thus, the total number of the immigrant's descendants, denoted $N(t)$, can be written as\footnote{If there is no immigrant arriving in the interval $[u_i,u_i+du)$, $N_{ID}(u, t_i)=0$ for all $t$.}
\begin{align}
N(t)&=\sum_{i=1}^{N(dr)} N_{ID}(u_i, t).
\end{align}
Then the immigrants and their descendants' contributions to the PGF, denoted $G_{ID}(z,t)$, should be given in the limit $du\to 0$ as,
\begin{align}
G_{ID}(z, t)=\lim_{du\to 0}\prod_{i=1}^{N(du)}\left\{1+\nu(u_i)(G_{ID}(z,u_i, t)-1)\,du\right\}.
\end{align}
On taking the natural logarithm of the above, we have
\begin{align}
\log G_{ID}(z,t)&=\lim_{du\to 0}\sum_{i=1}^{N(du)}\log\left\{ 1+\left[\nu(u_i)(G_{ID}(z,u_i, t)-1)\right]\,du\right\}\nonumber\\
&=\lim_{du\to 0}\sum_{i=1}^{N(du)}\nu(u_i)(G_{ID}(z,u_i, t)-1)\,du\nonumber\\
&=\int_0^t \nu(u)(G_{ID}(z, u, t)-1)\,du,
\end{align}
which leads to
\begin{equation}
\fbox{
\begin{minipage}{8.5cm}
\[
G_{ID}(z,t)=\exp\left(\int_0^t \nu(u)(G_{ID}(z,u, t)-1)\,du\right).
\]
\end{minipage}
} \label{Batlett-Bailey}
\end{equation}
This is where Bailey ends (see ibid, p. 116 Eqn, (9.51)).  Bartlett does not give explicitly this expression, but he states steps to arrive at this equation.  

In order to demonstrate the equivalence of the solution form of $G_{ID}(z, t)$ given by (\ref{Batlett-Bailey}) to our result (\ref{G_ID-our-result}), substitute using (\ref{G_I(z,u,t)}) into the above, obtaining
\begin{align}
G_{ID}(z, t)=\exp\left(\int_0^t\frac{\nu(u)}{\frac{e^{-s(u,t)}}{z-1}-L(u, t)}\,du\right)   \label{G_I-2}
\end{align}
It is straightforward to see
\begin{align}
s(u,t)&=s(t)-s(u), 
\end{align}
but it is important to note that $L(u,t)\neq L(t)-L(u)$, but instead
\begin{align}
L(u,t)&=\int_u^t\lambda(\tau)e^{-s(u,\tau)}\,d\tau=e^{s(u)}(L(t)-L(u)),\label{L(u,t)}
\end{align}
Then (\ref{G_I-2}) can be written as
\begin{align}
G_{ID}(z, t) &=\exp\left(\int_0^t\frac{\nu(u)e^{s(t)-s(u)}(z-1)}{1-e^{s(t)}(L(t)-L(u))(z-1)}\,du\right) \label{G_I-3}
\end{align}
which is equivalent to (\ref{G_Immigration}).  Then the rest of the steps to arrive at the important result 
(\ref{G_ID-our-result}) should be straightforward.  We suspect that the pitfall (\ref{L(u,t)}) might have prevented Bailey and others from completing this insightful heuristic approach.

\subsection{The Mean and Coefficient of Variation of the BDI Process}

As for the internal infection rate $\lambda(t)$, we consider the same example as that in \cite{kobayashi:2021a}, Figures 1 \& 2 with $d=5$ (which corresponds to the curve plotted in red). The recovery rate  $\mu(t)$ i assumed to be constant as before, i.e., $\mu_0=0.1$. 
 
As Corollary \ref{coro:nonhomo-BDI-NBD} implies, if we choose the function 
$\nu(t)=r \lambda(t)$, where $r=\frac{\nu_0}{\lambda_0}=2/3$, then with the initial condition $I_0=0$, the BDI process $I_{BDI:0}(t)$ reduces to the immigration process $I_{I:0}(t)$, which is negative binomial distributed: NB$(r, \beta(t))$.  Thus, we have 
\begin{align}
\oI_{BDI:0}(t)= \frac{r\beta(t)}{1-\beta(t)}. \label{unstable-mean-BDI:0}
\end{align}
But the formula (\ref{unstable-mean-BDI:0}) is numerically unstable in the region where $\beta(t)$ is close to unity.  A better computational formula is given by (14) of \cite{kobayashi:2021a} with $I_0=0$, i.e.,
\begin{align}
\oI_{BDI:0}(t)= e^{s(t)}N(t),  \label{stable-mean-BDI:0}
\end{align}
where $N(t)$ is defined in (15), ibid:
\begin{align}
N(t)\triangleq \int_0^t\nu(t)e^{-s(u)}\,du.\label{def-N(t)}
\end{align}
The equivalence of the above to (\ref{unstable-mean-BDI:0}) can be shown from the definition of $\beta(t)$ and the aforementioned identity $L(t)=M(t)+1-s^{-s(t)}$.

\begin{figure}[thb]
\begin{minipage}[t]{0.45\textwidth}
\centering
\includegraphics[width=\textwidth]{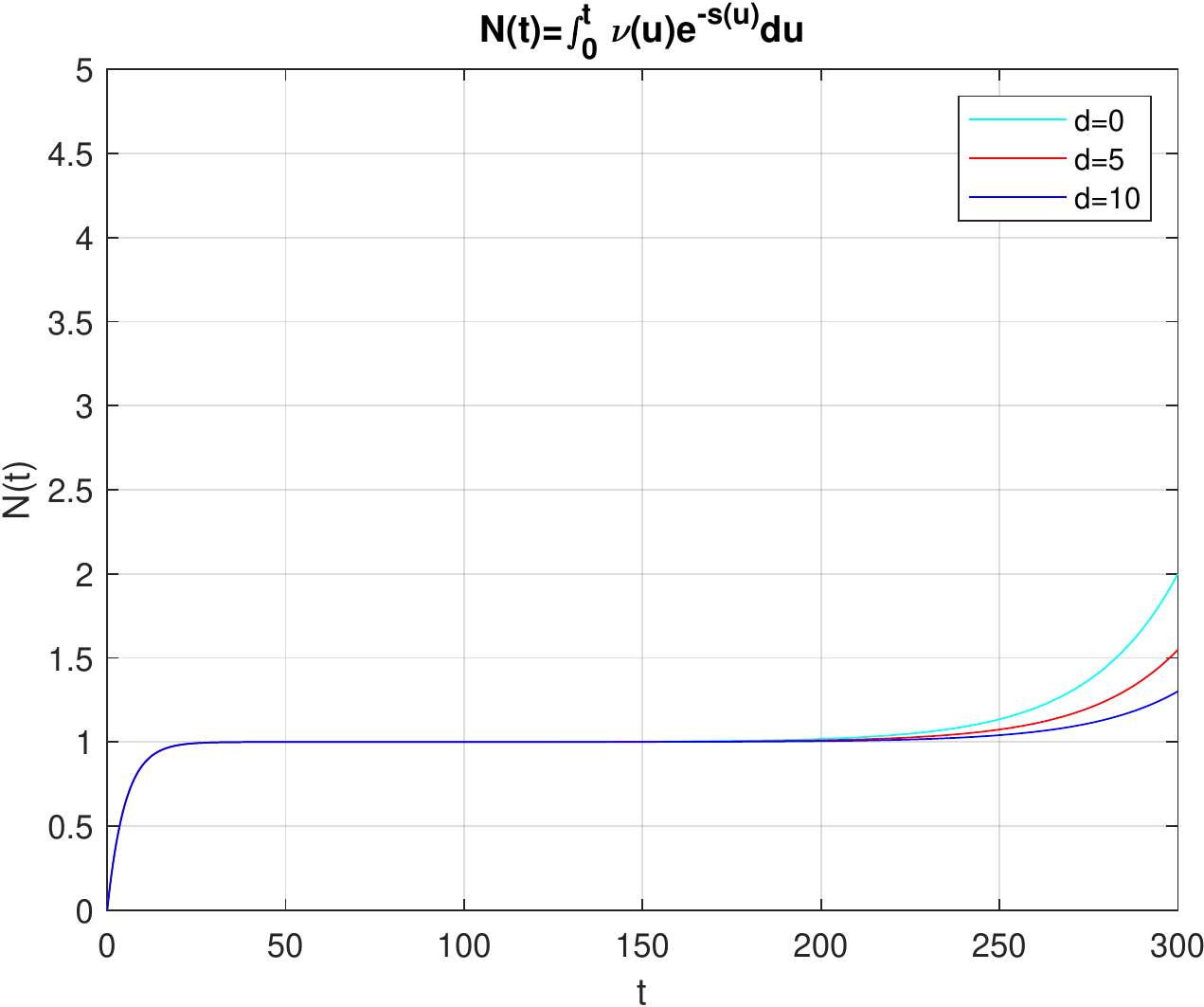}
\caption{\sf The function $N(t)$ when $\nu(t)=r\lambda(t)$.}
\label{fig:N(t)-BDI-r-lambda}
\end{minipage} 
\qquad
\begin{minipage}[t]{0.45\textwidth}
\centering
\includegraphics[width=\textwidth]{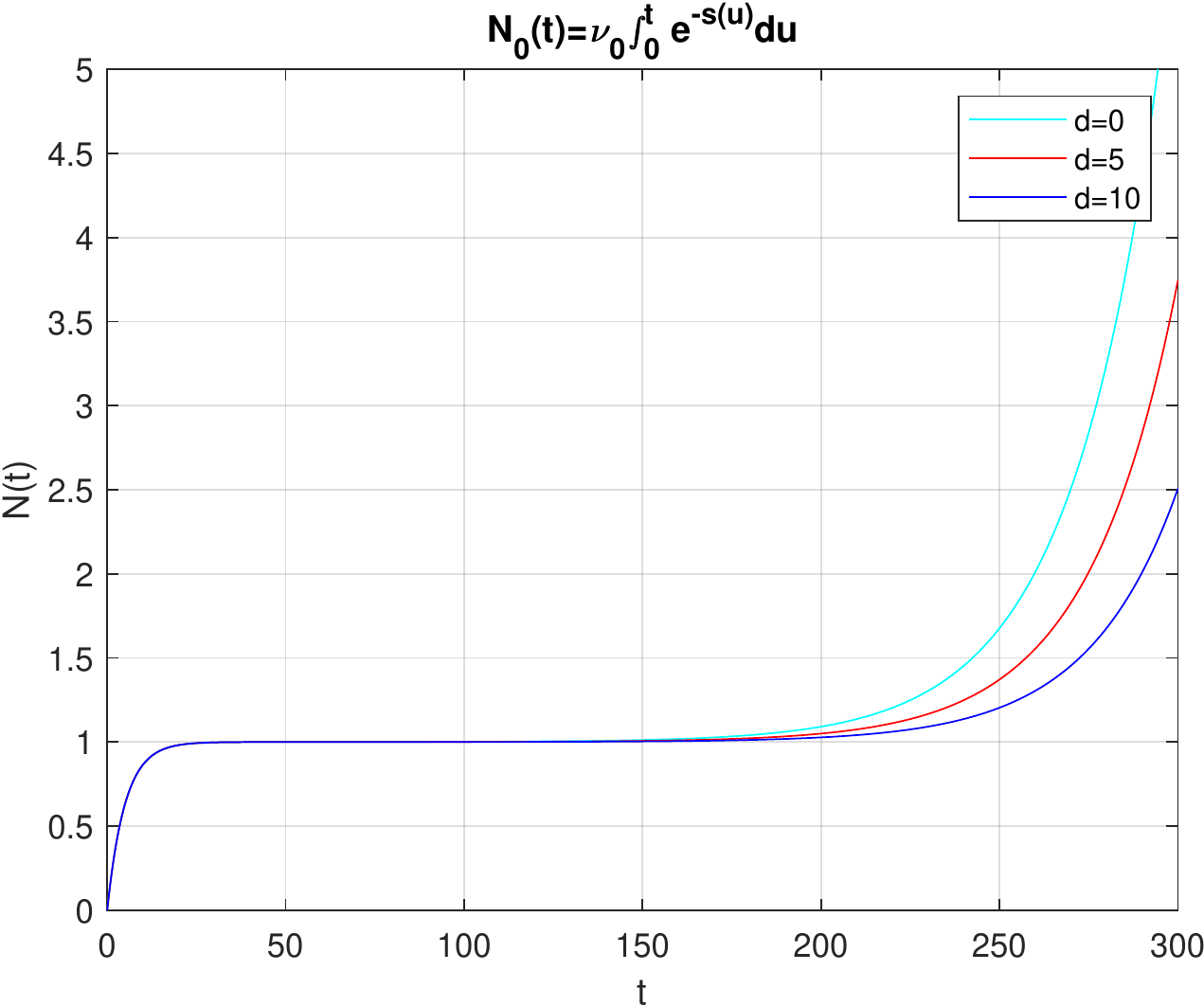}
\caption{\sf The function $N(t)$ when $\nu(t)=\nu_0$.} 
\label{fig:N(t)-BDI-nu_0}
\end{minipage}
\end{figure}
\begin{figure}[bht]
\begin{minipage}[b]{0.45\textwidth}
\centering
\includegraphics[width=\textwidth]{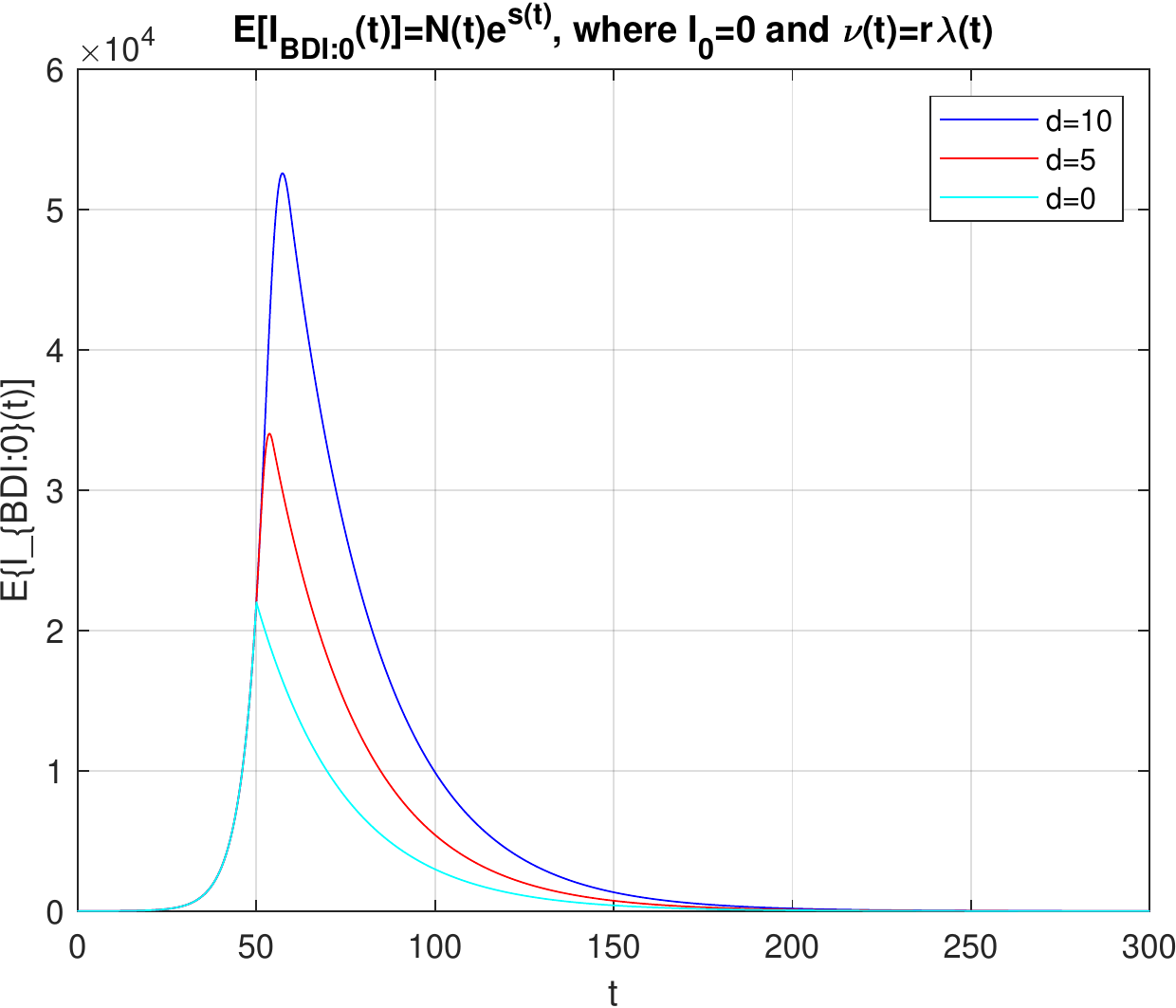}
\caption{\sf  $\oI_{BDI:0}(t)$ when $\nu(t)=r\lambda(t)$.}
\label{fig:EI_BDI-with-N(t)}
\end{minipage} 
\qquad
\begin{minipage}[b]{0.45\textwidth}
\centering
\includegraphics[width=\textwidth]{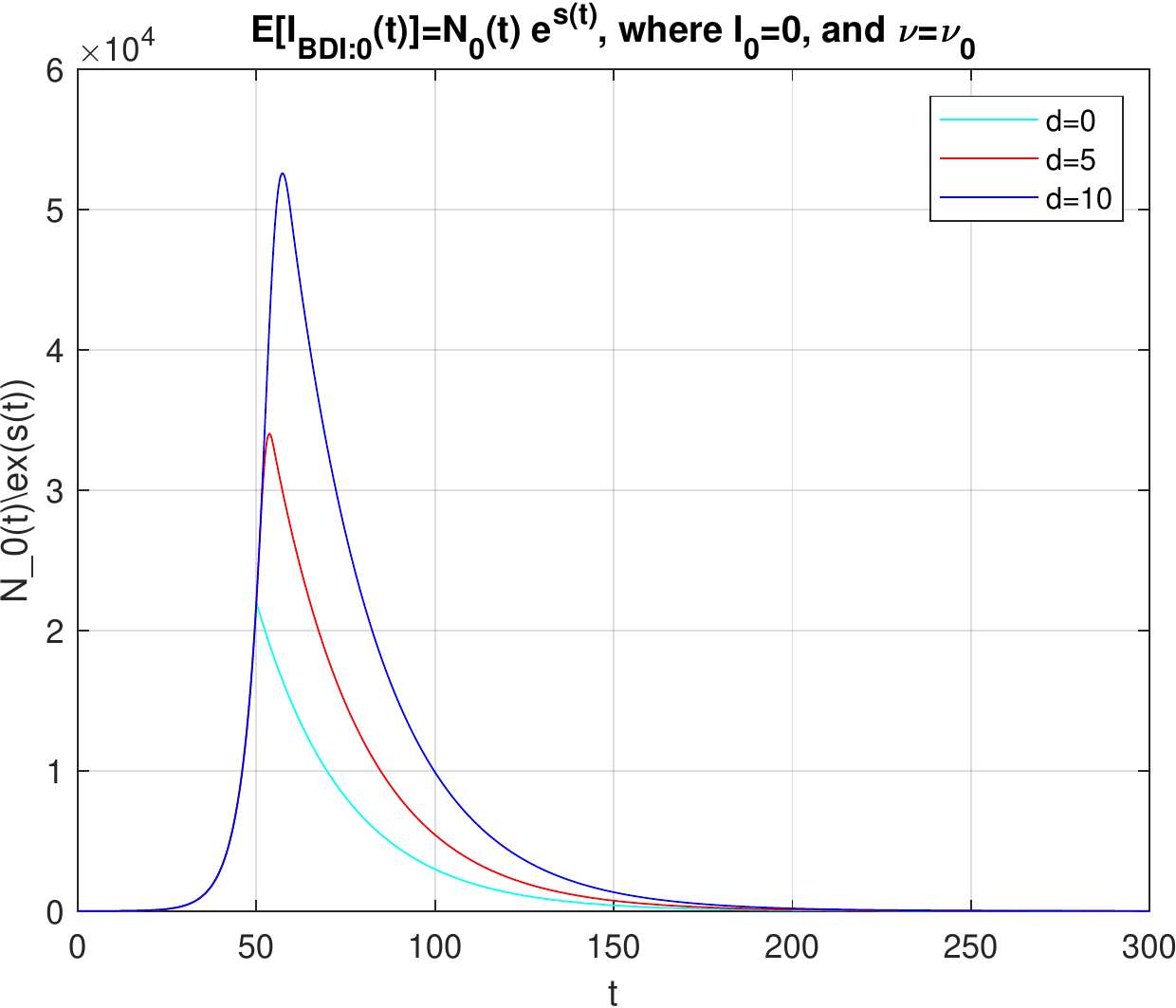}
\caption{\sf $\oI_{BDI:0}(t)$  when $\nu(t)=\nu_0$.} 
\label{fig:EI_BDI-nu_0}
\end{minipage}
\end{figure}
\begin{figure}[hbt]
\begin{minipage}[b]{0.45\textwidth}
\centering
\includegraphics[width=\textwidth]{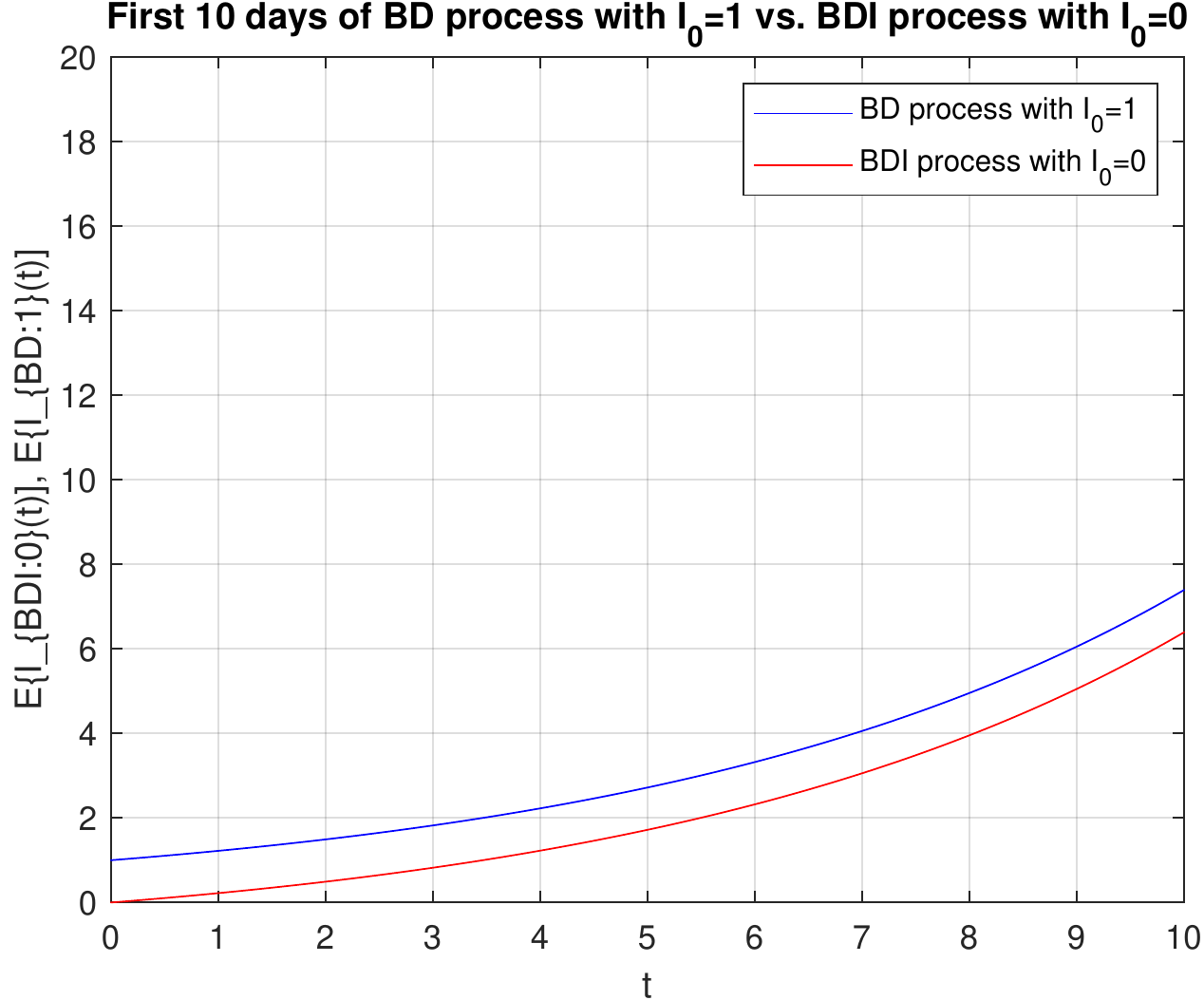}
\caption{\sf  $\oI_{BDI:0}(t)$ and $\oI_{BD:1}(t)$.}
\label{fig:BD-I_0=1-and-BDI-I_0=0}
\end{minipage}
\qquad
\begin{minipage}[b]{0.45\textwidth}
\centering
\includegraphics[width=\textwidth]{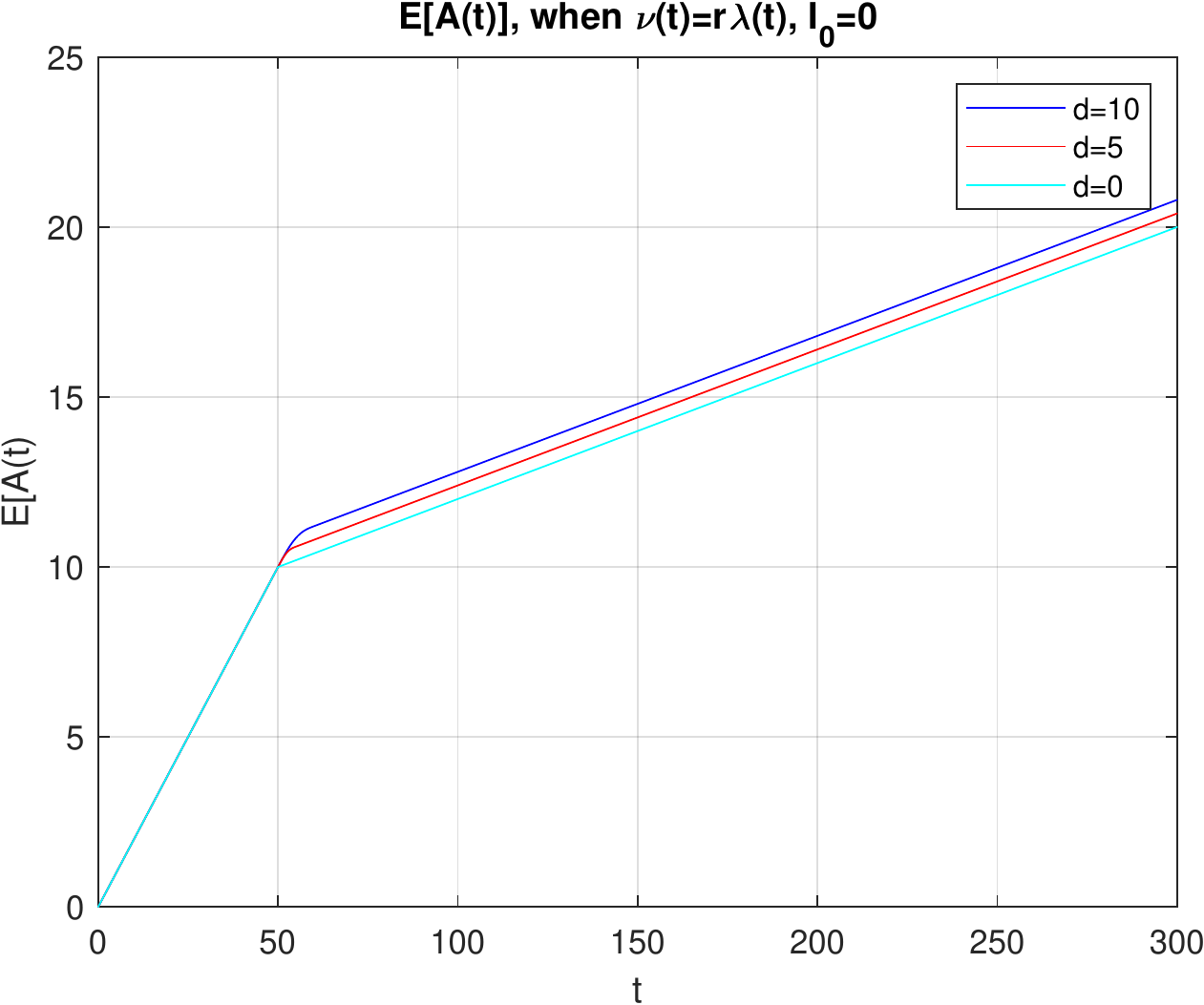}
\caption{\sf $\oA_{BDI:0}(t)$  when $\nu(t)=r\lambda(t)$.} 
\label{fig:EA-BDI-I_0=0}
\end{minipage}
\end{figure}
\begin{figure}[hbt]
\begin{minipage}[b]{0.45\textwidth}
\centering
\includegraphics[width=\textwidth]{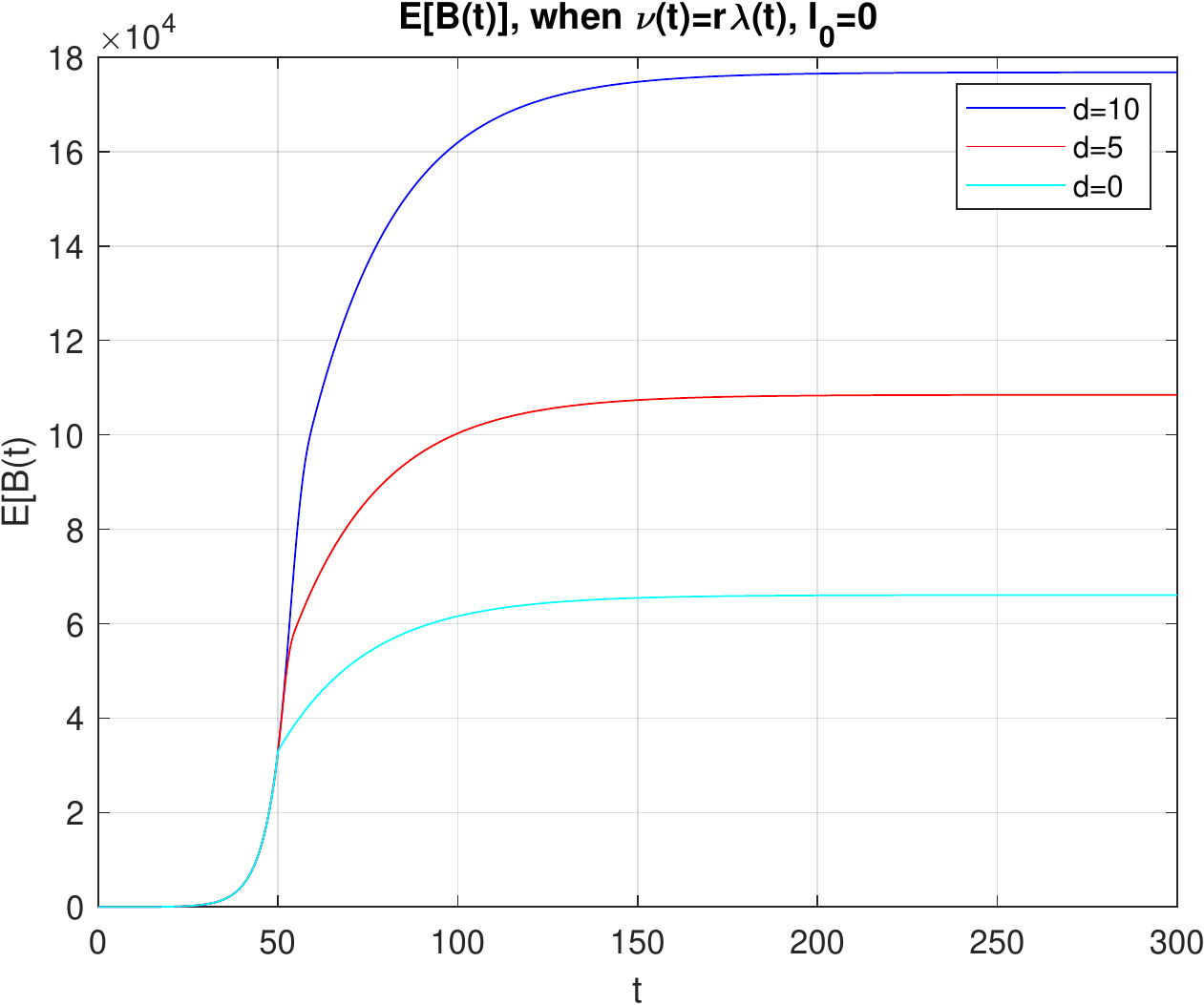}
\caption{\sf $\oB_{BDI:0}(t)$  when $\nu(t)=r\lambda(t)$.}
\label{fig:EB-BDI-I_0=0}
\end{minipage}
\qquad
\begin{minipage}[b]{0.45\textwidth}
\centering
\includegraphics[width=\textwidth]{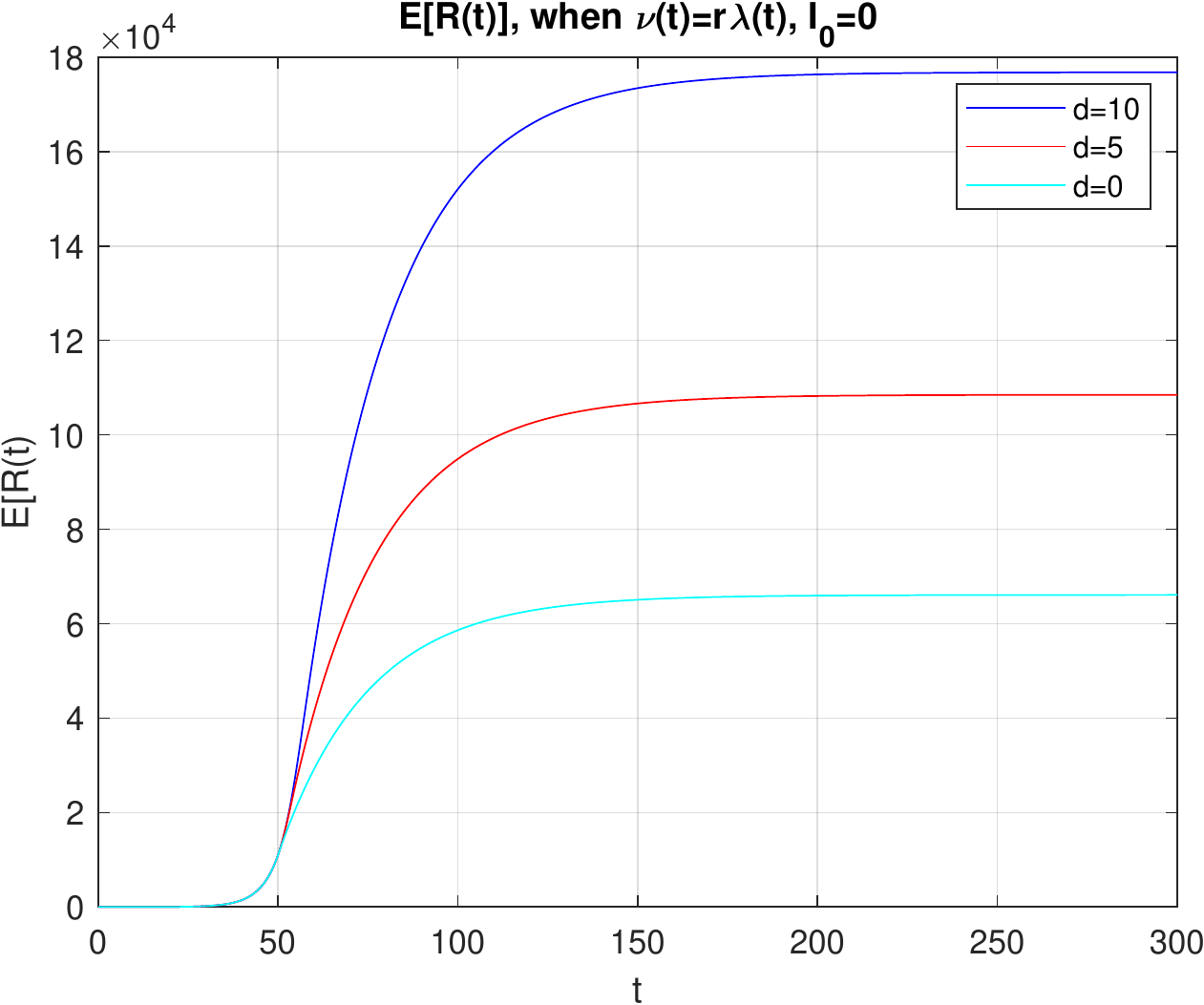}
\caption{\sf $\oR_{BDI:0}(t)$  when $\nu(t)=r\lambda(t)$.} 
\label{fig:ER-BDI-I_0=0}
\end{minipage}
\end{figure}
\begin{figure}[hbt]
\begin{minipage}[b]{0.45\textwidth}
\centering
\includegraphics[width=\textwidth]{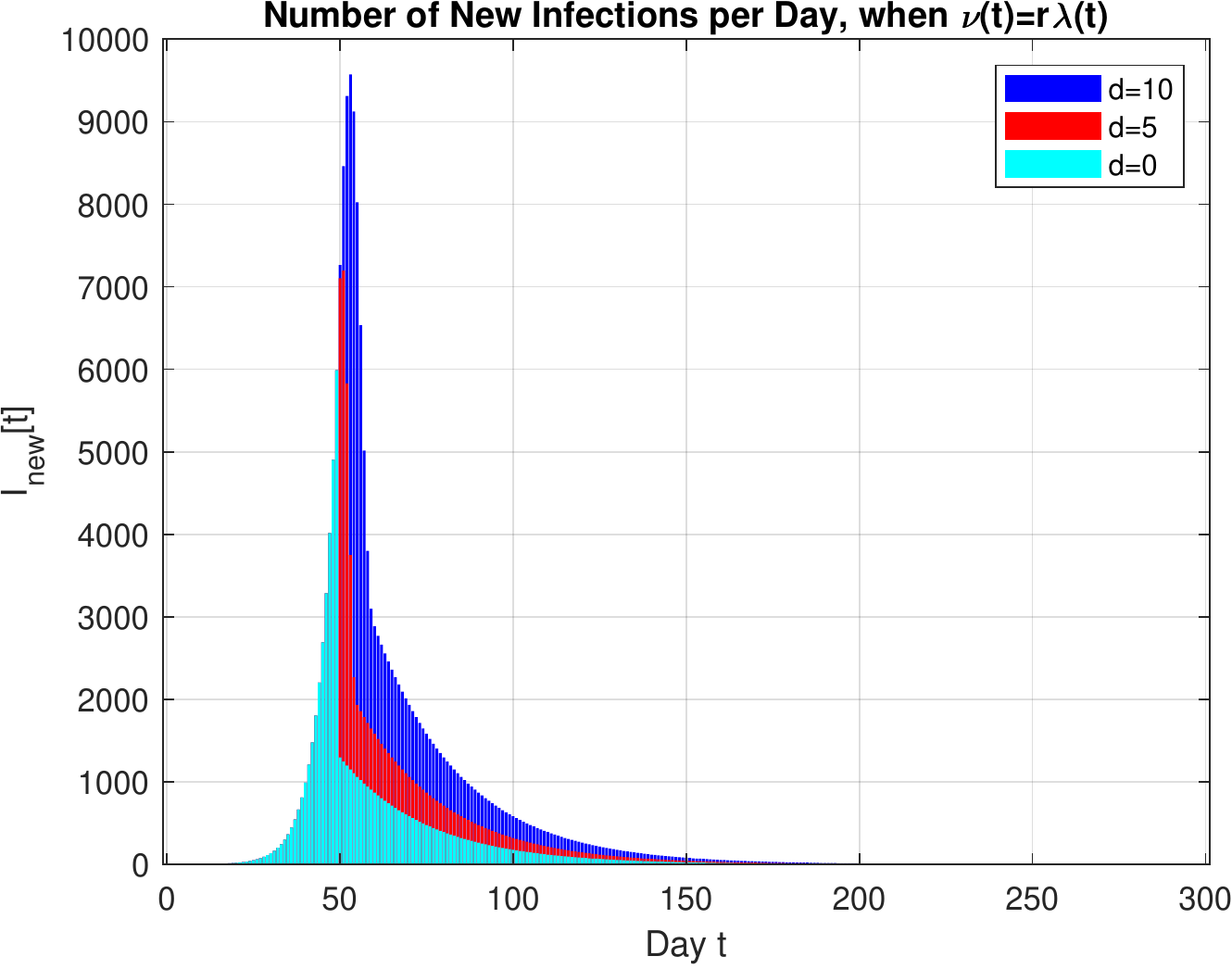}
\caption{\sf New Daily Infections $\oI_{new}[t], t-0, 1, 2, \ldots$  when $\nu(t)=r\lambda(t)$.}
\label{fig:New-daily-infections}
\end{minipage}
\qquad
\begin{minipage}[b]{0.45\textwidth}
\centering
\includegraphics[width=\textwidth]{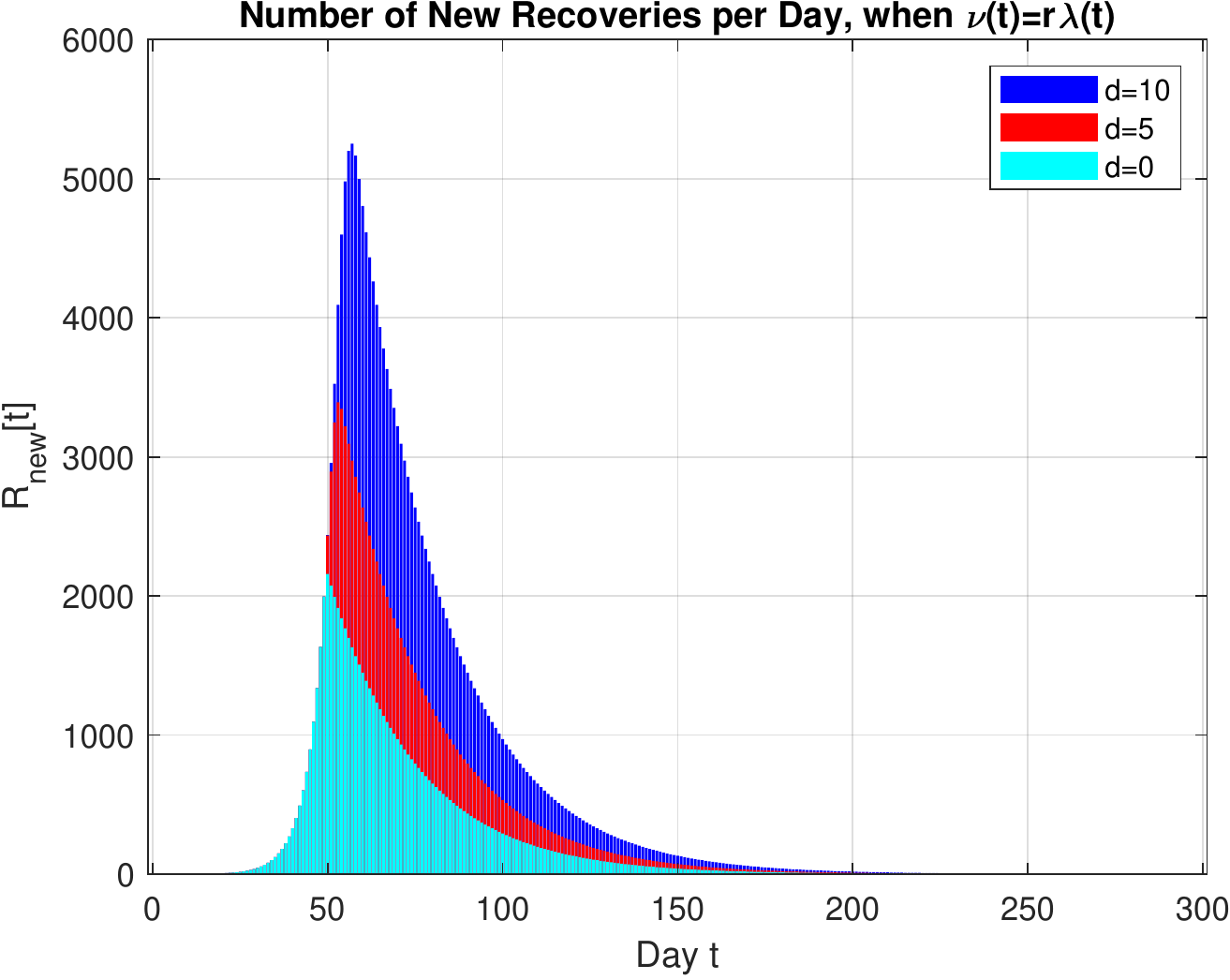}
\caption{\sf New Daily Recoveries $\oR_{new}[t], t-0, 1, 2, \ldots$  when $\nu(t)=r\lambda(t)$.} 
\label{fig:New-daily-recoveries}
\end{minipage}
\end{figure}
\begin{figure}[hbt]
\begin{minipage}[b]{0.45\textwidth}
\centering
\includegraphics[width=\textwidth]{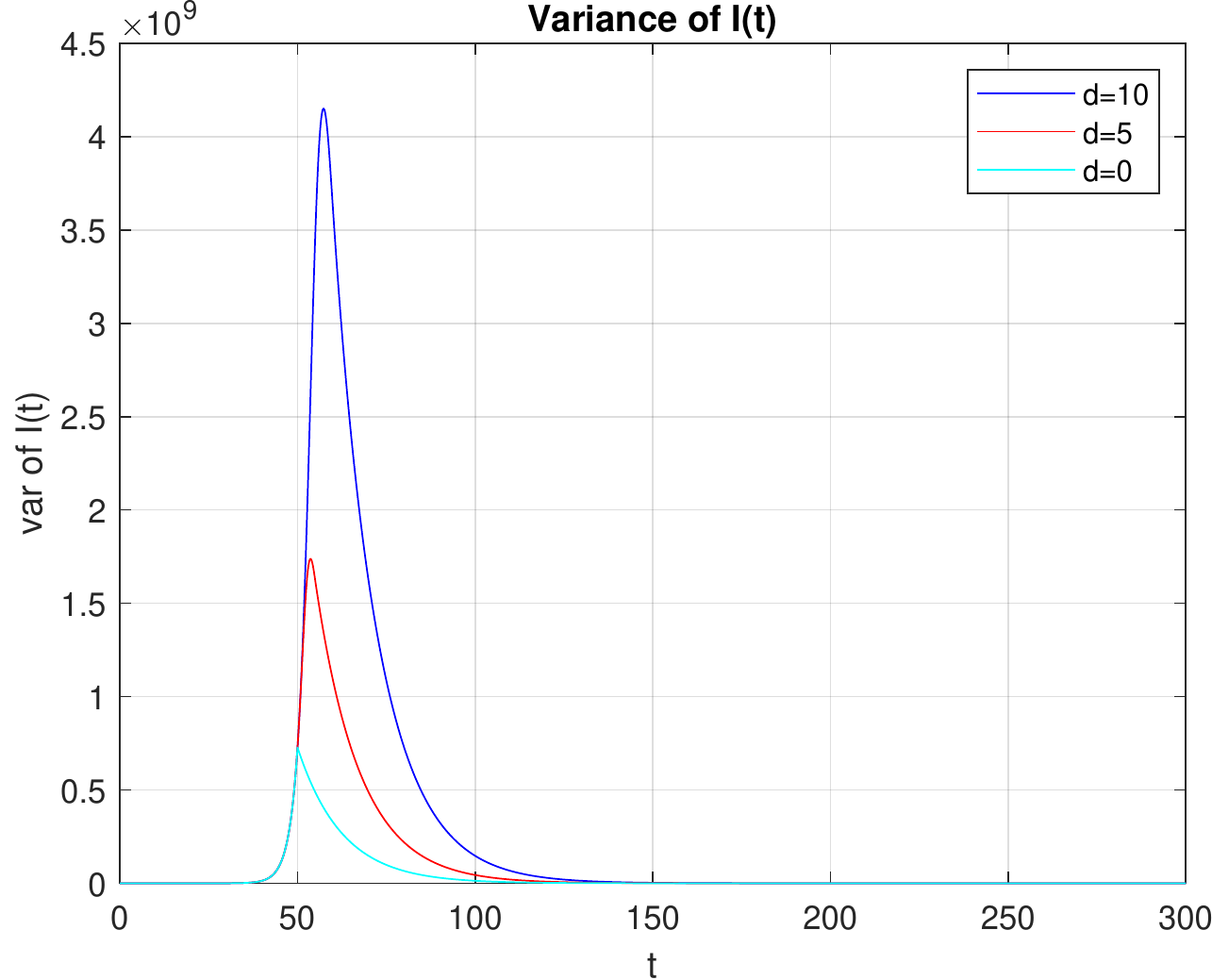}
\caption{\sf Variance $\sigma^2_{BDI:0}(t)$, when $\nu(t)=r\lambda(t)$.}
\label{fig:variance-BDI-I_0=0}
\end{minipage}
\qquad
\begin{minipage}[b]{0.45\textwidth}
\centering
\includegraphics[width=\textwidth]{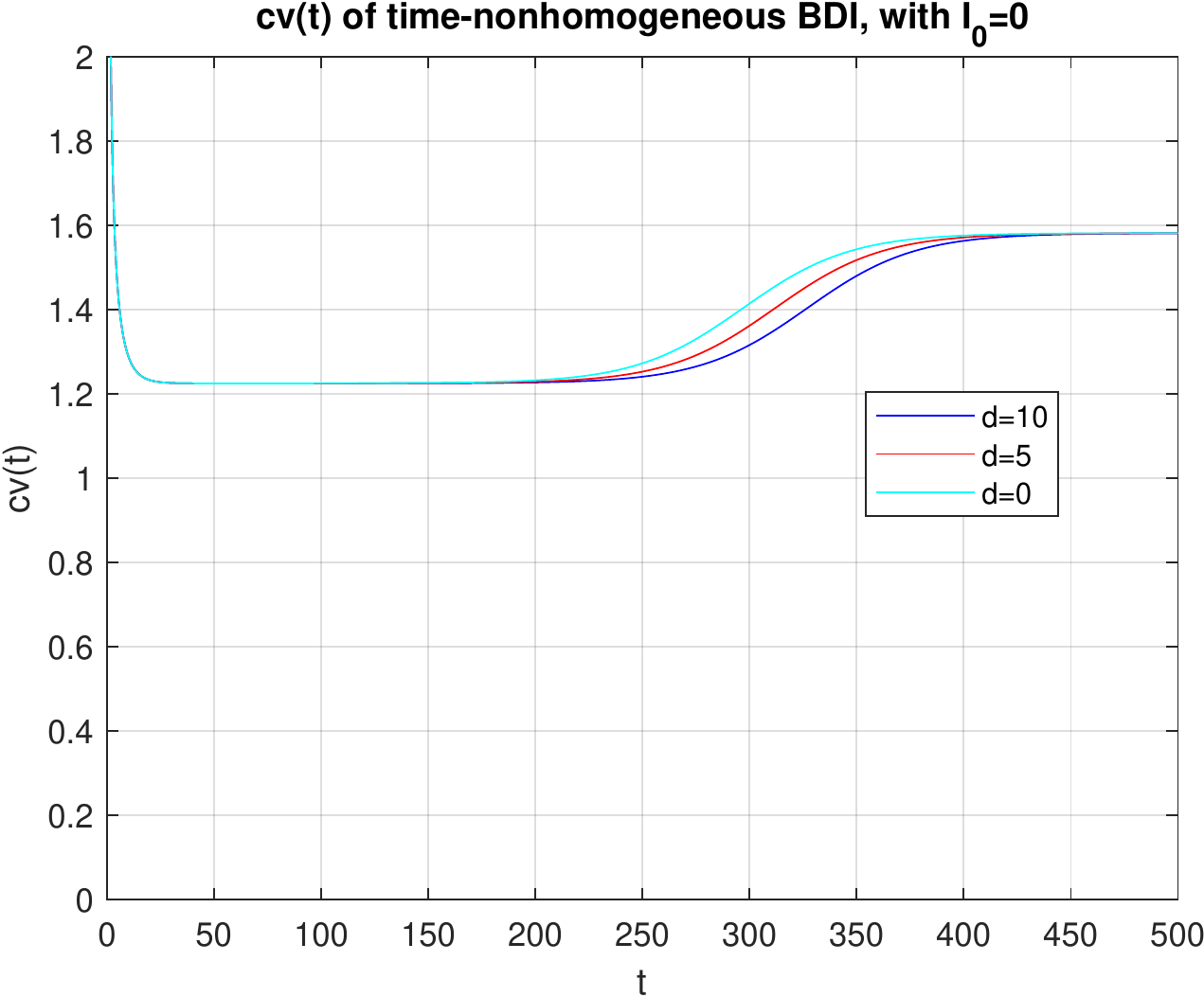}
\caption{\sf Coefficient of variation $cv_{BDI:0}(t)$,  when $\nu(t)=r\lambda(t)$.} 
\label{fig:cv-BDI-I_0=0}
\end{minipage}
\end{figure}
Figure \ref{fig:N(t)-BDI-r-lambda} shows the function $N(t)$ when $\nu(t)=r\lambda(t)$,  whereas Figure \ref{fig:N(t)-BDI-nu_0} shows $N(t)$ when $\nu(t)=\nu_0$, i.e., $N(t)=\nu_0\Sigma(t)$, where $\Sigma(t)$ was defined in (18) of \cite{kobayashi:2021a} and plotted in Figure 5, ibid. Note that these two curves are hardly distinguishable up to time $t\approx 175$. The resultant curves of $\oI_{BDI:0}(t)$, shown in Figures \ref{fig:EI_BDI-with-N(t)} and  \ref{fig:EI_BDI-nu_0}, respectively, are hardly indistinguishable for the entire $t$, beyond  $t\approx 175$.  This insensitivity of $I_{BDI:0}(t)$ to $\nu(t)$ was explained in the last paragraph of page 10 of \cite{kobayashi:2021a}. These curves are also hardly distinguishable from that of $\oI_{BD:1}(t)$, i.e., the BD process with the initial value $I_0=1$, shown in Figure 4 of \cite{kobayashi:2021a}:
\begin{align}
\oI_{BD:1}(t)=e^{s(t)}.  \label{eq:BD-with-I_0=1}
\end{align}
Figure \ref{fig:BD-I_0=1-and-BDI-I_0=0} given below shows the first ten days of the curves (\ref{eq:BD-with-I_0=1}) and (\ref{stable-mean-BDI:0}).  By noting that the function $s(t)=(\lambda_0-\mu_0)t\triangleq a_0 t$, and $\nu(t)=\nu_0$ for $0\leq t\le t_1=50$ in our example, we can compute
\begin{align}
N(t)=\nu_0\int_0^t e^{-a_0 u}\,du=\frac{\nu_0}{a_0}\left(1-e^{-a_0t}\right).
\end{align}
In this example, $\frac{\nu_0}{a_0}=1.0$. Thus, the difference between the two curves  in Figure  $t\approx 175$ is a shift by unity at $t=0$: the BD process starts from 1, whereas the BDI process starts from 0.  

Figures \ref{fig:EA-BDI-I_0=0}, \ref{fig:EB-BDI-I_0=0} and \ref{fig:ER-BDI-I_0=0} are the expected cumulative counts of the external arrivals $\oA_{BDI:0}(t)$, internal infections $\oB_{BDI:0}(t)$, and recovery/removal/death $\oB_{BDI:0}(t)$, respectively.  The last two curves are hardly distinguishable from Figures 9 and 10 of \cite{kobayashi:2021a}, as expected.

Figures \ref{fig:New-daily-infections} and \ref{fig:New-daily-recoveries} are bar plots of the new daily infections and recoveries defined in (41) and (40) of \cite{kobayashi:2021a}.  Again they are virtually the same as Figures 11 and 12, ibid. Although the terms $\oA(t)-\oA(t-1)$ in (41) is non-zero for the BDI process, this contribution is negligibly small compared with the terms $\oB(t)-\oB(t-1)$.

The variance of $I_{BDI:0}(t)$ can be expressed, using the formula of the negative binomial distribution (see e.g., (53) of \cite{kobayashi:2020a-arXiv}) as
\begin{align}
\sigma^2_{BDI:0}(t)=\frac{r\beta(t)}{(1-\beta(t))^2},
\end{align}
which is computationally unstable when $\beta(t)\approx 1$.  A better computational formula is 
\begin{align}
\sigma^2_{BDI:0}(t)=rL(t)(1+M(t))e^{2s(t)}, \label{variance-BDI}
\end{align}
from which we can compute the coefficient of variation as
\begin{align}
c_{BDI:0}(t)=\frac{\sqrt{rL(t)(1+M(t))}}{N(t)}. \label{cv-BDI}
\end{align}
In Figures \ref{fig:variance-BDI-I_0=0} and \ref{fig:cv-BDI-I_0=0} we plot the variance and the c.v. (coefficient of variation) of the process $I_{BDI:0}(t)$.  Figure \ref{fig:cv-BDI-I_0=0} should be compared with the c.v. of the BD process with $I_0=1$, which is $\sqrt{L(t)+M(t)}$ as given in (90), and is plotted in Figure 15 of \cite{kobayashi:2021a}.  

In the case of the BD process which starts with $I_0=1$, its c.v. is zero at $t=0$, whereas in the BDI process with $I_0=0$ both $\sigma_{BDI:0}(t)$ and $\oI_{BDI:0}(t)$ start from zero at $t=0$, but their ratio at $t=0$ is $\infty$.  This is because we can approximate $L(t)$ and $N(t)$ as follows at $t\approx 0$ 
\begin{align}
L(t)&\approx\frac{\lambda_0}{a_0}a_0 t=\lambda_0 t,~~~\mbox{and}~~~
N(t)\approx\nu_0 t,~~~\mbox{at}~~~t\approx 0.
\end{align}
Thus,
\begin{align}
c_{BDI:0}(t)=\frac{\sqrt{r\lambda_0 t}}{\nu_0 t}
=\frac{1}{\sqrt{\nu_0 t}} \to\infty,~~~\mbox{as}~~t\to 0. \label{cv-BDI-near-t=0}
\end{align}

The flat values of the above c.v. are
\begin{align}
c_{BDI:0}(t)&\approx \frac{\sqrt{r\lambda_0(a_0+\mu_0)}}{\nu_0}=\frac{1}{\sqrt{r}};  \\
c_{BD:1}(t)&\approx \sqrt{\frac{(\lambda_0+\mu_0)}{a_0}}.
\end{align}
For $r=2/3$, we find $c_{BDI:0}(t)\approx \sqrt{1.5}=1.2247$ and 
$c_{BD:1}(t)=\sqrt{\frac{0.3+0.1}{0.3-0.1}}=\sqrt{2}=1.414$. 

We should note that the coefficient variation of the BD process with the initial value $I_0$ is, as evident from the PGF (\ref{PGF:BD}).
\begin{align}
c_{BD:I_0}(t)&=\frac{c_{BD:1}(t)}{\sqrt{I_0}}\approx \sqrt{\frac{(\lambda_0+\mu_0)}{a_0I_0}}.
\end{align}

\clearpage

\subsection{Probability Distributions of the BDI Process}

The PMF of the BDI process is given by (\ref{PMF-nonhomo}), i.e.,
\begin{align}
P^{(BDI:0)}_k(t)={k+r-1\choose k}(1-\beta(t))^r\beta(t)^k,~~~k=0,1, 2, \ldots.\label{PMF-BDI}
\end{align}

\begin{figure}[hbt]
\begin{minipage}[b]{0.45\textwidth}
\centering
\includegraphics[width=\textwidth]{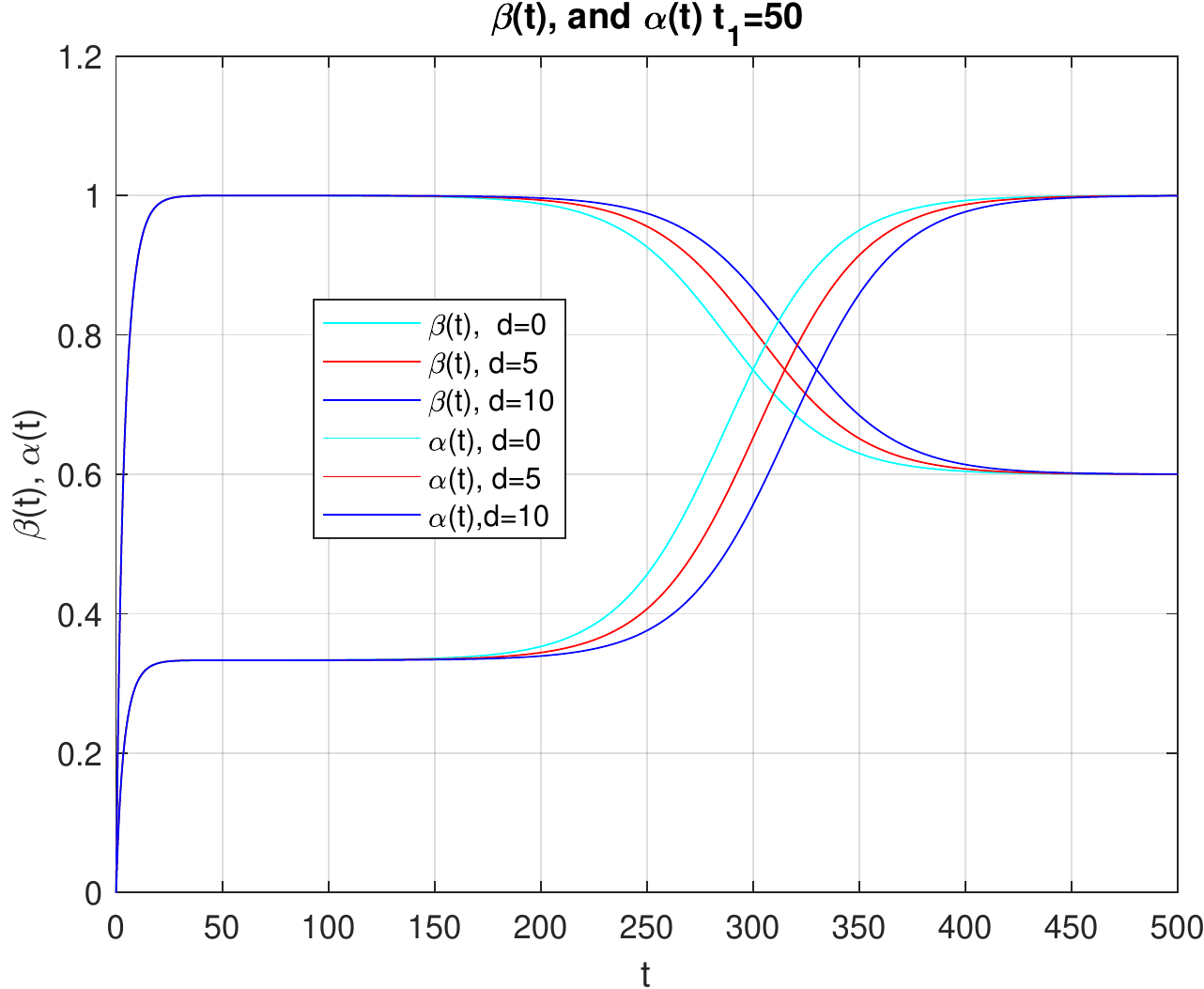}
\caption{\sf $\alpha(t)$ and $\beta(t)$.}
\label{fig:alpha(t)-beta(t)}
\end{minipage}
\qquad
\begin{minipage}[b]{0.45\textwidth}
\centering
\includegraphics[width=\textwidth]{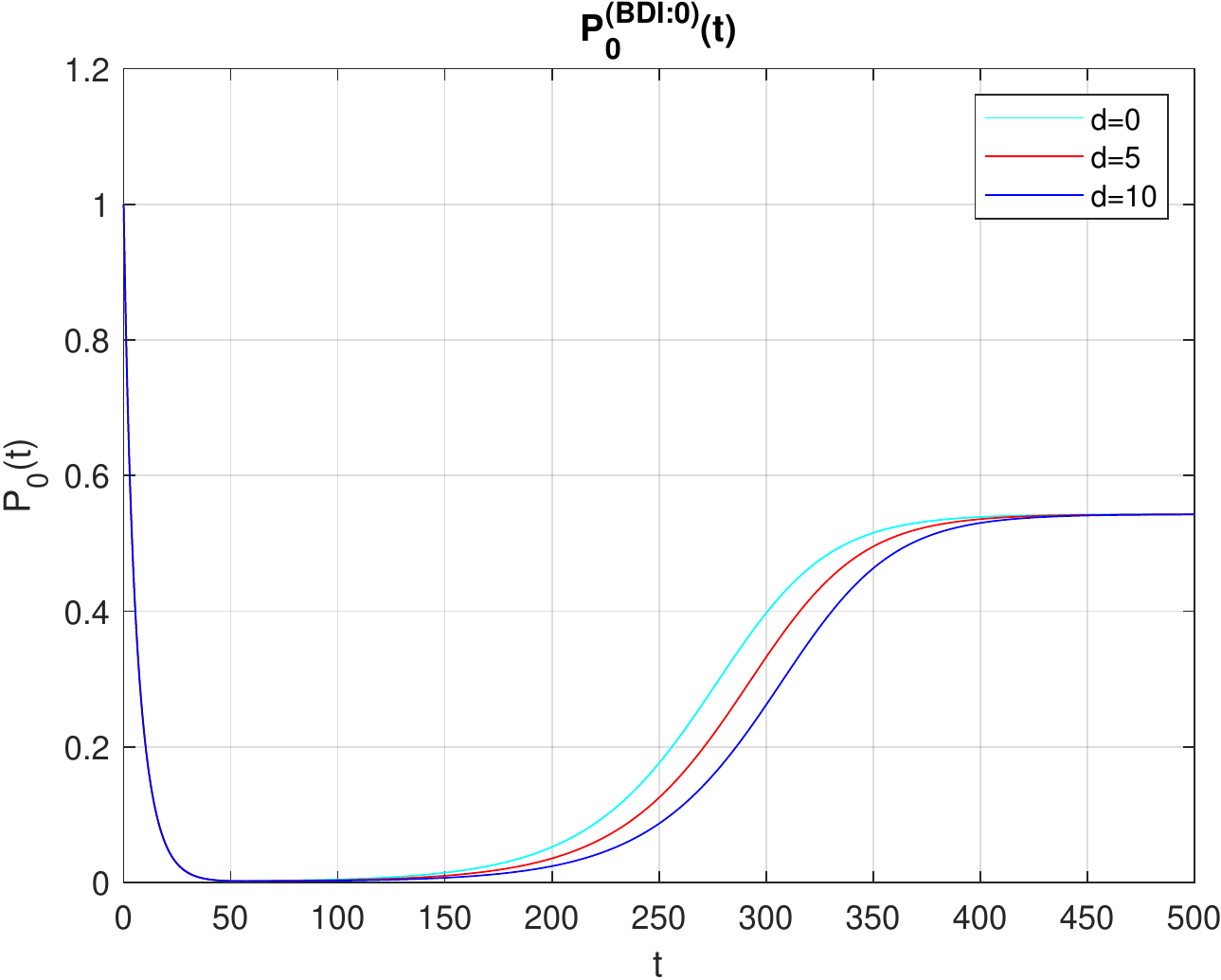}
\caption{\sf $P_0(t)$ for the BDI process with $I_0=0$.} 
\label{fig:P_0(t)-BDI-I_0}
\end{minipage}
\end{figure}

For the convenience of the readers we reproduce in Figure \ref{fig:alpha(t)-beta(t)} of the plot of $\alpha(t)$ (the one that is a monotone increasing) and $\beta(t)$ which were given in  Figure 14 of \cite{kobayashi:2021a}.  The functions $\alpha(t)$ and $\beta(t)$ are intimately related to each other. From (81) ibid, for instance, we know
\begin{align}
\frac{1-\alpha(t)}{1-\beta(t)}=e^{s(t)},  \label{alpha-beta-relation}
\end{align}
Thus, we have two different expressions for $P_0^{(BDI:0)}(t)$:
\begin{align}
P_0^{(BDI:0)}(t)=(1-\beta(t))^r=(1-\alpha(t))^re^{-rs(t)}.\label{P_0(t)-BDI}
\end{align}

By comparing the PMFs of the BDI process at various $t$ with those of the BD process given in Figure 16-27 of \cite{kobayashi:2021a}\footnote{We only considered the case $d=0$ in the PMF plots of the BD process.}, we will notice the following differences, although the mean values of the two processes, $\oI_{BD:1}(t)$ and $\oI_{BDI:0}(t)$, are hardly distinguishable, as noted earlier:
\begin{enumerate}
\item The BD process process begins at $I_0(\geq 1)$ (in the example we computed we assumed $I_0=1$), whereas the BDI starts from $I_0=0$.
\item The maximum probability the BD process remains at initial value $k=1 (=I_0)$ until $t\approx 4$, but the peak moves to $k=0$ by $t=8$, and its value remains around at $\frac{\mu_0}{\lambda_0}=1/3$ until $t\approx 150$.  Between $t=150$ to $350$, the probability at $k=0$ gradually increases towards unity. As $t$ further increases, it will converges to one, that is the process $I(t)$ converges to zero with probability one. Refer to Appendix A for a more detailed discussion.
\item The PMFs of the BDI process at $0\leq t\leq t_1(=50)$, i.e., Figures 15-20 are the same as those we showed in Part I \cite{kobayashi:2020a-arXiv}, Figures 9-13. As discussed in Section 5.1 and shown in Figure 6 of Part I \cite{kobayashi:2020a-arXiv}, the negative binomial distribution, NB($r,\beta$) with $r<1$, is monotone decreasing distribution, and has a fat tail when $\beta$ is close to unity, because $P_k(t)\propto \beta(t)^k$.

\item The BDI with $I_0=0$ begins, by definition has its peak value of probability distribution at $k=0$.  As (\ref{P_0(t)-BDI}) suggests, $P_0(t)$ quickly comes down to zero and stays there until $t\approx 150$ and gradually increases towards the equilibrium in the $a=\lambda_0-\mu_0=0.06-0.2<0$ regime.  As we see from (84) of \cite{kobayashi:2021a}
\begin{align}
\lim_{t\to\infty}\beta(t)=\frac{\lambda_1}{\mu_0}=0.06/0.1=0.6.
\end{align} 
The $P_0(t)\to (1-\beta(t))^r\to (1-0.6)^{2/3}=0.5429$.  In general 
\begin{align}
\lim_{t\to\infty} P_k(t)={k+r-1\choose k}\left(1-{\cal R}_\infty\right)^r {{\cal R}_\infty}^k,
\end{align}
where
\begin{align}
{\cal R}_\infty=\lim_{t\to\infty}\frac{\lambda(t)}{\mu(t)}<1,
\end{align}
which is the effective reproduction number in the equilibrium state.
\item Although the transition of the internal infection rate $\lambda(t)$ and external arrival rate $\nu(t)$ take place at $t_1=50$, it effects to the the transition of the probability distribution towards the equilibrium state takes place with much delay and takes a long interval. As we discussed in Section 4.1 \cite{kobayashi:2021a}, this delay and duration are determined by the function $\Sigma(t)$ defined by (18), ibid, which in turn depends on $s(t)$. The inverse of $\Sigma(t)$ is plotted in Figure 13, ibid. The transition takes place when $s(t)$ decreases toward around 5 ($\exp(-5)\approx 0.0067$) and further decreases towards  

\end{enumerate}

\begin{figure}[thb]
\begin{minipage}[t]{0.30\textwidth}
\centering
\includegraphics[width=\textwidth]{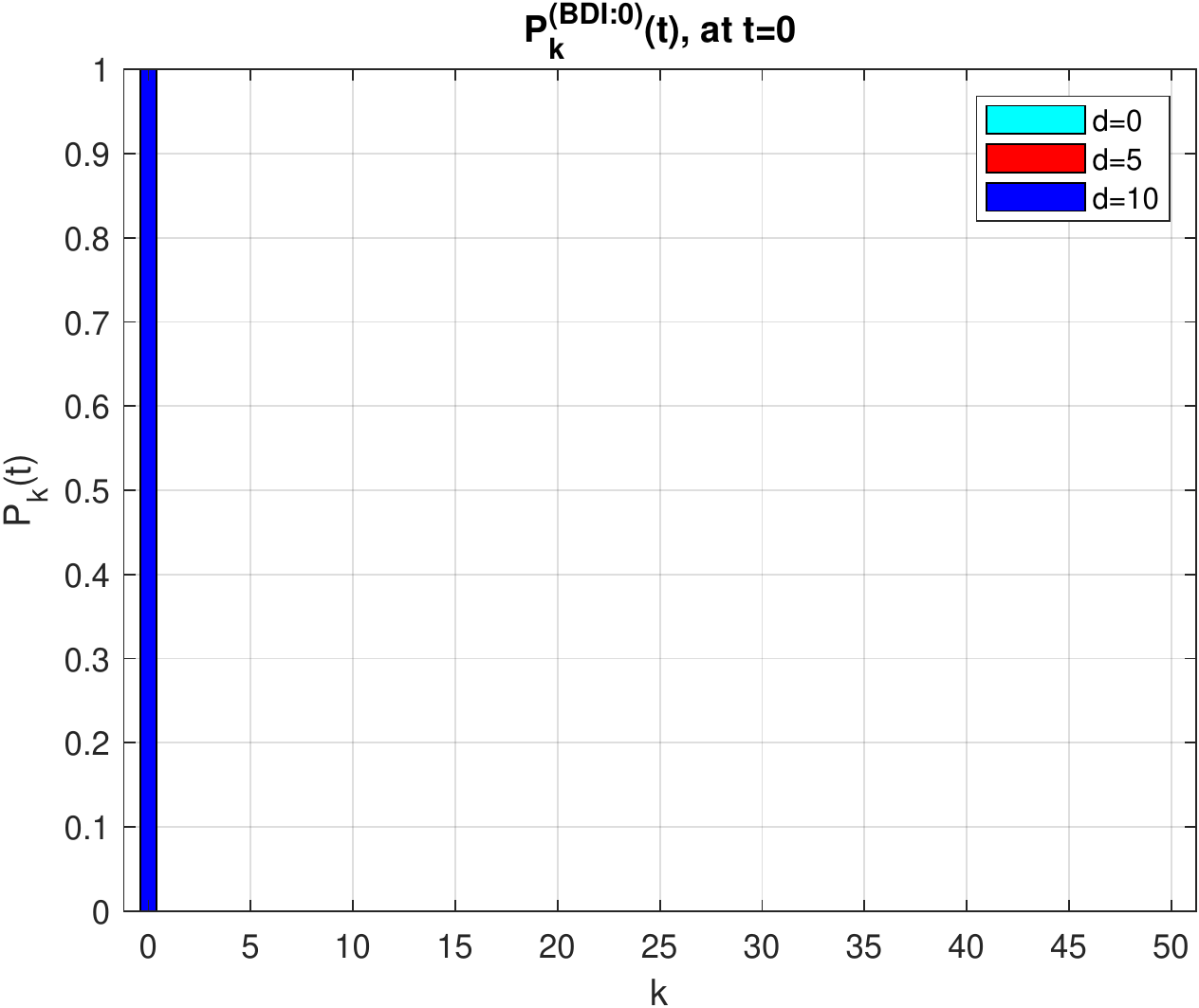}
\caption{\sf $P_k(0)$ of BDI.}
\label{fig:P_k(0)-BDI}
\end{minipage}
\hspace{0.5cm}
\begin{minipage}[t]{0.30\textwidth}
\centering
\includegraphics[width=\textwidth]{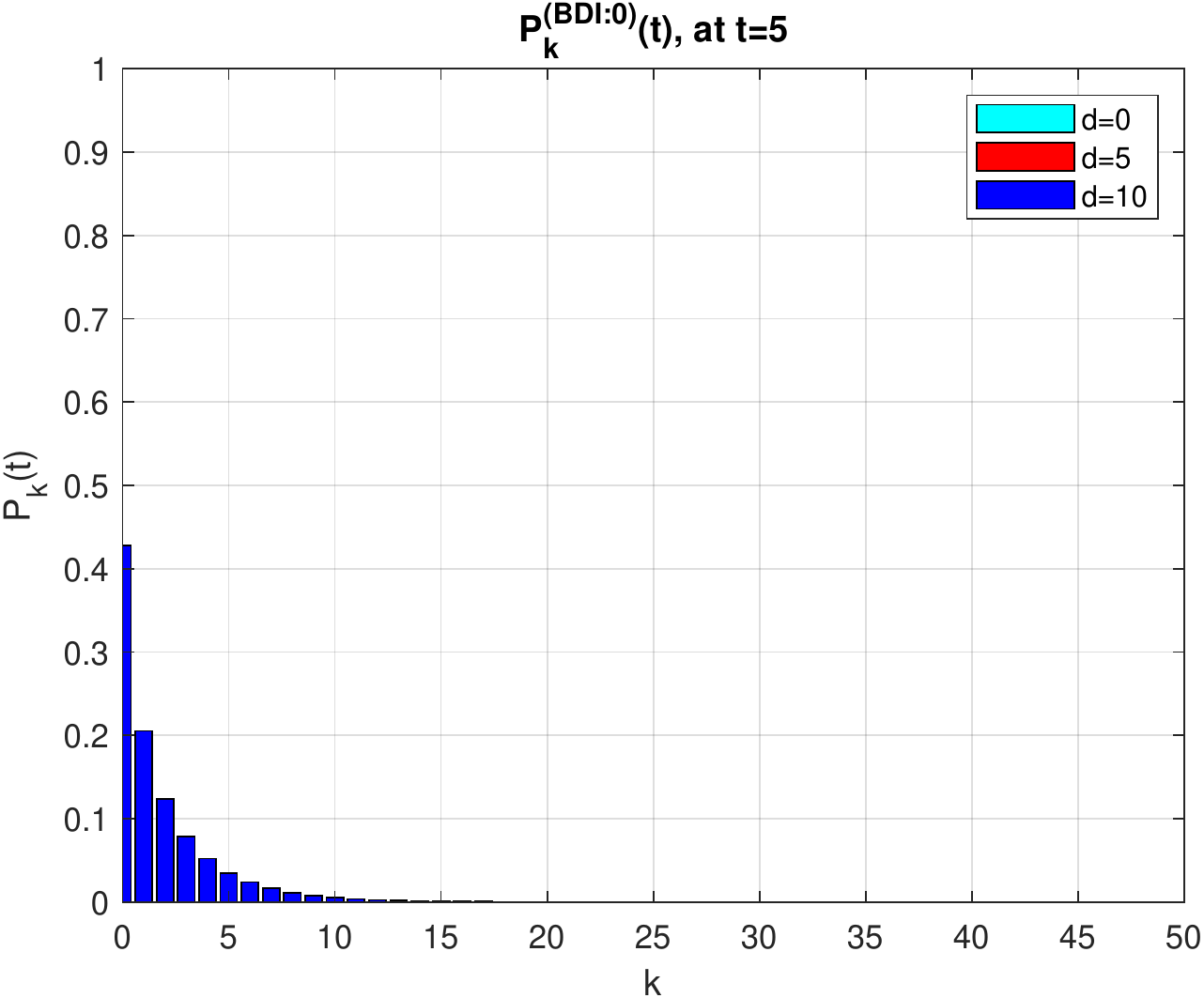}
\caption{\sf $P_k(5)$ of BDI.}
\label{fig:P_k(5)-BDI}
\end{minipage}
\hspace{0.5cm}
\begin{minipage}[t]{0.30\textwidth}
\centering
\includegraphics[width=\textwidth]{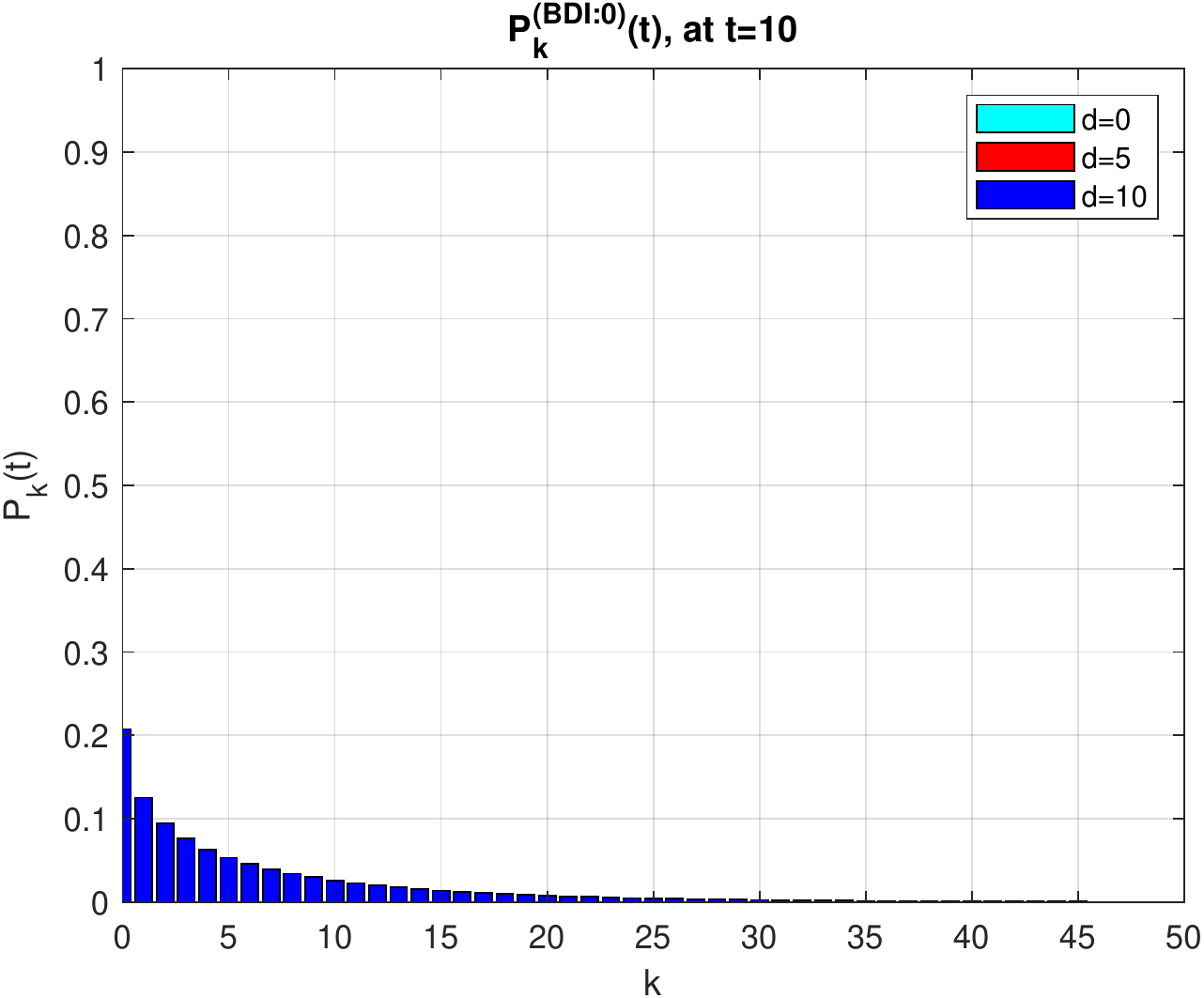}
\caption{\sf $P_k(10)$ of BDI.}
\label{fig:P_k(10)-BDI}
\end{minipage}
\end{figure}
\begin{figure}[hbt]
\begin{minipage}[h]{0.30\textwidth}
\centering
\includegraphics[width=\textwidth]{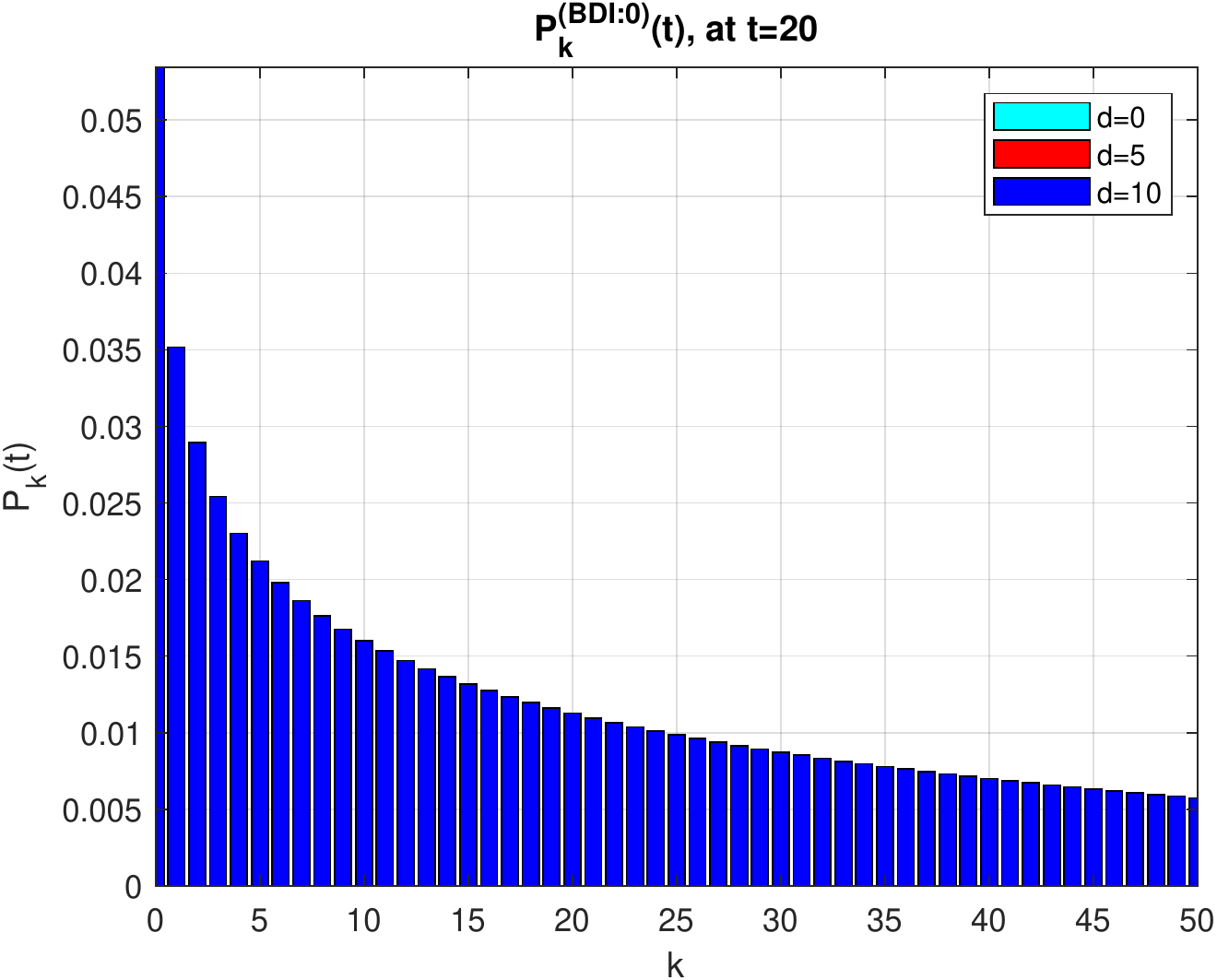}
\caption{\sf  $P_k(20)$ of BDI, scaled up.}
\label{fig:P_k(20)-BDI-up}
\end{minipage}
\hspace{0.5cm}
\begin{minipage}[h]{0.30\textwidth}
\centering
\includegraphics[width=\textwidth]{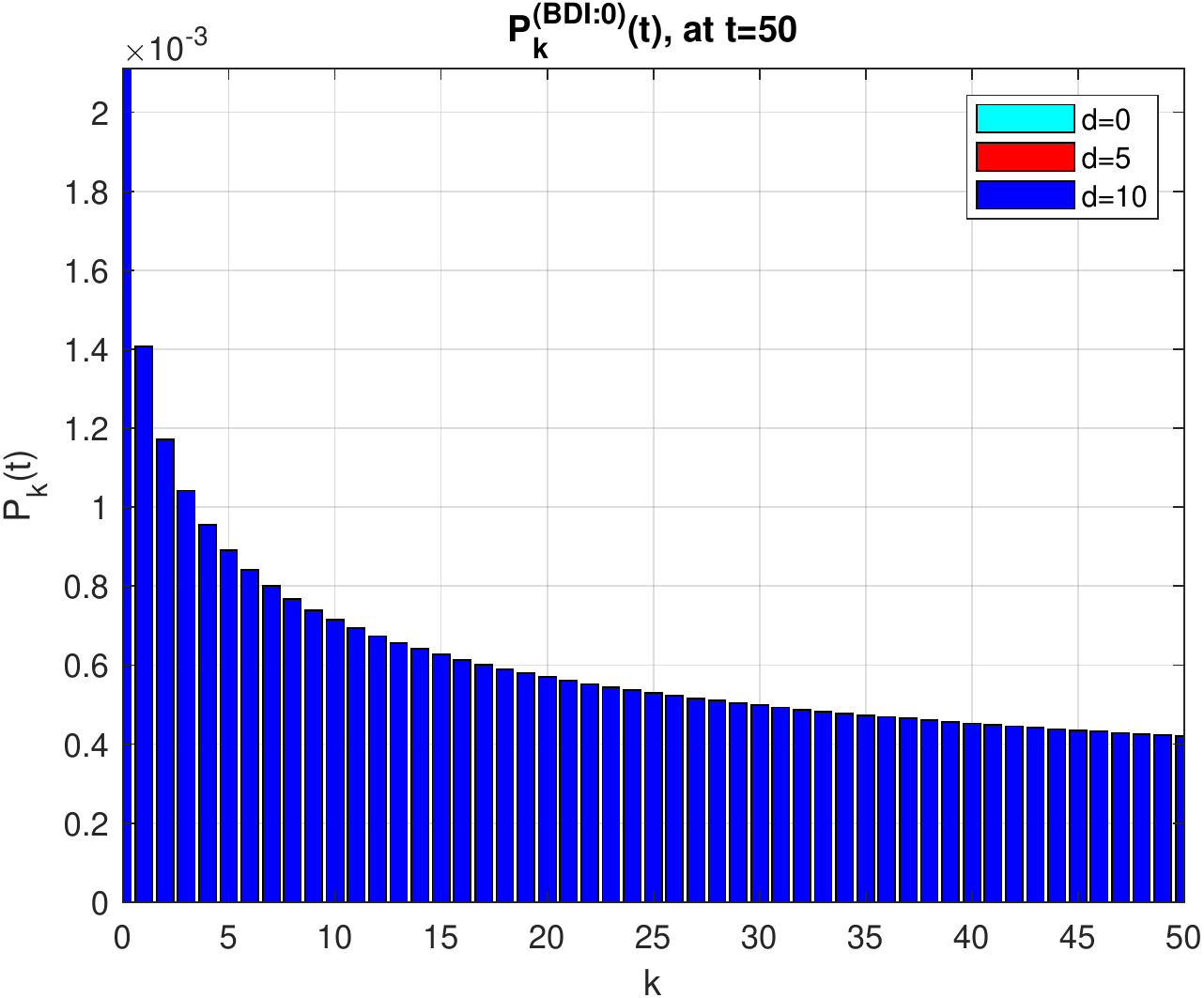}
\caption{\sf $P_k(50)$ of BDI, scaled up.}
\label{fig:P_k(50)-BDI-up}
\end{minipage}
\hspace{0.5cm}
\begin{minipage}[h]{0.30\textwidth}
\centering
\includegraphics[width=\textwidth]{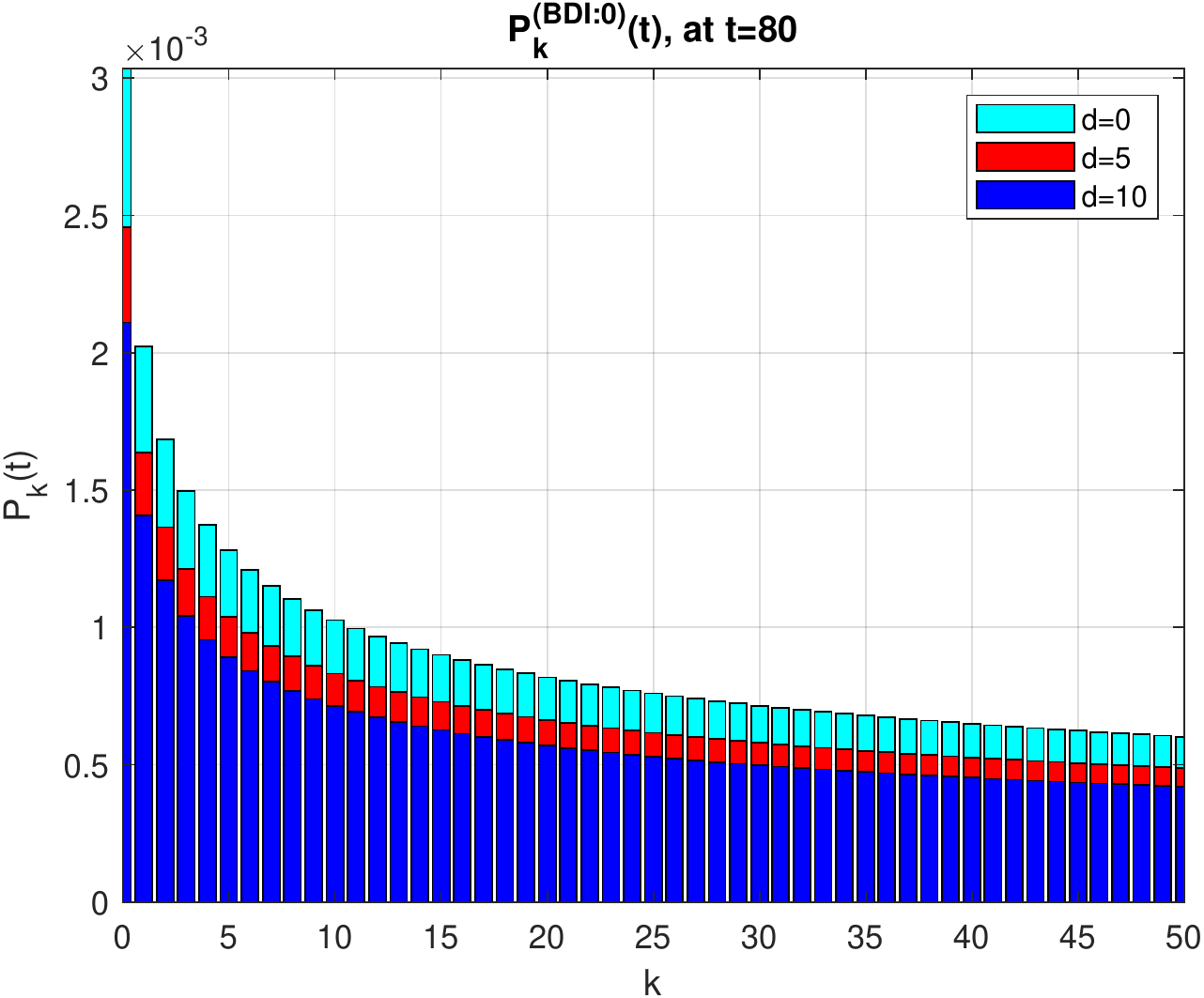}
\caption{\sf $P_k(80)$ of BDI, scaled up.}
\label{fig:P_k(80)-BDI-up}
\end{minipage}
\end{figure}
\begin{figure}[thb]
\begin{minipage}[t]{0.30\textwidth}
\centering
\includegraphics[width=\textwidth]{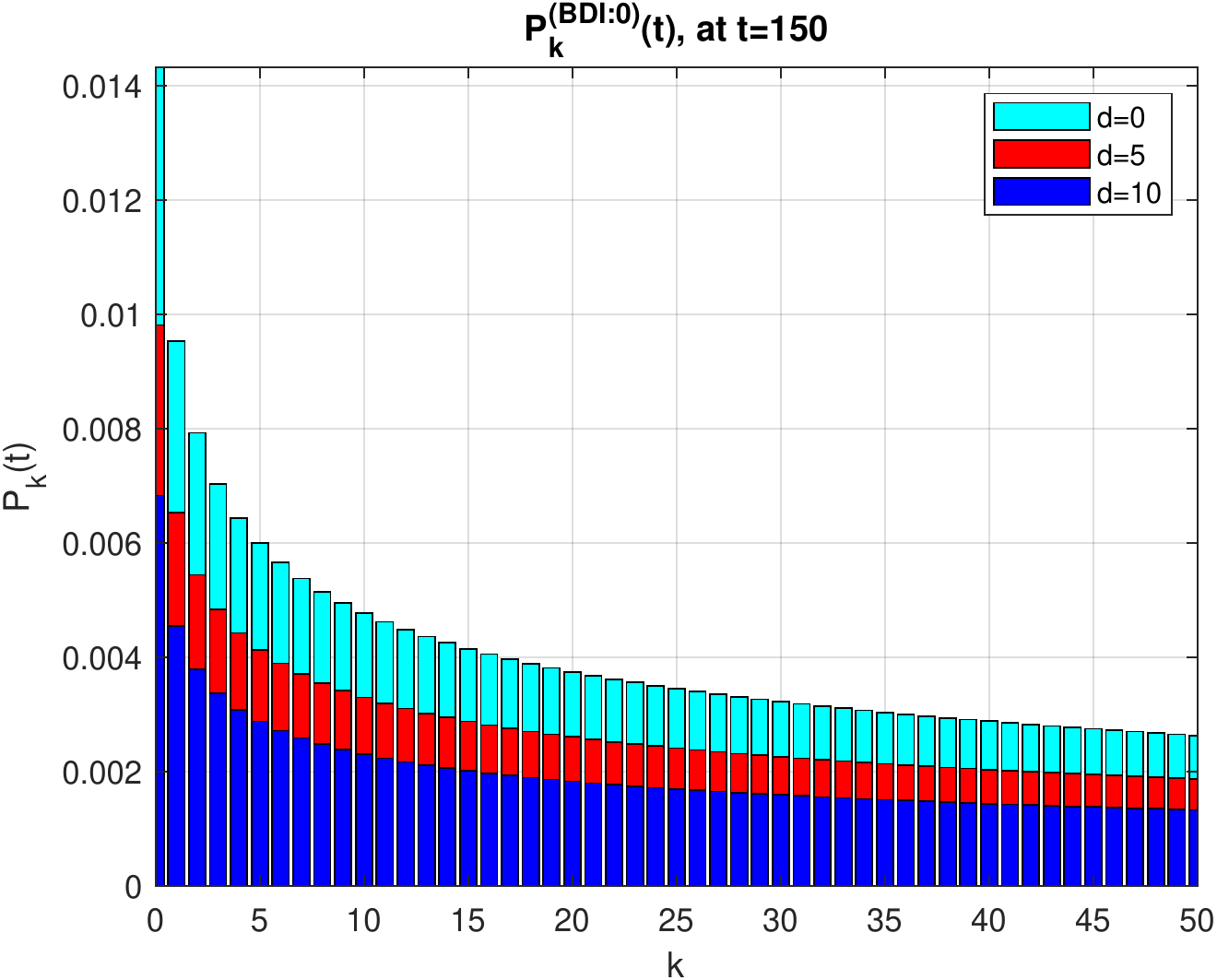}
\caption{\sf $P_k(150)$ of BDI, scaled up.}
\label{fig:P_k(150)-BDI-up}
\end{minipage}
\hspace{0.5cm}
\begin{minipage}[t]{0.30\textwidth}
\centering
\includegraphics[width=\textwidth]{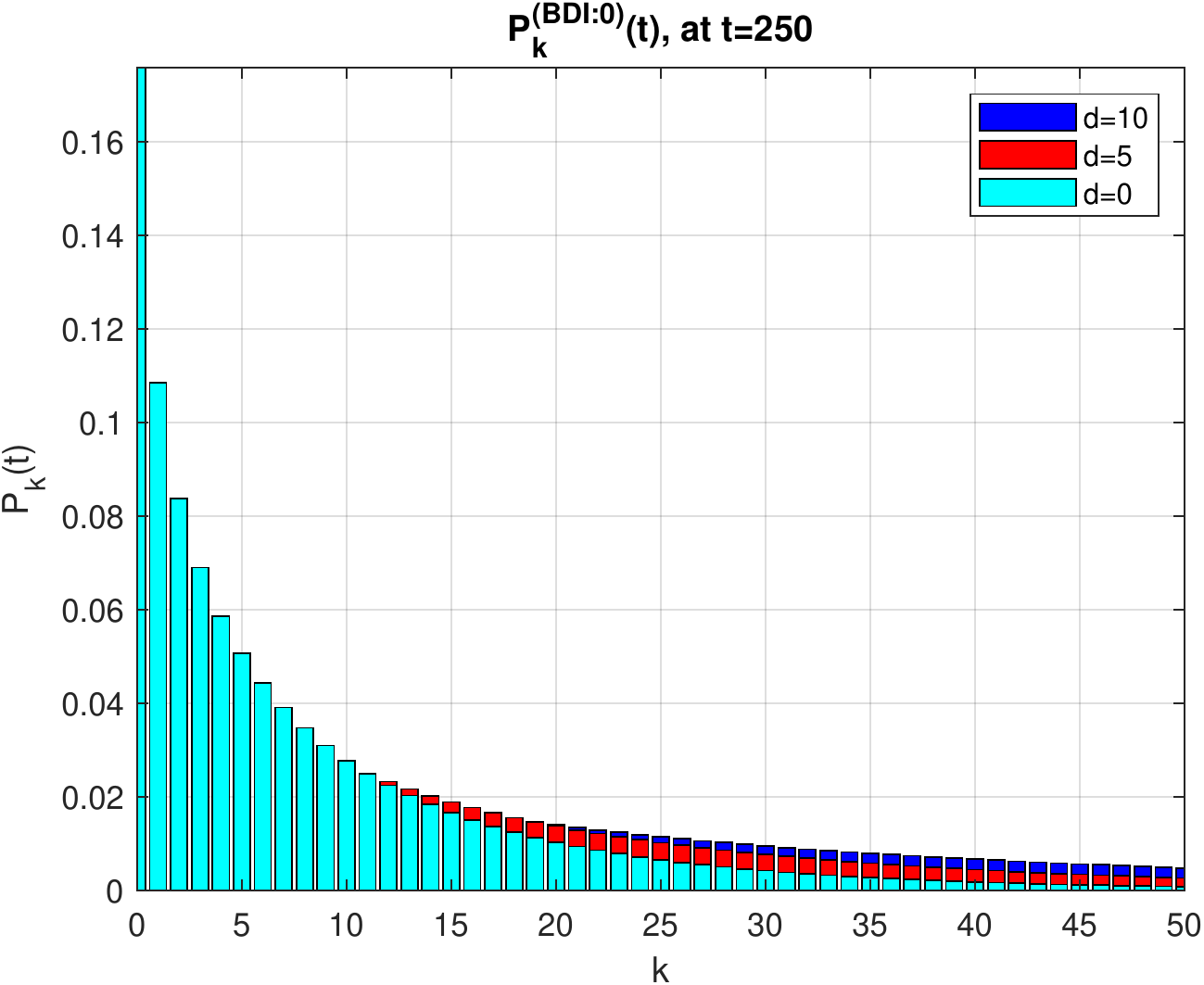}
\caption{\sf $P_k(250)$ of BDI, scaled up, $d=0$ in front.}
\label{fig:P_k(250)-BDI-up-alt}
\end{minipage}
\hspace{0.5cm}
\begin{minipage}[t]{0.30\textwidth}
\centering
\includegraphics[width=\textwidth]{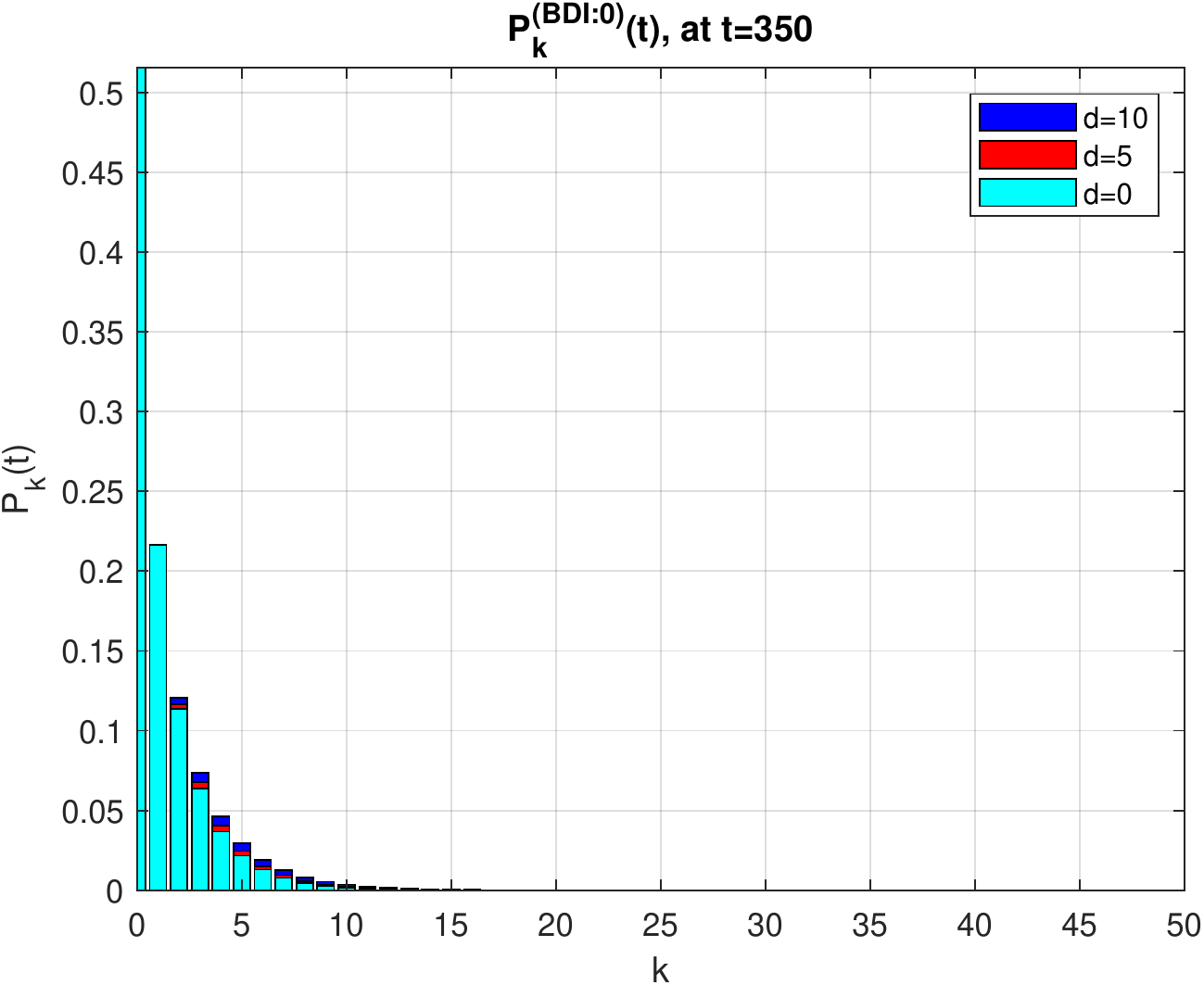}
\caption{\sf $P_k(350)$ of BDI, scaled up, $d=0$ in front.}
\label{fig:P_k(350)-BDI-up-alt}
\end{minipage}
\end{figure}

\clearpage

\subsubsection*{Cross Sections of \boldmath{$P^{(BDI:0)}_k(t)$} along $t$ for given $k$}

\begin{figure}[thb]
\begin{minipage}[t]{0.30\textwidth}
\centering
\includegraphics[width=\textwidth]{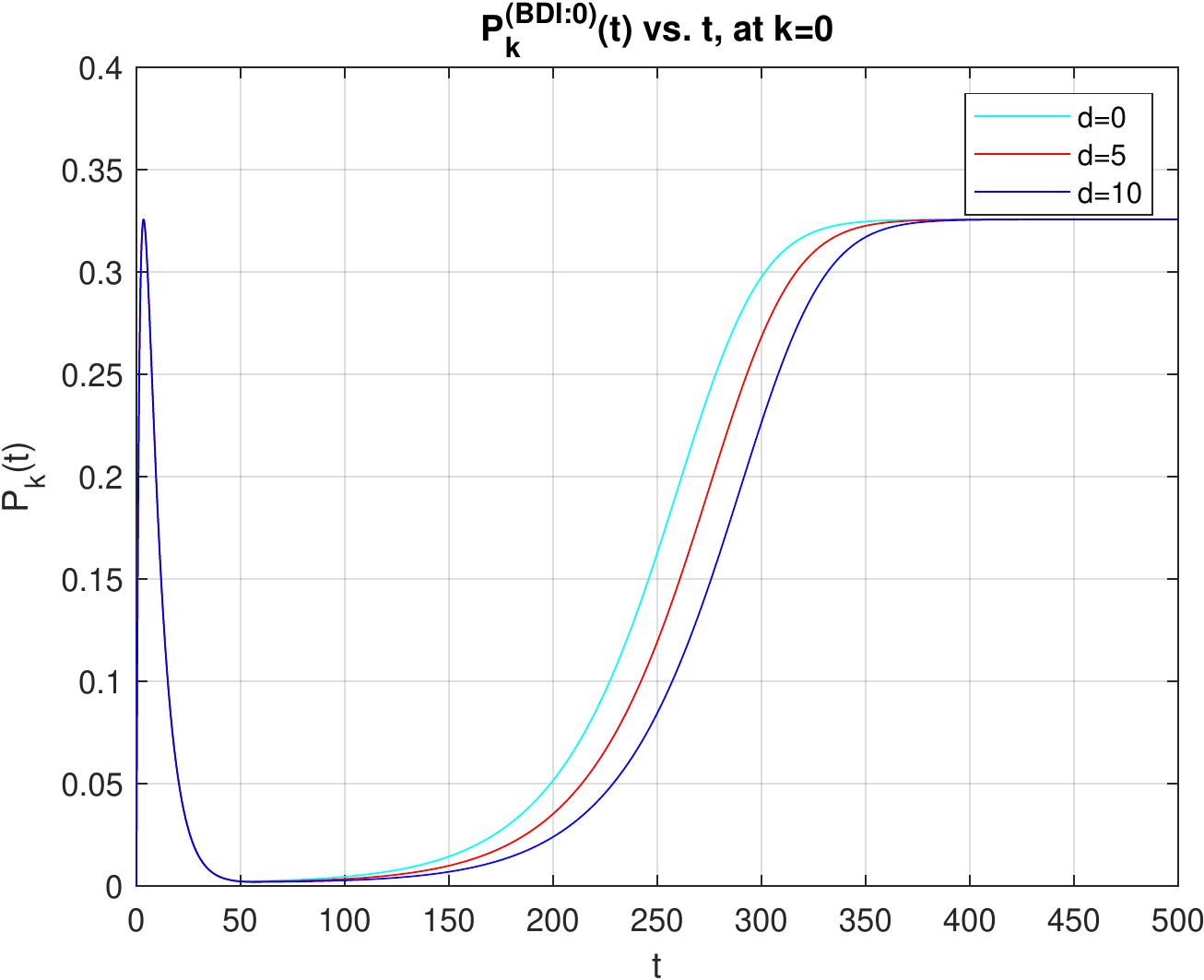}
\caption{\sf $P_0(t)$ of BDI.}
\label{fig:P_0(t)-BDI}
\end{minipage}
\hspace{0.5cm}
\begin{minipage}[t]{0.30\textwidth}
\centering
\includegraphics[width=\textwidth]{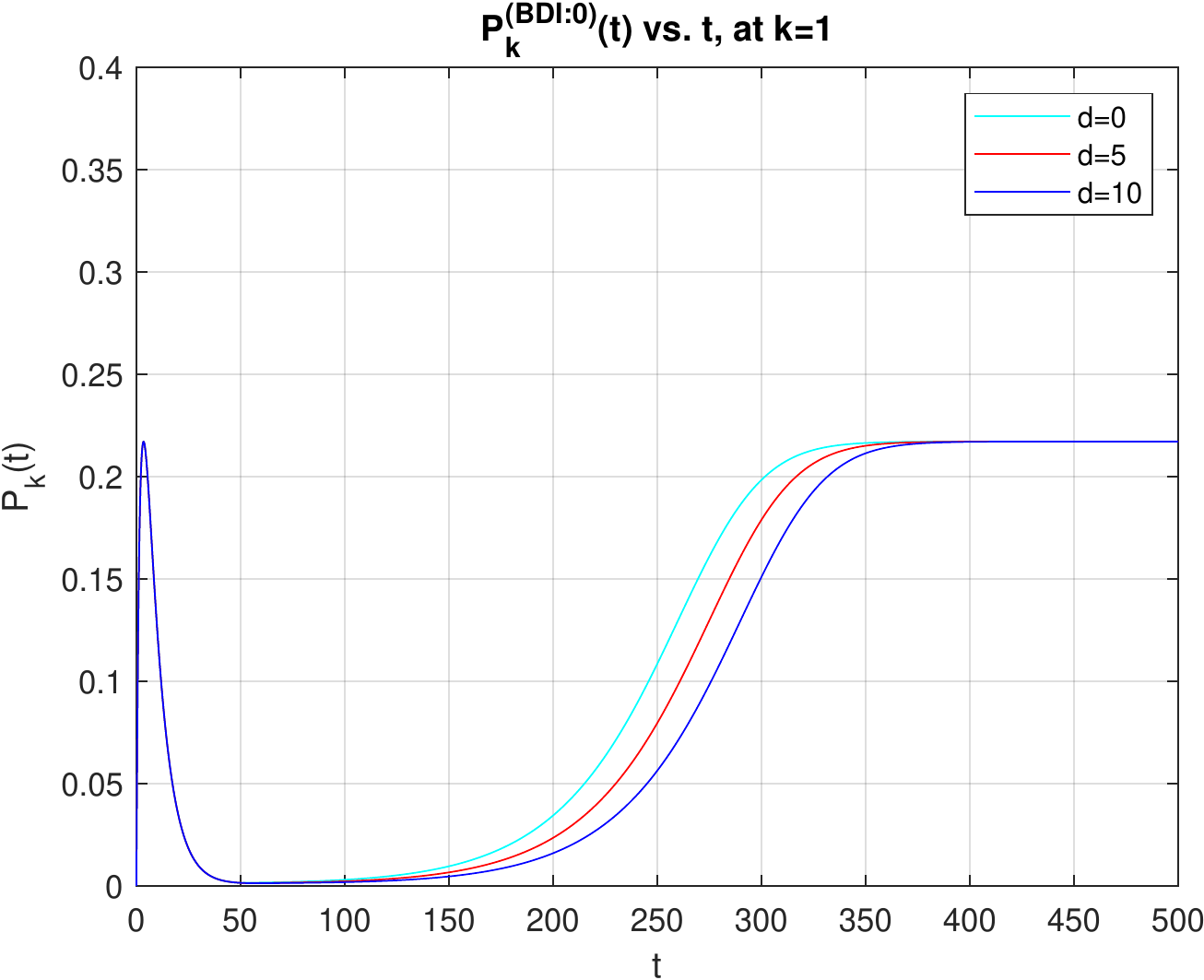}
\caption{\sf $P_1(t)$ of BDI.}
\label{fig:P_1(t)-BDI}
\end{minipage}
\hspace{0.5cm}
\begin{minipage}[t]{0.30\textwidth}
\centering
\includegraphics[width=\textwidth]{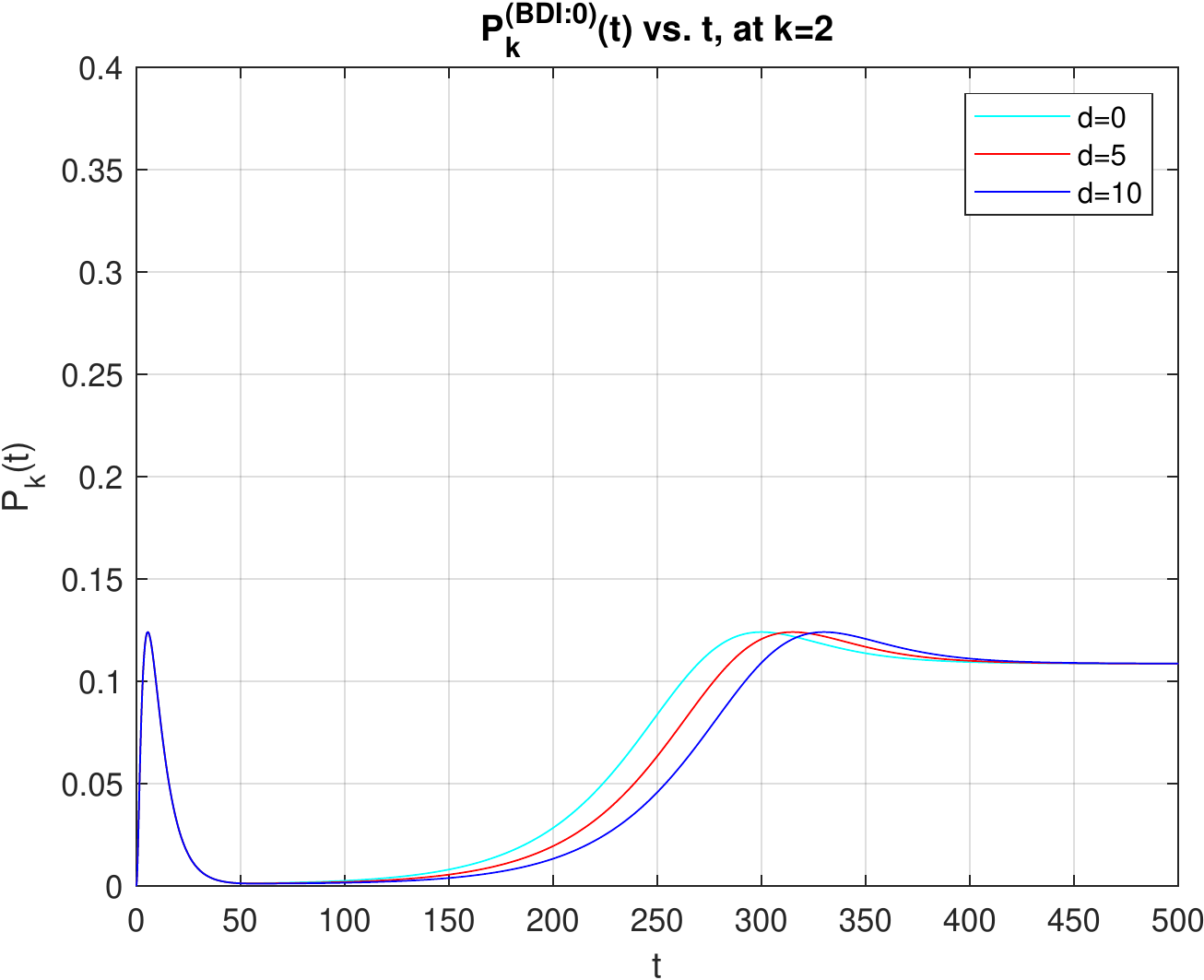}
\caption{\sf $P_2(t)$ of BDI.}
\label{fig:P_2(t)-BDI}
\end{minipage}
\end{figure}
\begin{figure}[hbt]
\begin{minipage}[h]{0.30\textwidth}
\centering
\includegraphics[width=\textwidth]{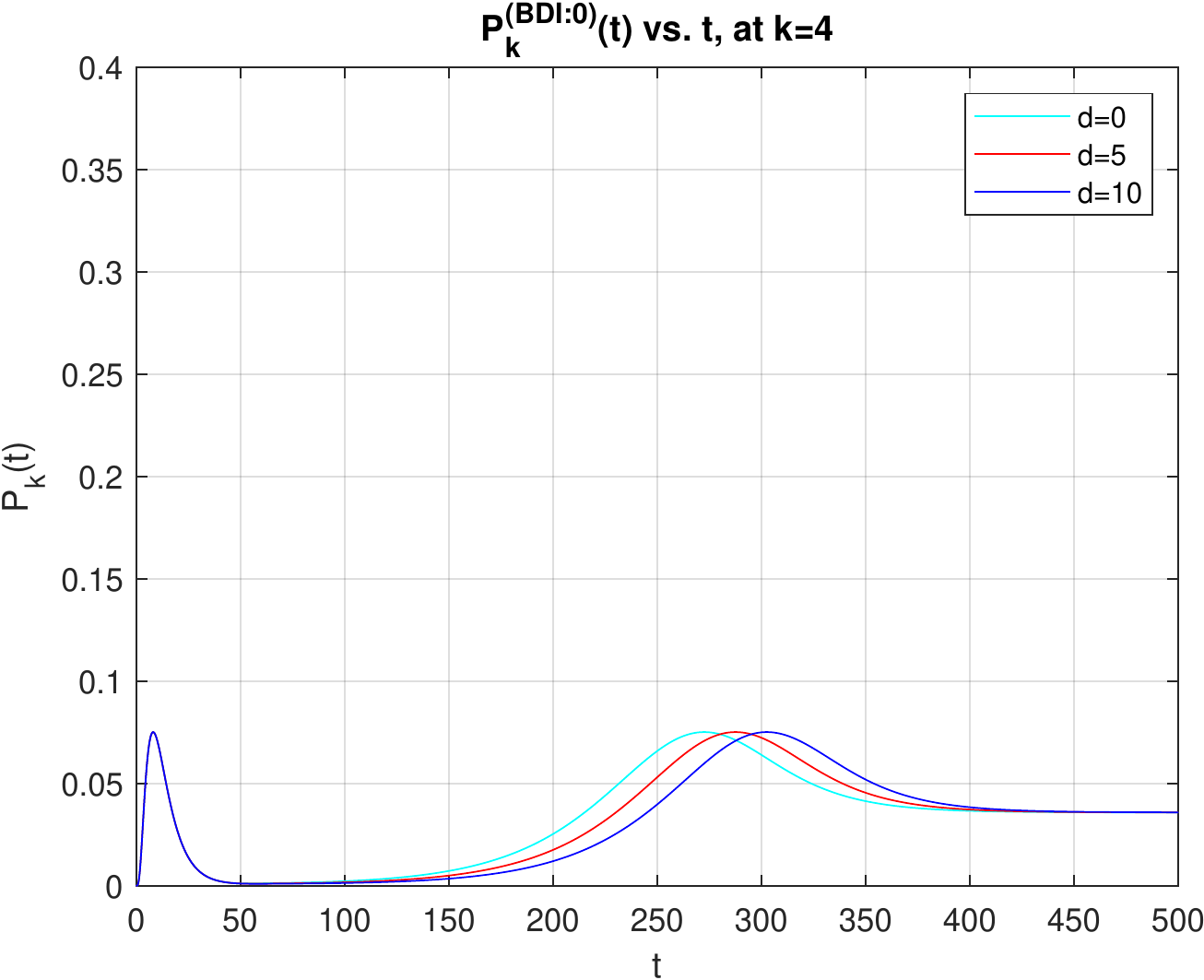}
\caption{\sf $P_4(t)$ of BDI.}
\label{fig:P_4(t)-BDI}
\end{minipage}
\hspace{0.5cm}
\begin{minipage}[h]{0.30\textwidth}
\centering
\includegraphics[width=\textwidth]{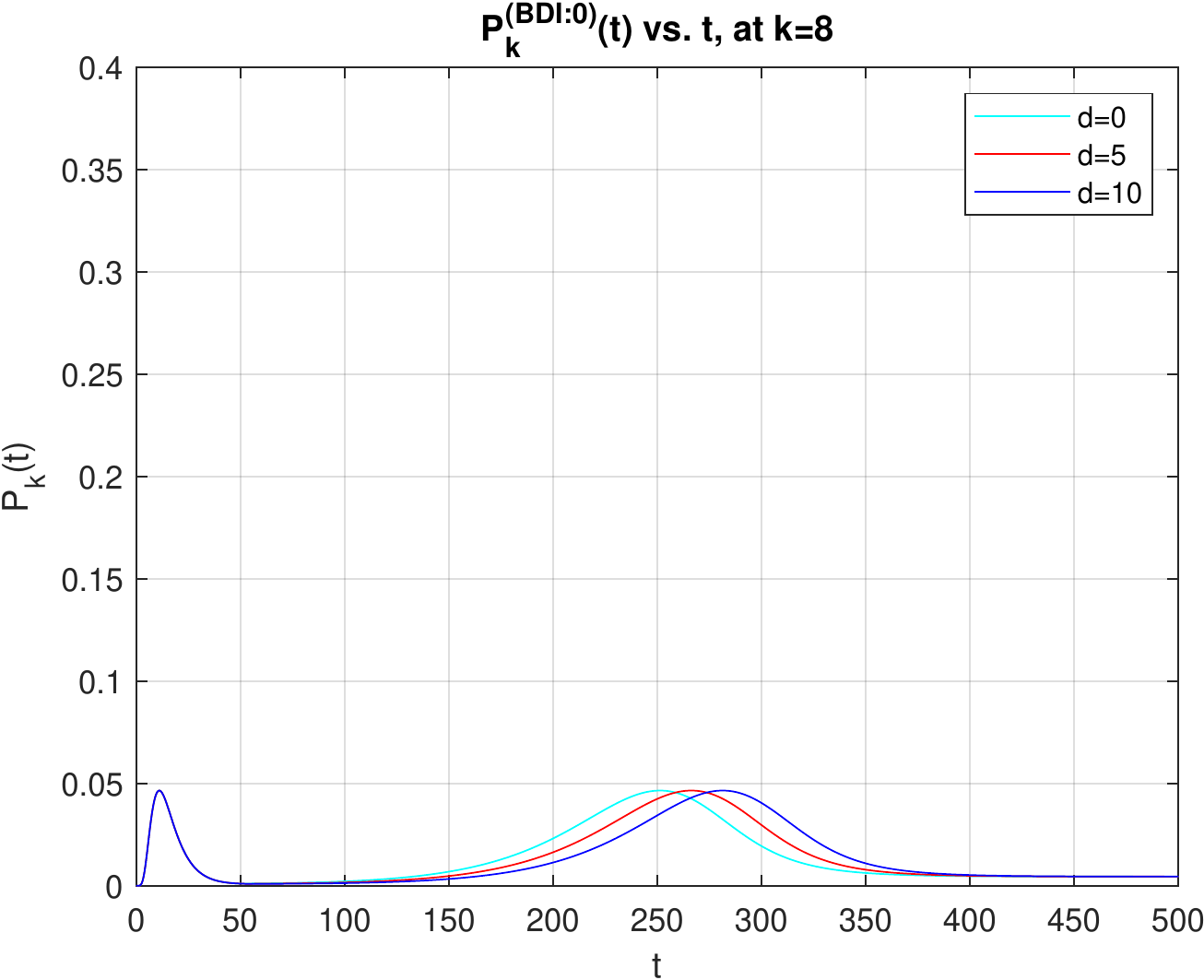}
\caption{\sf $P_8(t)$ of BDI.}
\label{fig:P_8(t)-BDI}
\end{minipage}
\hspace{0.5cm}
\begin{minipage}[h]{0.30\textwidth}
\centering
\includegraphics[width=\textwidth]{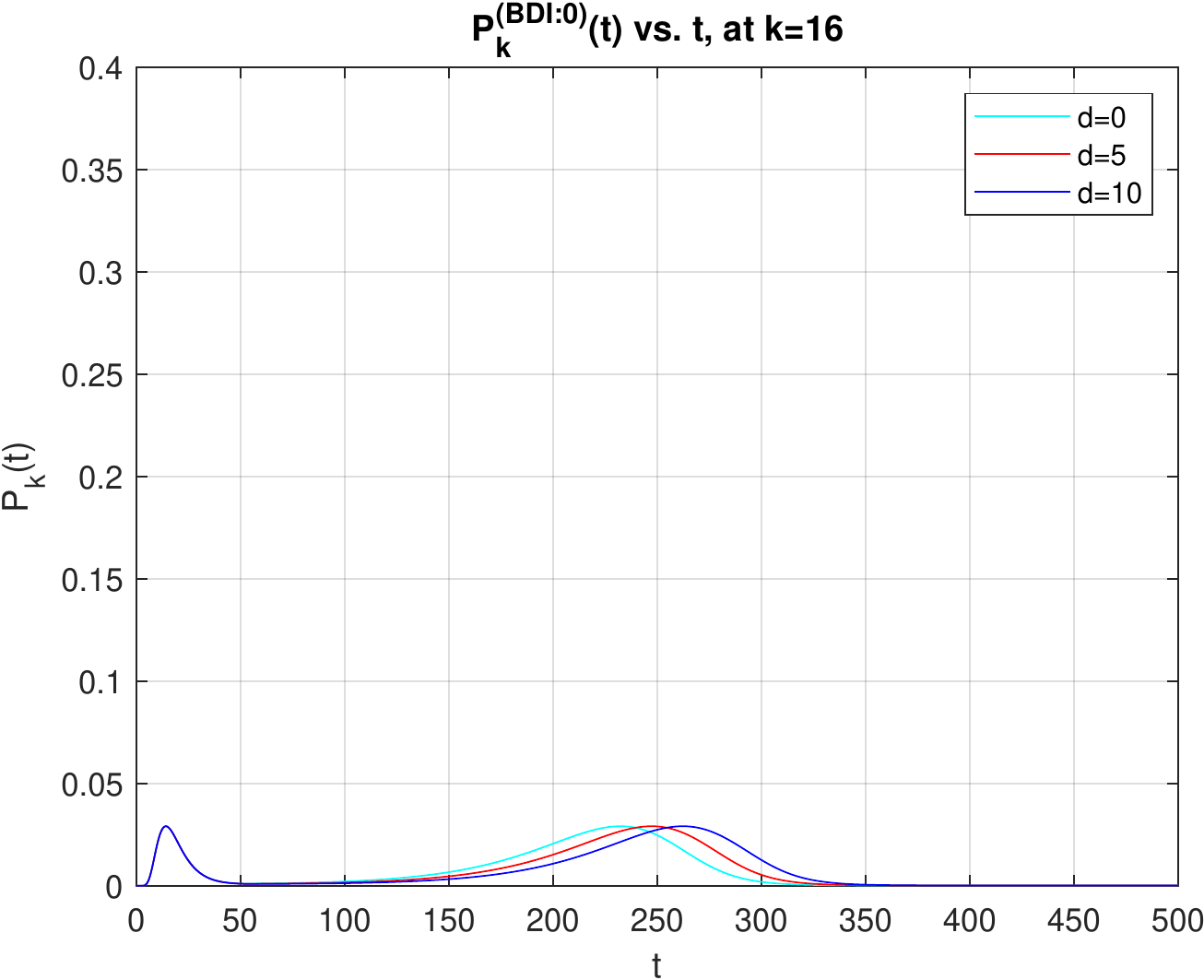}
\caption{\sf $P_{16}(t)$ of BDI.}
\label{fig:P_16(t)-BDI}
\end{minipage}
\end{figure}
\begin{figure}[bht]
\begin{minipage}[b]{0.30\textwidth}
\centering
\includegraphics[width=\textwidth]{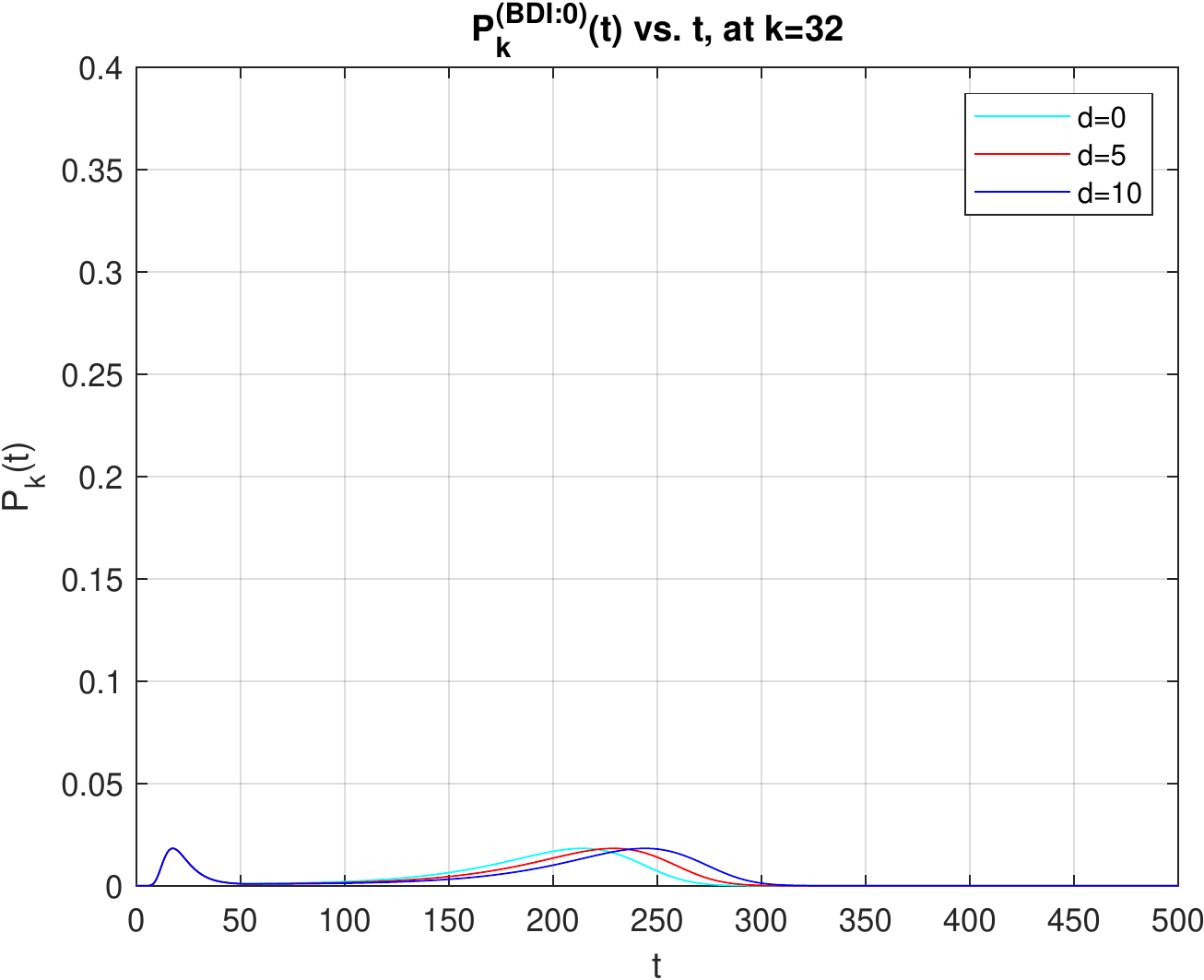}
\caption{\sf $P_{32}(t)$ of BDI.}
\label{fig:P_32(t)-BDI}
\end{minipage}
\hspace{0.5cm}
\begin{minipage}[b]{0.30\textwidth}
\centering
\includegraphics[width=\textwidth]{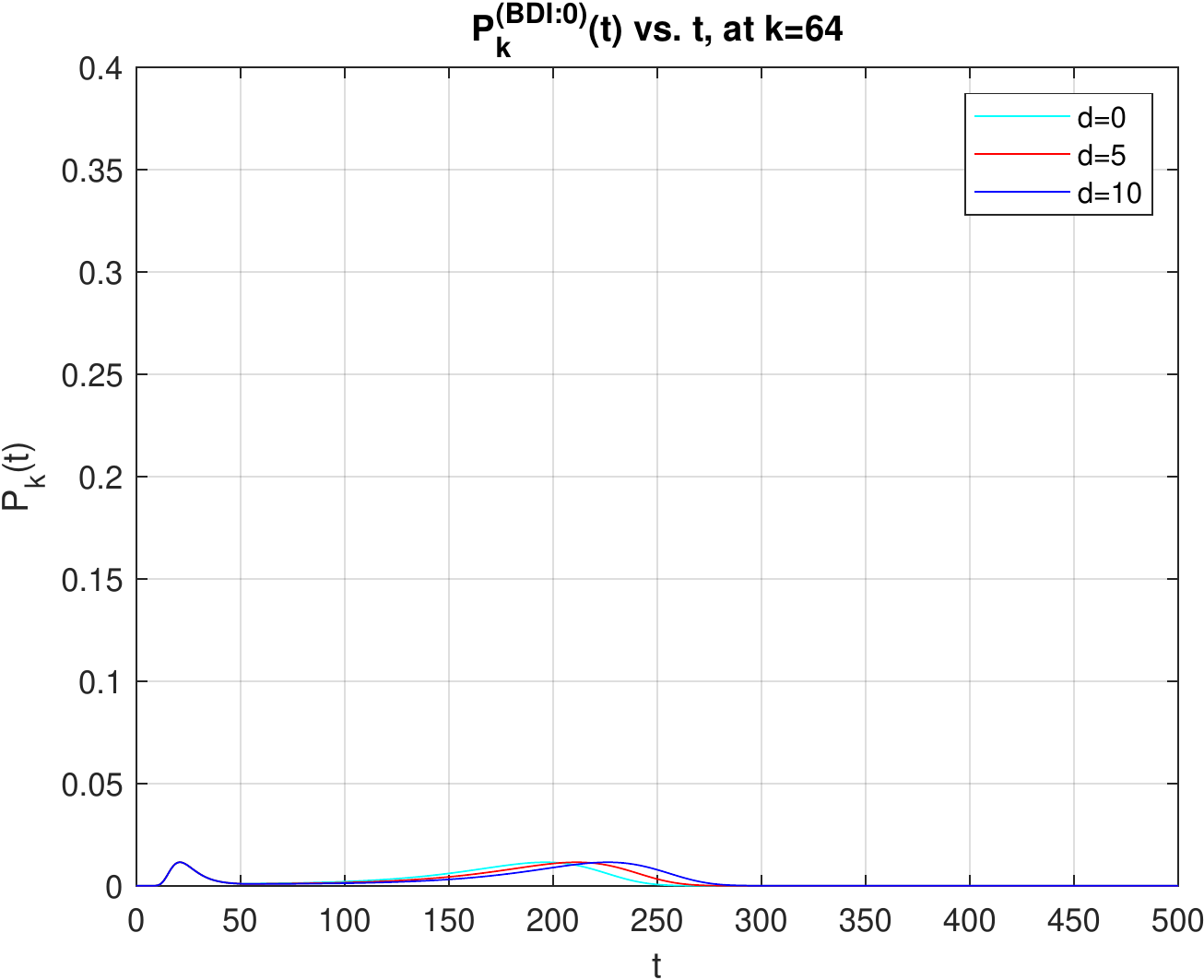}
\caption{\sf $P_{64}(t)$ of BDI.}
\label{fig:P_64(t)-BDI}
\end{minipage}
\hspace{0.5cm}
\begin{minipage}[b]{0.30\textwidth}
\centering
\includegraphics[width=\textwidth]{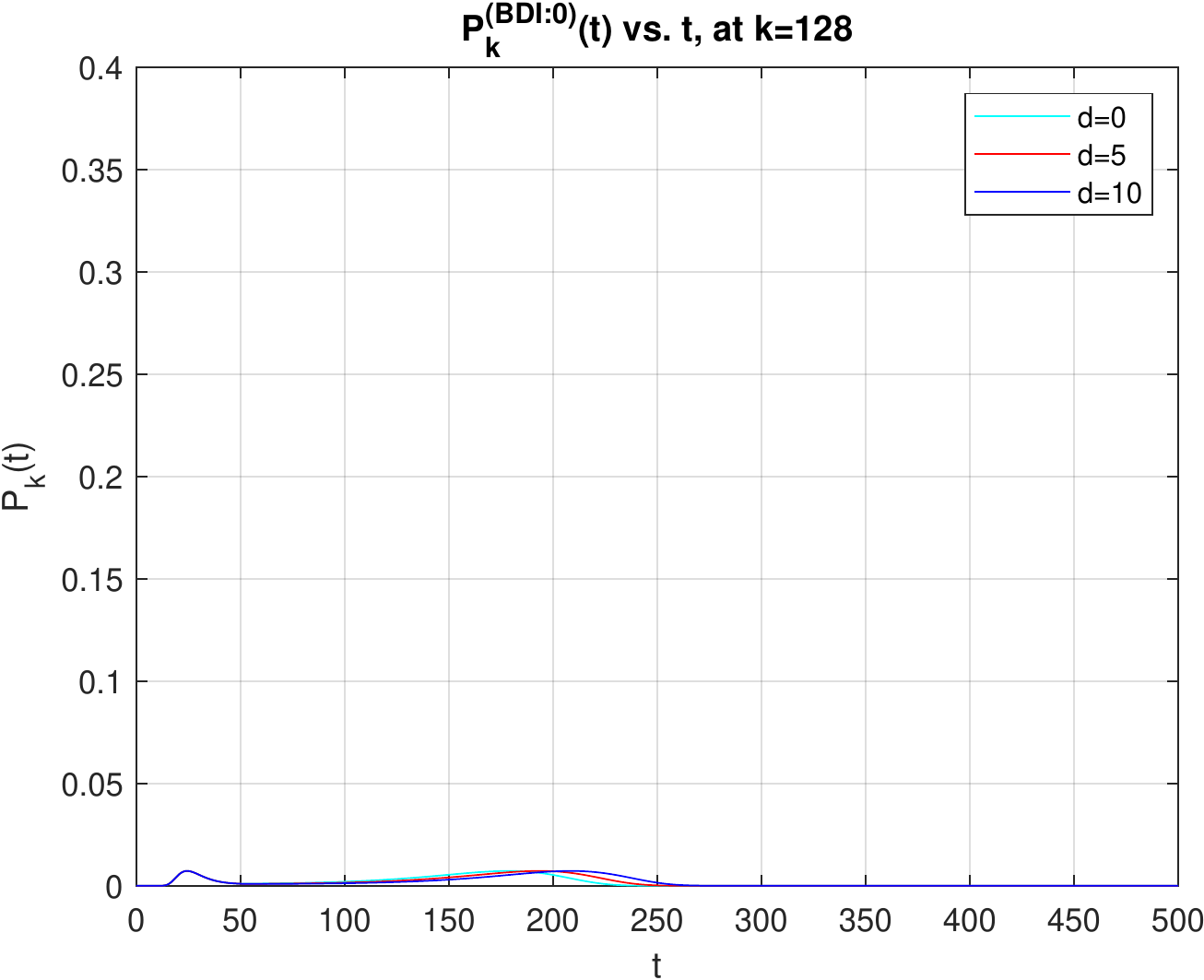}
\caption{\sf $P_{128}(t)$ of BDI.}
\label{fig:P_128(t)-BDI}
\end{minipage}
\end{figure}

\clearpage

\section{Simulating a Time-Nonhomogeneous BD Process with \boldmath{$I_0=1$}} \label{sec:nonhomo-BD-simulation}

We now report on the results of simulating the BD process without immigration to compare them against the analysis conducted in the previous report \cite{kobayashi:2021a}.  We assume the same $\lambda(t)$ and $\mu(t)$ as in the numerical example analyzed in \cite{kobayashi:2021a}.  We consider the case of $d=5$ [days] i.e., the transition curve of $\lambda(t)$ from $\lambda_0=0.3$ to $\lambda_1=0.06$ as shown in the red curve in Figures 1 and 2, ibid. The $\mu(t)$ is treated as constant 0.1.

We assume $I_0=1$ in all cases.  Cases with $I_0>1$ will be discussed in a subsequent report \cite{kobayashi:2021bb}.

\subsection{The Process $I_{BD:1}(t)$}

We conducted simulation experiments by running an event-driven simulator implemented in MATLAB, by extending the simulation scripts used in our earlier simulation experiments \cite{kobayashi:2020b-arXiv} done for time-homogeneous models. 
\begin{figure}[thb]
\begin{minipage}[t]{0.50\textwidth}
\centering
\includegraphics[width=\textwidth]{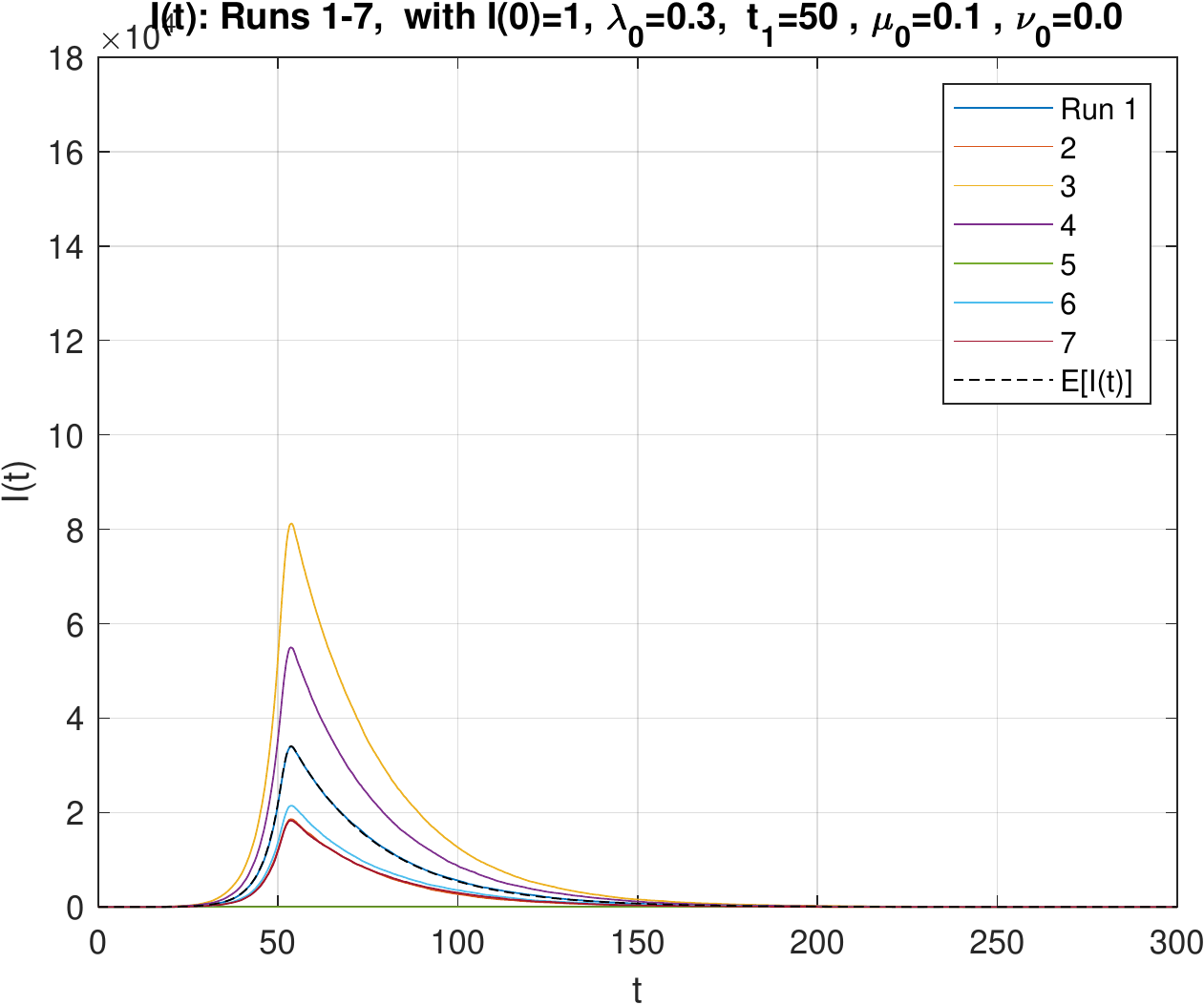}
\caption{\sf Runs 1-7: $I(t)$ of BD process with $I_0=1$.}
\label{fig:BD-I(t)_Runs_1-7-I_0=1}
\end{minipage}
\hspace{0.5cm}
\begin{minipage}[t]{0.50\textwidth}
\centering
\includegraphics[width=\textwidth]{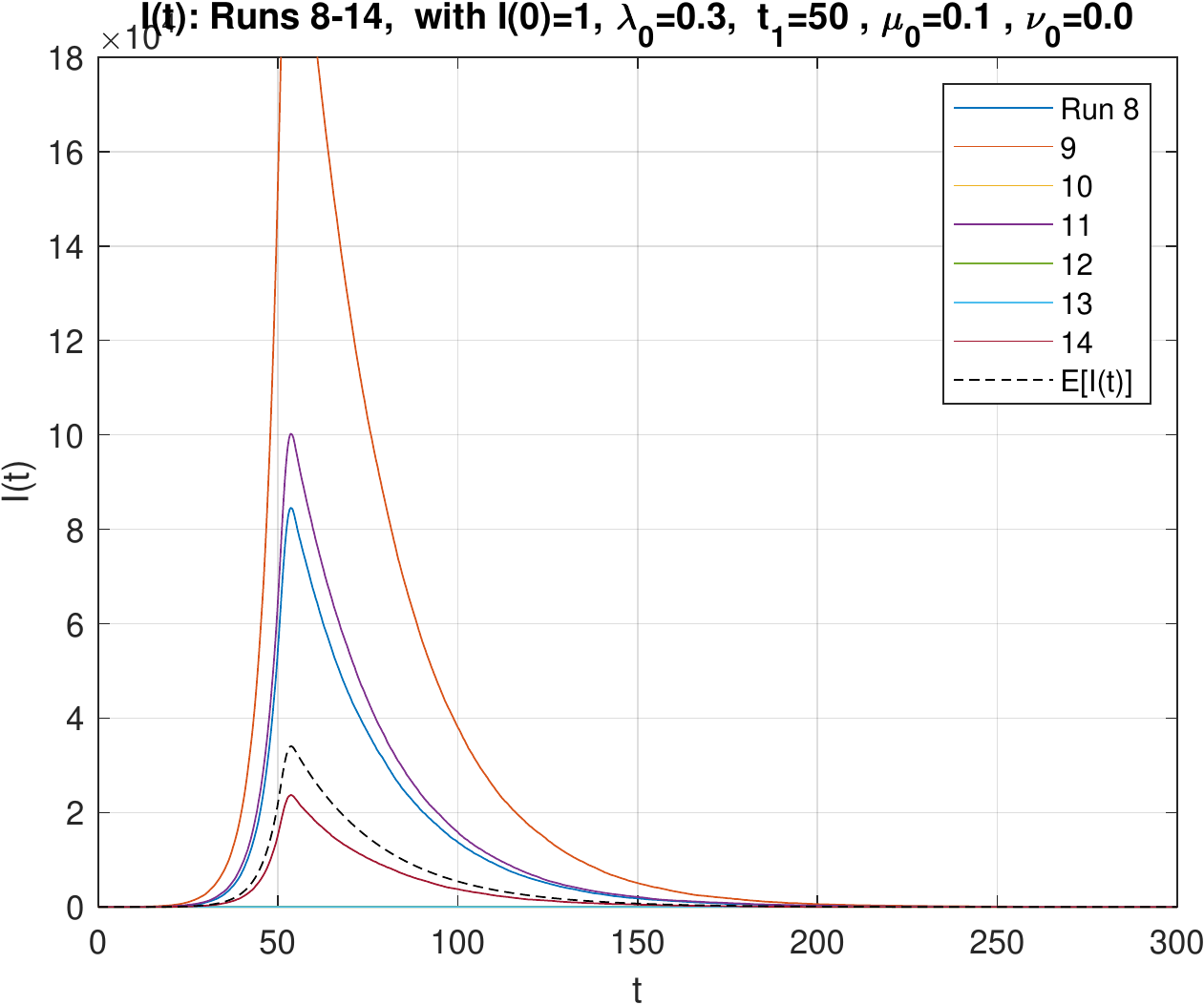}
\caption{\sf Runs 8-14: $I(t)$ of BD process with $I_0=1$.}
\label{fig:BD-I(t)_Runs_8-14-I_0=1}
\end{minipage}
\end{figure}
\begin{figure}
\begin{minipage}[t]{0.50\textwidth}
\centering
\includegraphics[width=\textwidth]{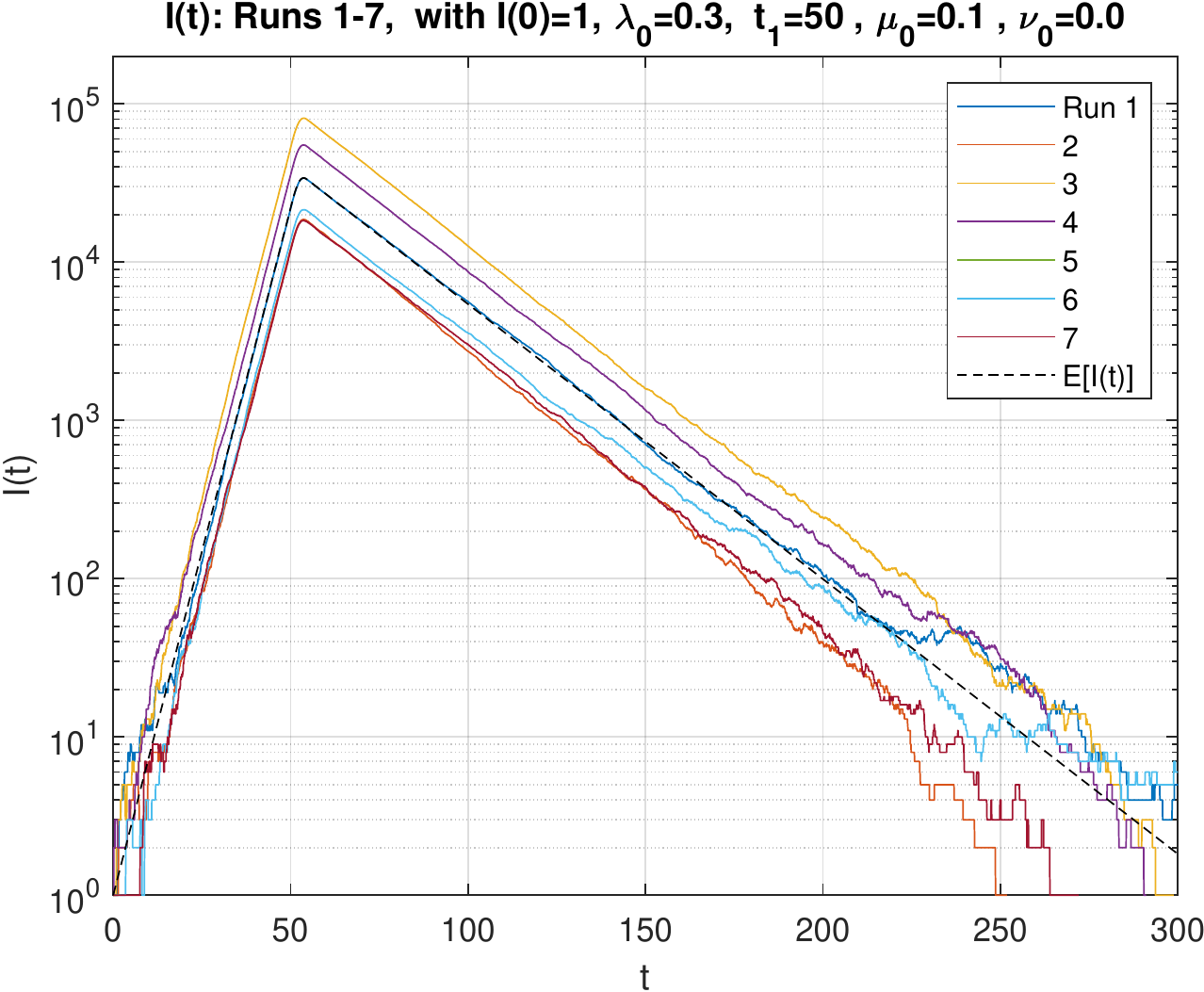}
\caption{\sf Runs 1-7: Semi-log plot: $I(t)$ of BD process with $I_0=1$.}
\label{fig:BD-Semilog_I(t)_Runs_1-7-I_0=1}
\end{minipage}
\hspace{0.5cm}
\begin{minipage}[t]{0.50\textwidth}
\centering
\includegraphics[width=\textwidth]{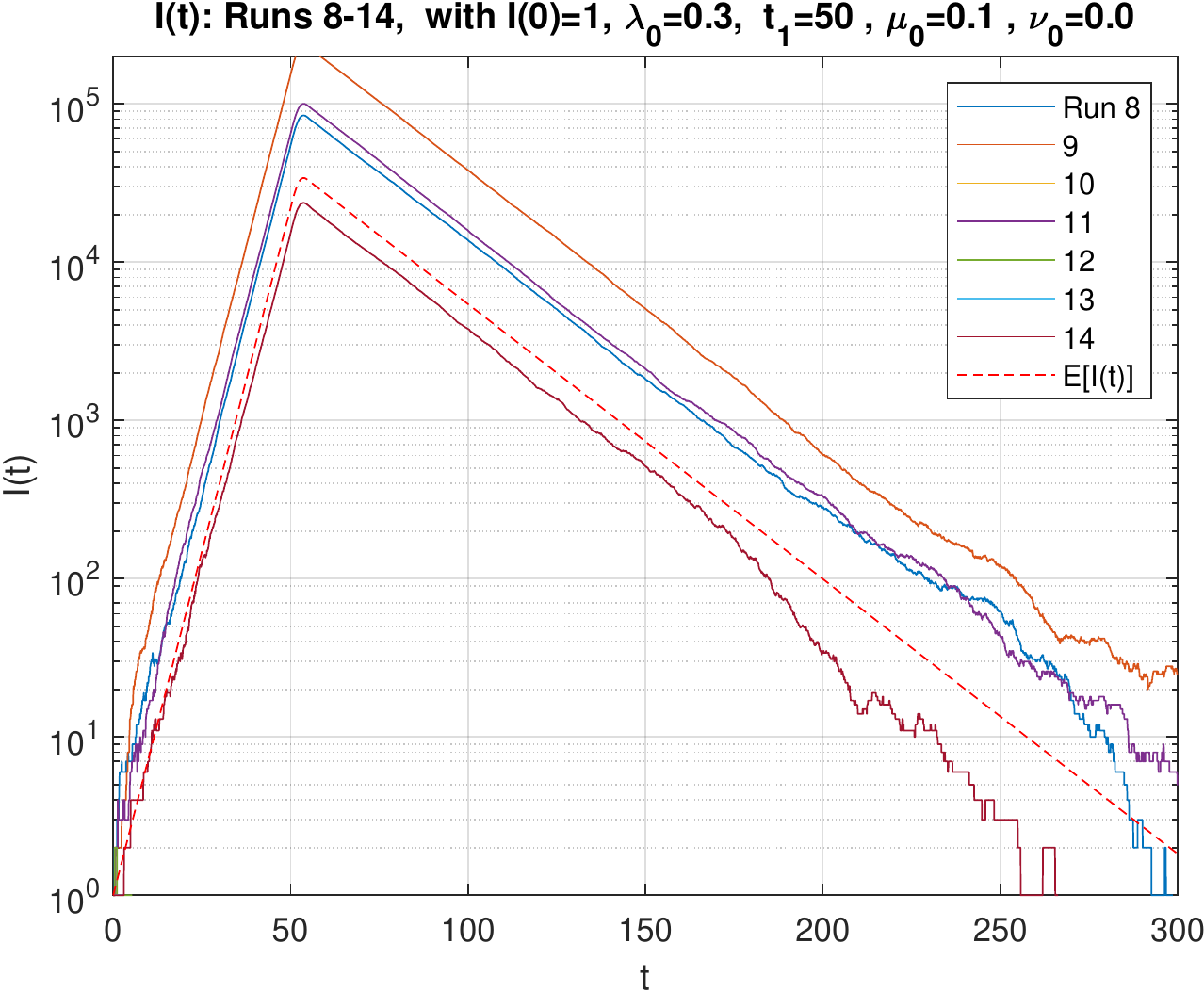}
\caption{\sf Runs 8-14: Semi-log plot: $I(t)$ of BD process with $I_0=1$.}
\label{fig:BD-Semilog_I(t)_Runs_8-14-I_0=1}
\end{minipage}
\end{figure}
\begin{figure}
\begin{minipage}[t]{0.50\textwidth}
\centering
\includegraphics[width=\textwidth]{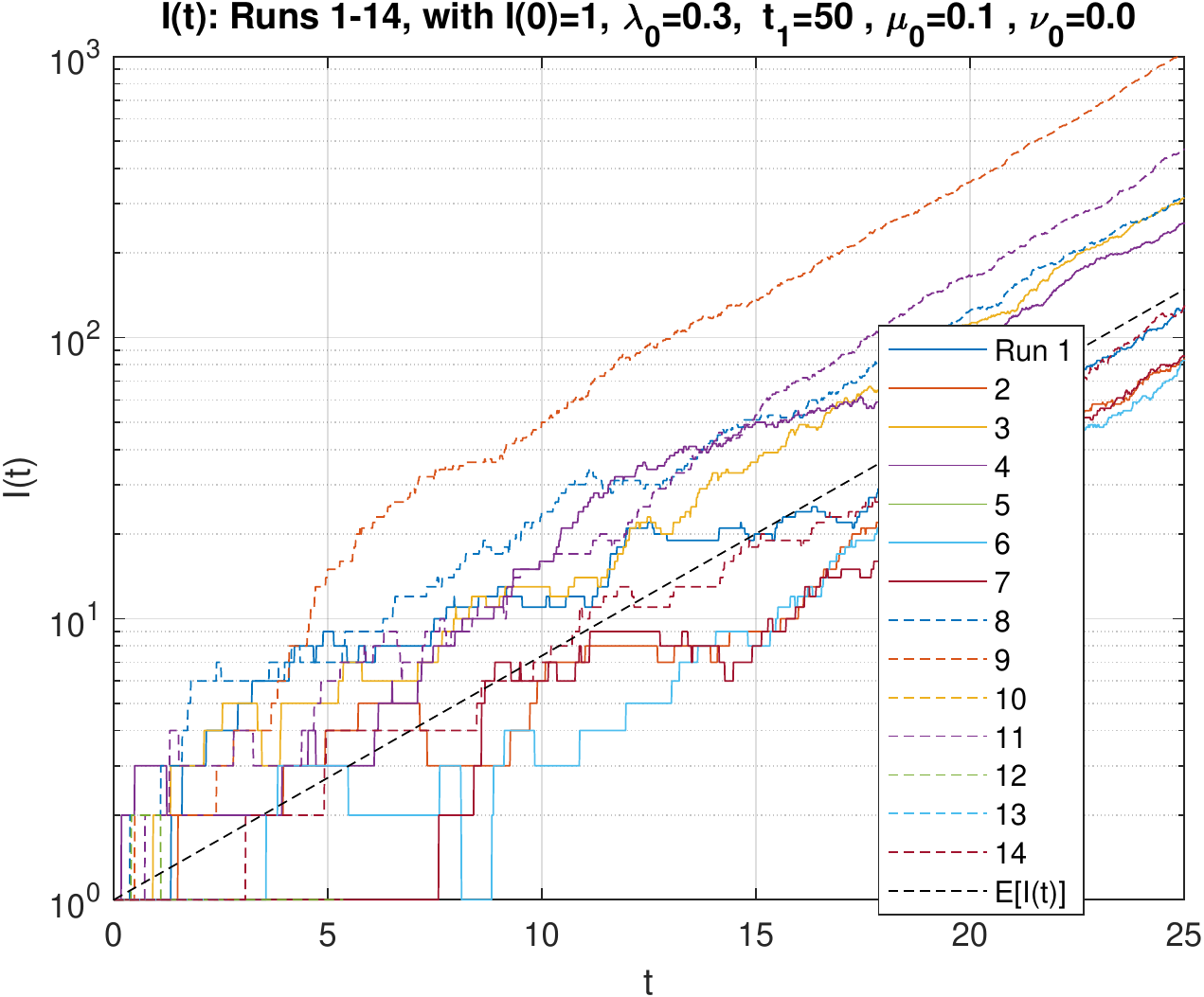}
\caption{\sf Runs 1-14: Initial 25 days, of $I(t)$ of BD process with $I_0=1$.}
\label{fig:BD-Semilog_I(t)_Runs_1-14-t_max=25-I_0=1}
\end{minipage}
\end{figure}
\begin{figure}
\begin{minipage}[t]{0.30\textwidth}
\centering
\includegraphics[width=\textwidth]{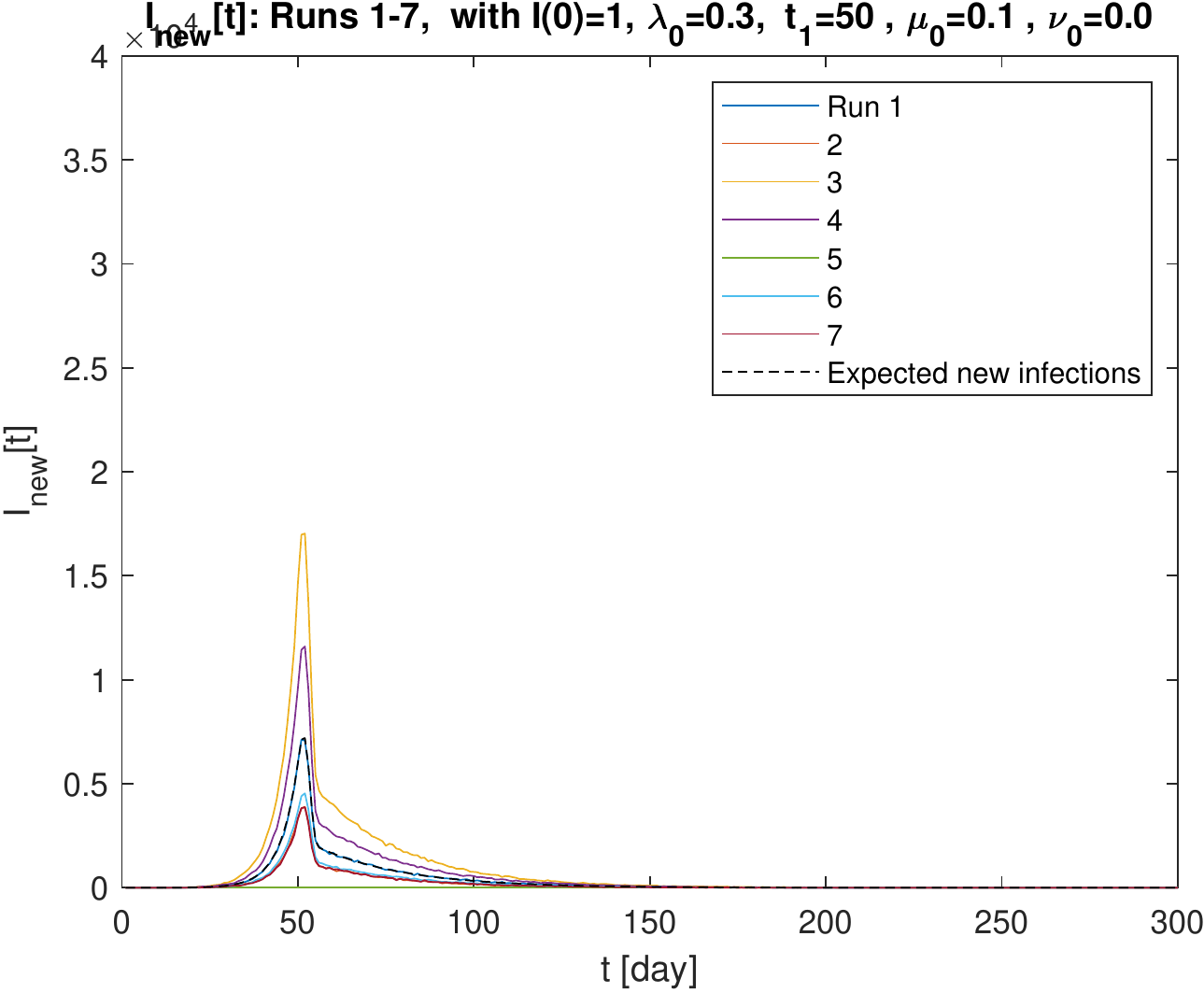}
\caption{\sf Runs 1-7: Daily new infections $I_{new}[t]$ of BD process with $I_0=1$.}
\label{fig:BD-I_new[t]_Runs_1-7-I_0=1}
\end{minipage}
\hspace{0.5cm}
\begin{minipage}[t]{0.30\textwidth}
\centering
\includegraphics[width=\textwidth]{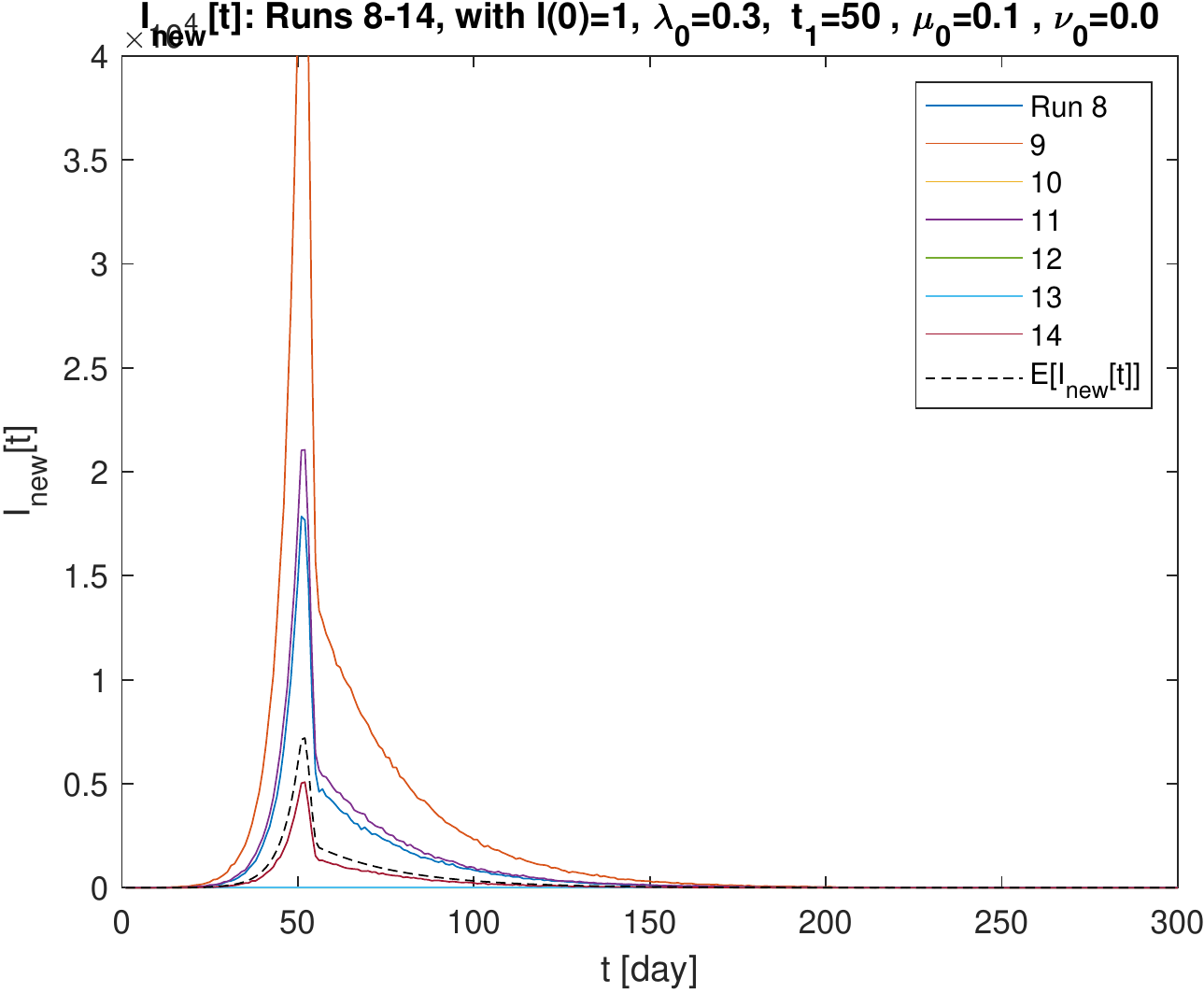}
\caption{\sf Runs 8-14: Daily new infections $I_{new}[t]$ of BD process with $I_0=1$.}
\label{fig:BD-I_new[t]_Runs_8-14-I_0=1}
\end{minipage}
\hspace{0.5cm}
\begin{minipage}[t]{0.30\textwidth}
\centering
\includegraphics[width=\textwidth]{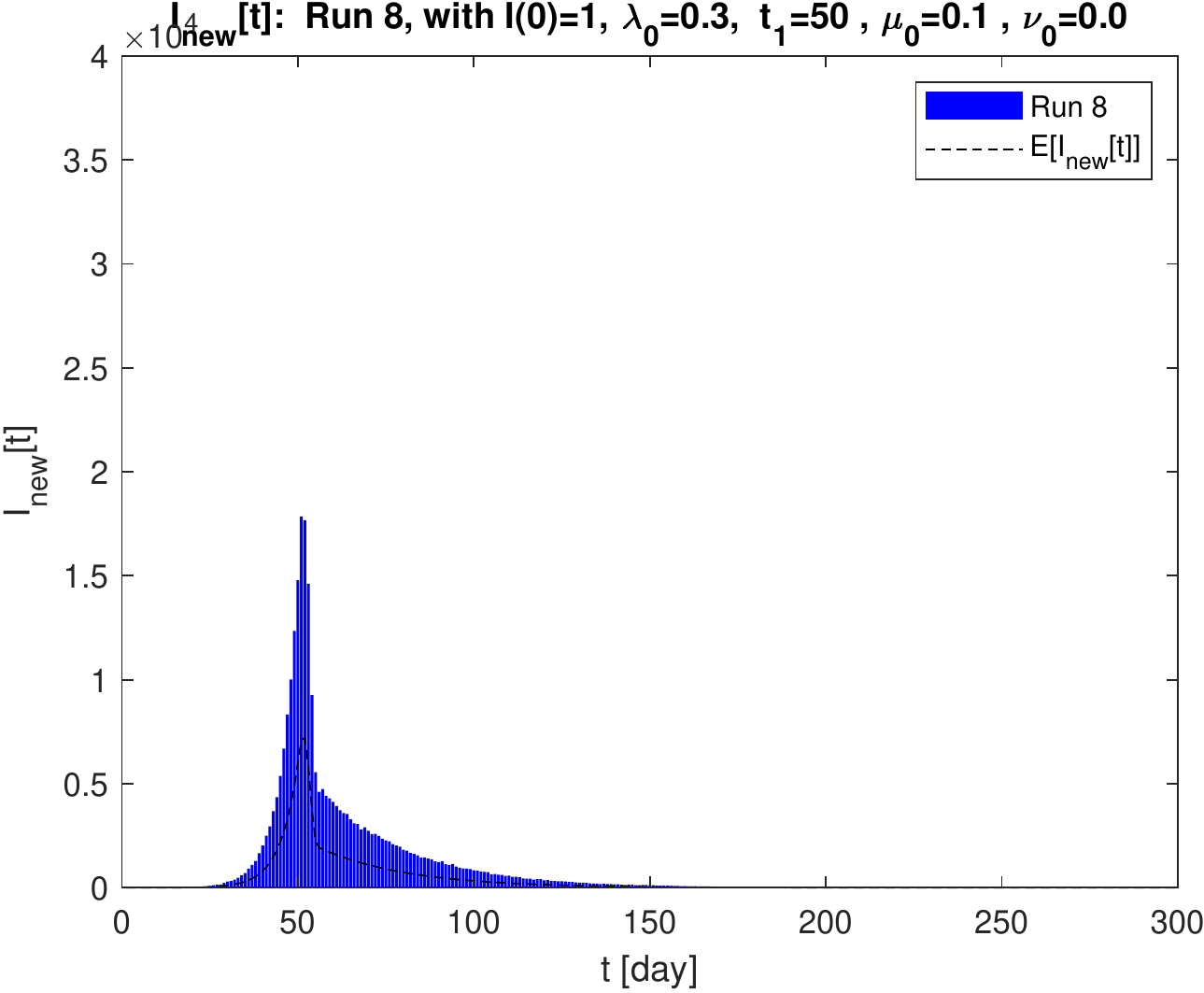}
\caption{\sf Run 8: Daily new infections $I_{new}[t]$ of BD process with $I_0=1$}
\label{fig:BD-I_new[t]_Run_8-I_0=1}
\end{minipage}
\end{figure}

Figure \ref{fig:BD-I(t)_Runs_1-7-I_0=1} \& \ref{fig:BD-I(t)_Runs_8-14-I_0=1} show 14 consecutive  simulation runs plotted in two figures. They represent 14 sample paths obtained in one execution of the simulation script which perform 14 runs, one after another. The seed of a random number sequence was not chosen by us. 
We set $B(0)=R(0)=0$, and $I(0)=I_0=0$ in each run. 

In Figure \ref{fig:BD-I(t)_Runs_1-7-I_0=1}, 
Run 5 terminates at $t=1.817$ [days], when the infected person present at $t=0$ recovers/removed or dies before he/she has a chance to infect another person.   
Run 3 (yellow) is the largest, Run 4 is the second largest, followed by Runs 1 6 7 and 2.

In Figure \ref{fig:BD-I(t)_Runs_8-14-I_0=1}, Runs 10 and 13 terminate at $t=2.82$ and 2.65 [days], respectively,  without infecting any, whereas in Run 12, one infection takes place, both die by $t==5.38$ [days]. Run 9 is by far the largest, and its peak $I(53.71)=236,848$, which is nearly 7 times as large as the expected value $\oI(53.66)-34,039$, predicted by the deterministic model.

Out of the 14 runs, 4 runs terminates soon after $t=0$; their proportion $4/14=28/57\%$ is somewhat lower than $\mu/\lambda=33.33\%$, the theoretical probability of extinction. 

Figures \ref{fig:BD-Semilog_I(t)_Runs_1-7-I_0=1} and \ref{fig:BD-Semilog_I(t)_Runs_8-14-I_0=1} are the plots of the same 14 runs in the semi-log scale.  Figure \ref{fig:BD-Semilog_I(t)_Runs_1-14-t_max=25-I_0=1} shows the first 25 days of all the 14 runs.  From these three figures we can confirm that the variability in each run and across different runs are indeed significant when $I(t)$ is less than 100, whereas each $\log I(t)$ moves in a nearly straight line in both increasing and decreasing phases when $I(t)>200$.  This is consistent with the CV (coefficient of variation) curve we obtained and plotted in Figure 15 of \cite{kobayashi:2021a}.

Each birth-death process may vary significantly in an unpredictable manner at any instant throughout the process, and their sum $I(t)$ is most visible in the initial and final phases, where the $I(t)$ is small, say less than 100.  But in the period when $I(t)$ exceeds some number, say, 200, then \emph{the law of large numbers} comes into a play, hence the aggregated of these independent births and deaths are more regular and predictable.

Figure \ref{fig:BD-I_new[t]_Runs_1-7-I_0=1} and \ref{fig:BD-I_new[t]_Runs_8-14-I_0=1} we show the daily statistics of the newly infected.  We defined this quantity in Section 1.3 of \cite{kobayashi:2021a}, and the expected value is plotted in Figure 11 (the red bar graph for the $d=5$ [days] case, which is replicated in dotted black curves here).  Note that the shape of the curve of 
$\oI_{new}[t]$ (see Eqn.(41) of \cite{kobayashi:2021a}) is appreciably different from $\oI(t)$ or $\oR_{new}[t]$ (see Eqn.(40), ibid), because of the multiplicative term $\lambda(u)$ in the integrand in (41), ibid.
Figure \ref{fig:BD-I_new[t]_Run_8-I_0=1} shows the bar graph of $\oI_{new}[t]$ of Run 8, as an example.

\subsection{The Processes $B_{BD:1}(t), R_{BD:1}(t)$ and $D_{BD:1}(t)$}
In this section we show the plots of the cumulative counts of internal infections $B(t)$ (see Figures \ref{fig:BD-B(t)_Runs_1-7-I_0=1}-\ref{fig:BD-Semilog-B(t)_Runs_1-14-t_max=25-I_0=1}), and the cumulative account of those who cease to be infectious any longer, either by recovering, or getting removed (to e.g., hospitals) or dead, denoted 
$R(t)$ (Figures \ref{fig:BD-D(t)_Runs_1-7-I_0=1} - \ref{fig:BD-Semilog-R(t)_Runs_1-14-I_0=1}).

We obtain the cumulative count of deaths $D(t)$ as a sub-process of $R(t)$. When an event of recovery/removal/death occurs in a simulation, we randomly choose and label it as a death. In our simulation we set the fatality rate  $r_f=0.02$ among the infected (see Figures \ref{fig:BD-D(t)_Runs_1-7-I_0=1}-\ref{fig:First_50_days_D(t)}).

\begin{figure}[thb]
\begin{minipage}[t]{0.30\textwidth}
\centering
\includegraphics[width=\textwidth]{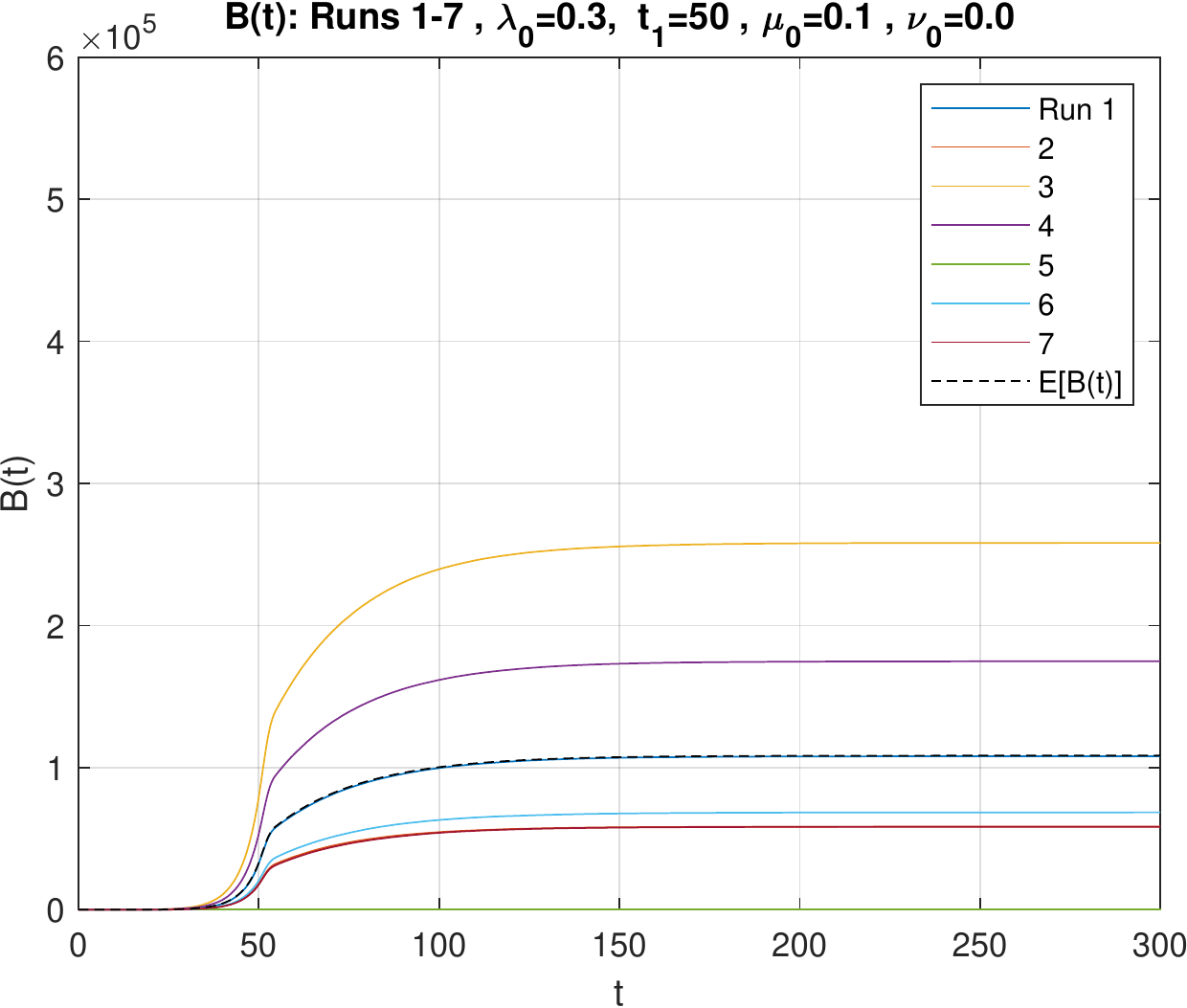}
\caption{\sf Runs 1-7: $B(t)$.}
\label{fig:BD-B(t)_Runs_1-7-I_0=1}
\end{minipage}
\hspace{0.5cm}
\begin{minipage}[t]{0.30\textwidth}
\centering
\includegraphics[width=\textwidth]{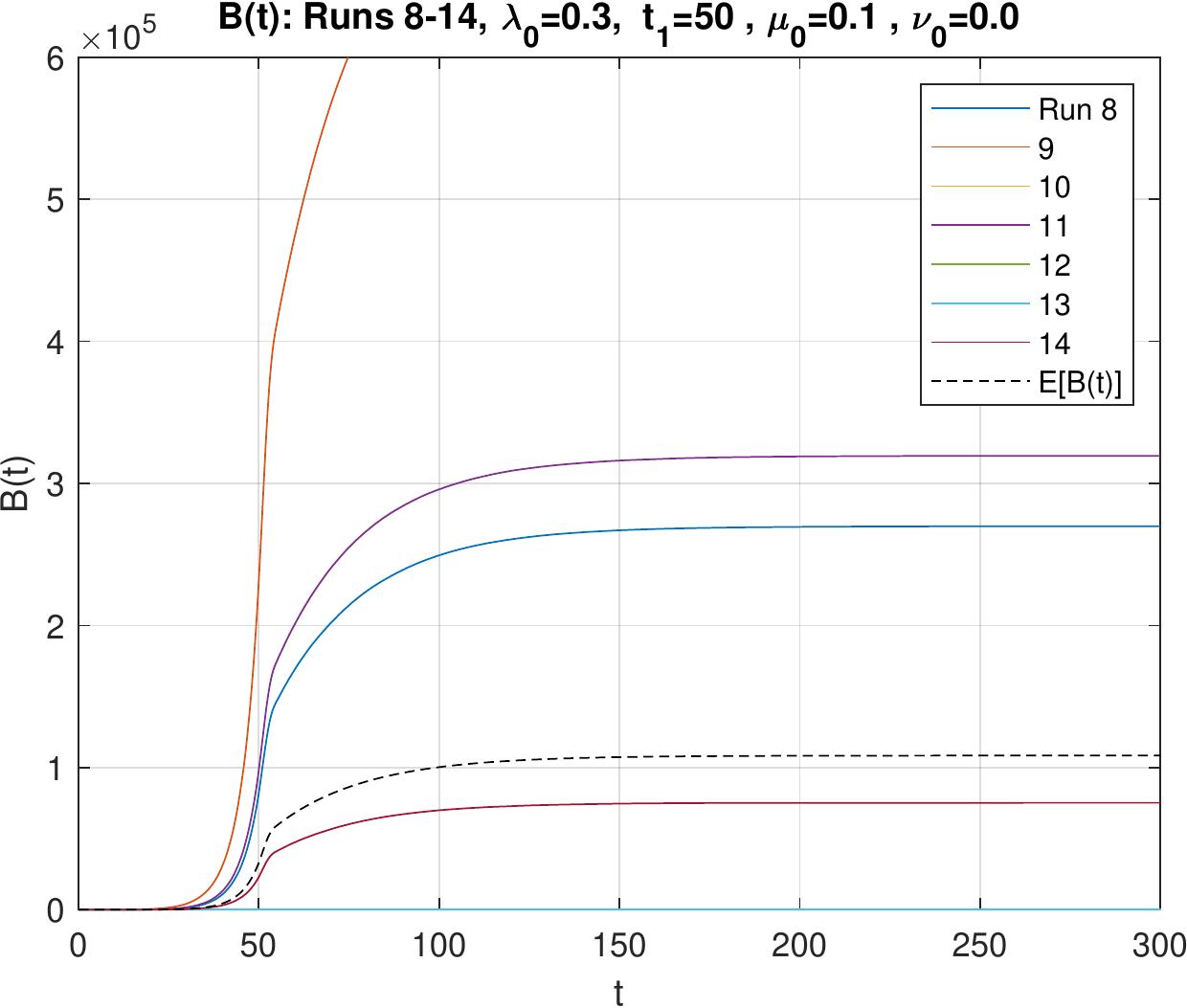}
\caption{\sf Runs 8-14: $B(t)$.}
\label{fig:BD-B(t)_Runs_8-14-I_0=1}
\end{minipage}
\hspace{0.5cm}
\begin{minipage}[t]{0.30\textwidth}
\centering
\includegraphics[width=\textwidth]{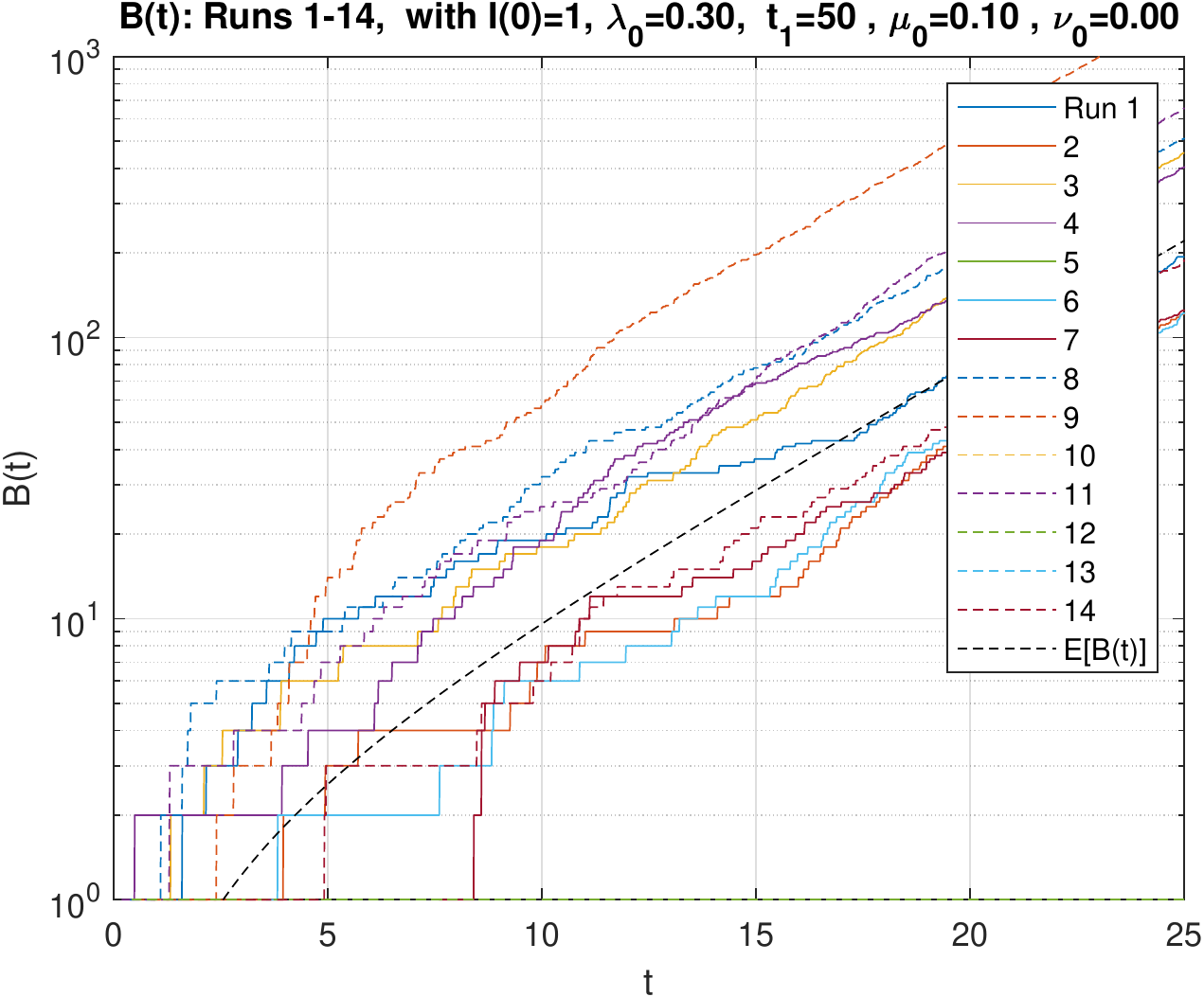}
\caption{\sf Runs 1-14: Initial 25 day, Semi-log plots of $B(t)$.}
\label{fig:BD-Semilog-B(t)_Runs_1-14-t_max=25-I_0=1}
\end{minipage}
\end{figure}
\begin{figure}[thb]
\begin{minipage}[t]{0.30\textwidth}
\centering
\includegraphics[width=\textwidth]{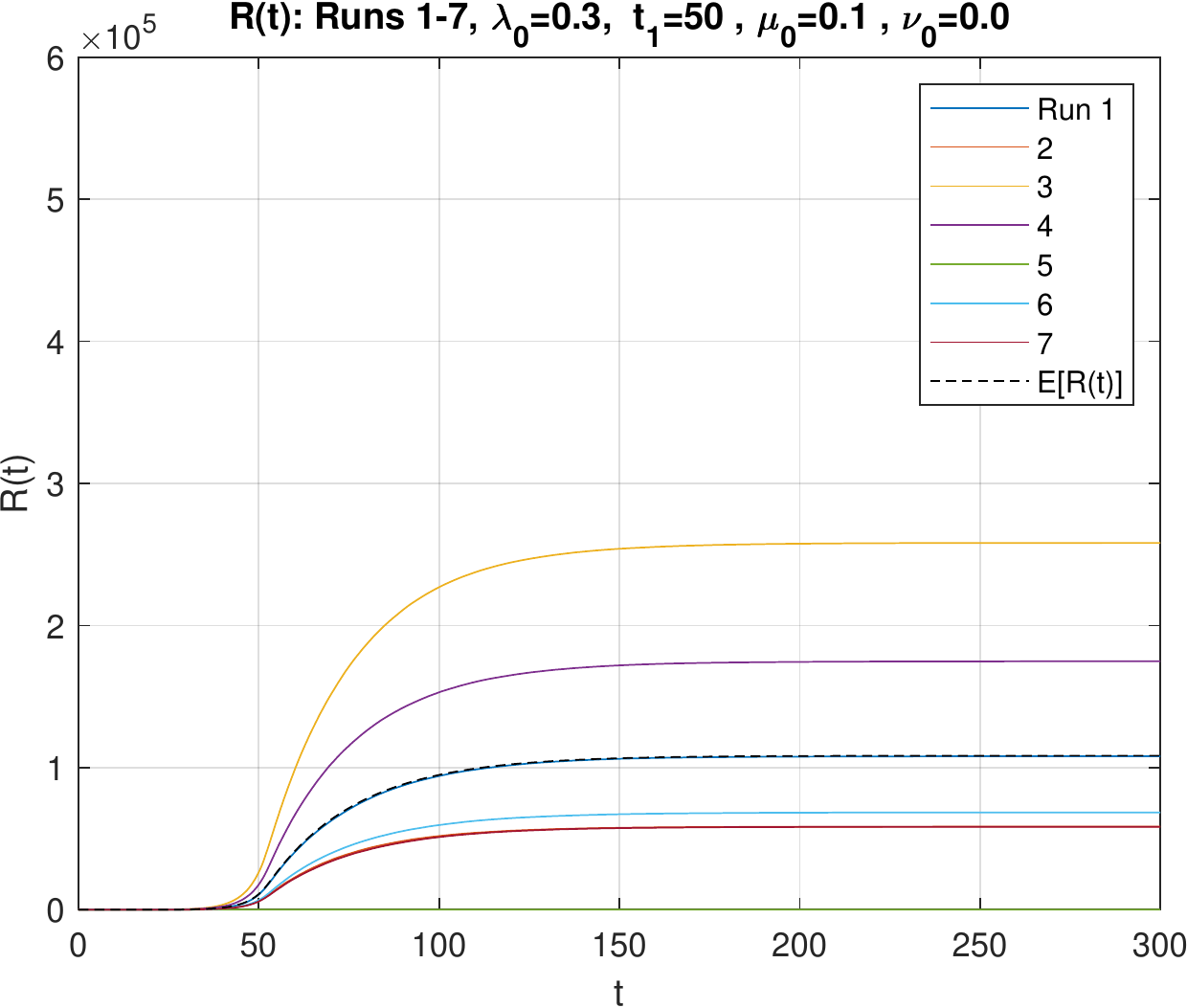}
\caption{\sf Runs 1-7: $R(t)$.}
\label{fig:BD-R(t)_Runs_1-7-I_0=1}
\end{minipage}
\hspace{0.5cm}
\begin{minipage}[t]{0.30\textwidth}
\centering
\includegraphics[width=\textwidth]{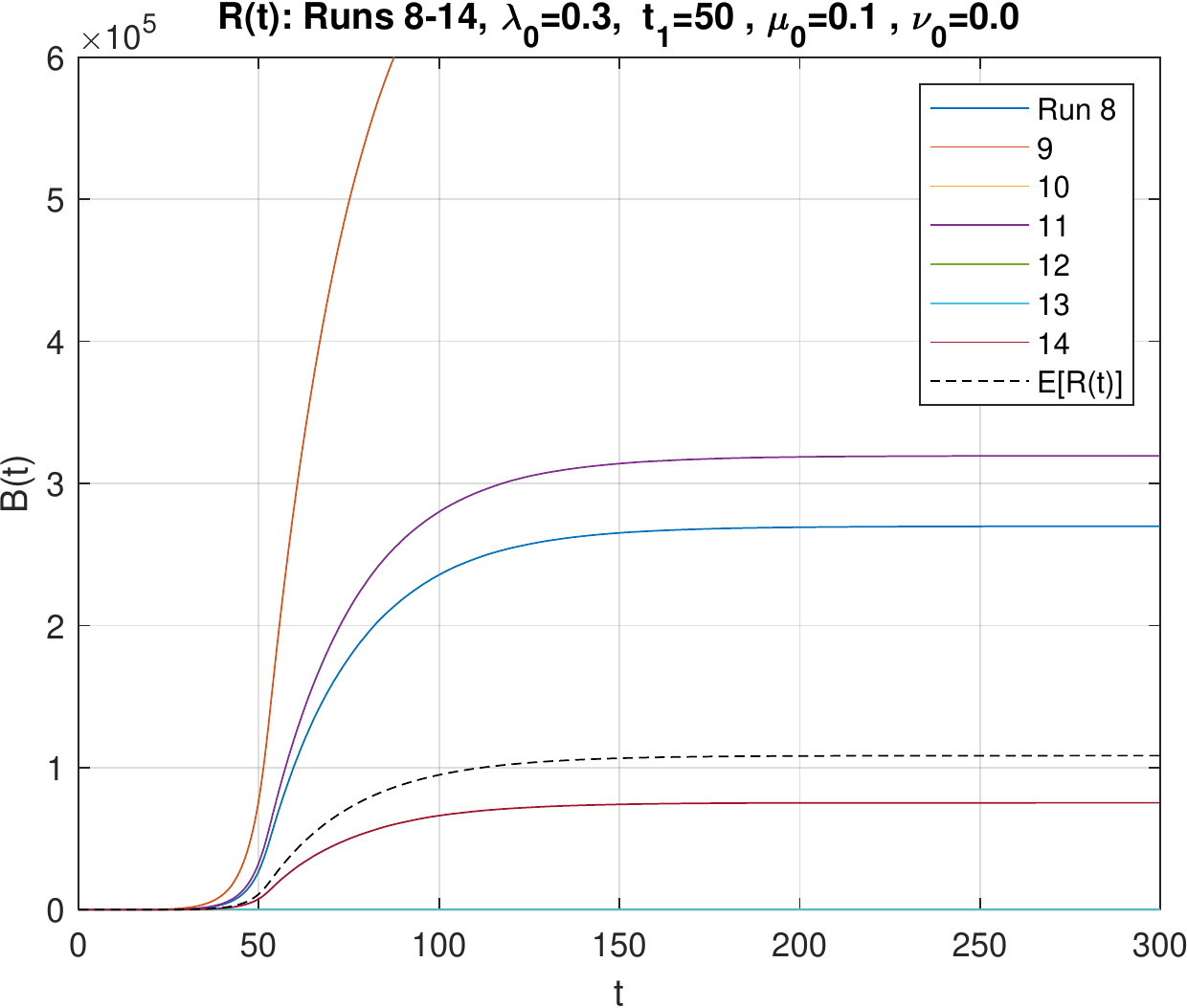}
\caption{\sf Runs 8-14: $R(t)$.}
\label{fig:BD-R(t)_Runs_8-14-I_0=1}
\end{minipage}
\hspace{0.5cm}
\begin{minipage}[t]{0.30\textwidth}
\centering
\includegraphics[width=\textwidth]{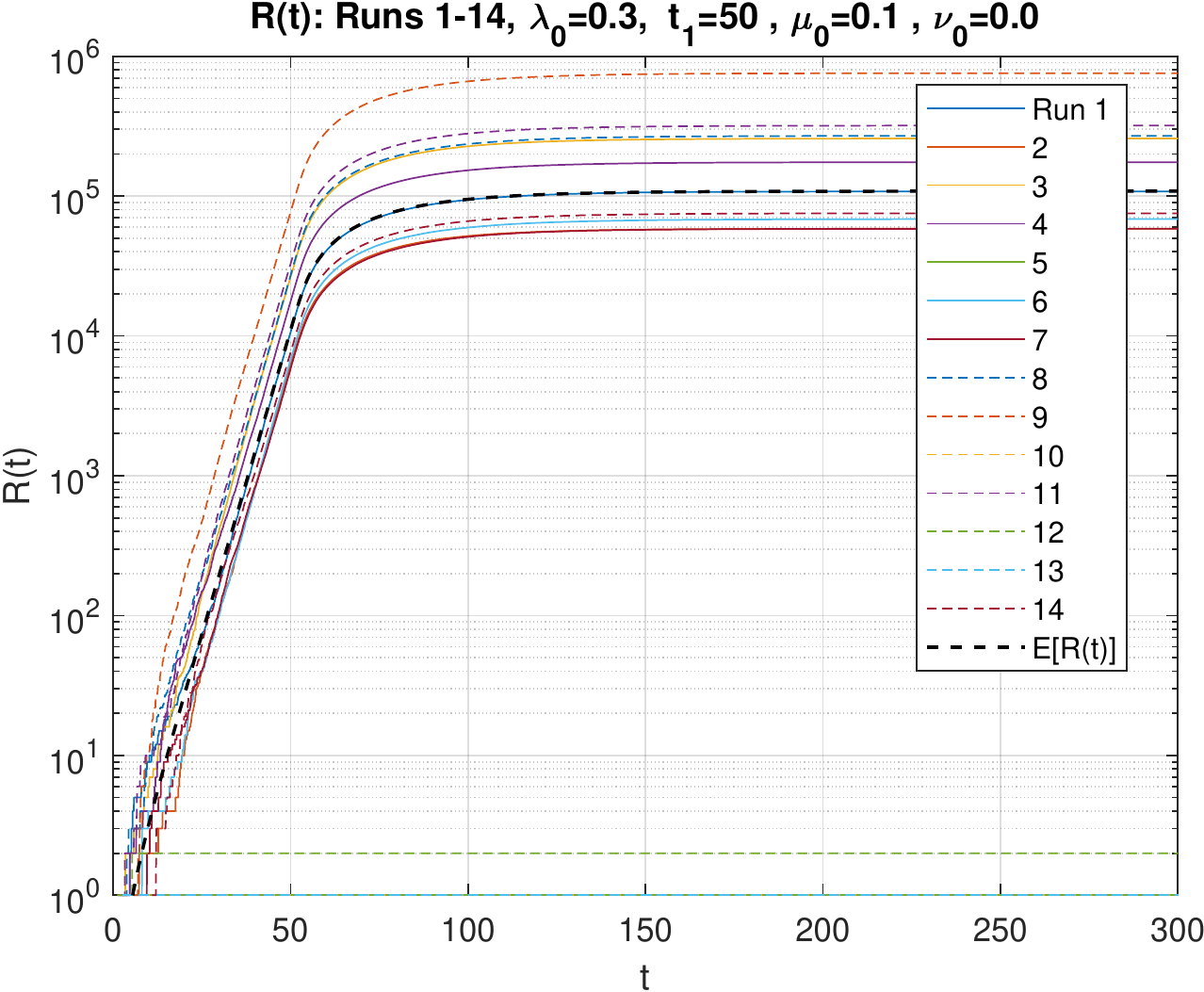}
\caption{\sf Runs 1-14: Semi-log plot of $R(t)$.}
\label{fig:BD-Semilog-R(t)_Runs_1-14-I_0=1}
\end{minipage}
\end{figure}
\begin{figure}[thb]
\begin{minipage}[t]{0.30\textwidth}
\centering
\includegraphics[width=\textwidth]{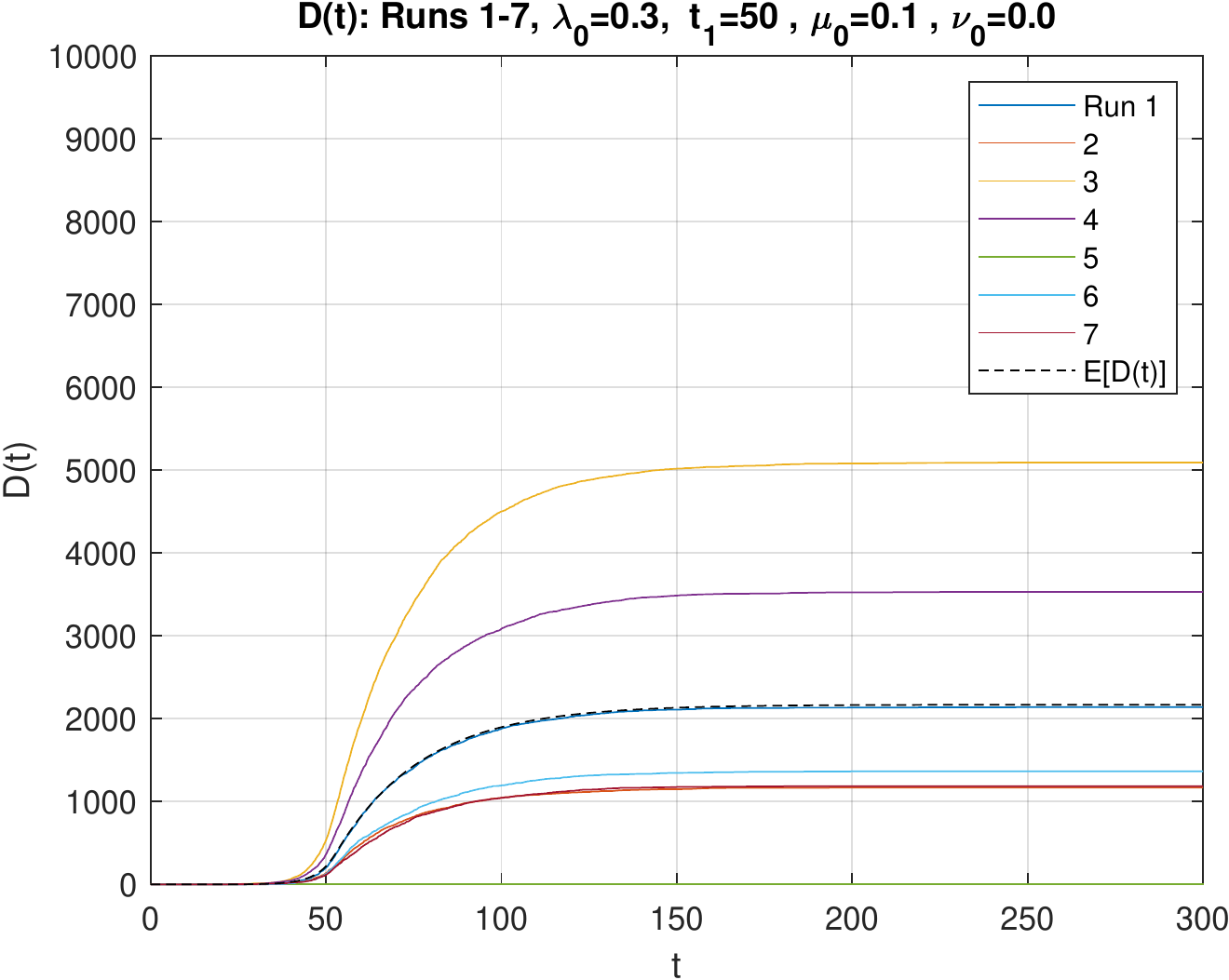}
\caption{\sf Runs 1-7: $D(t)$, the cumulative number of deaths.}
\label{fig:BD-D(t)_Runs_1-7-I_0=1}
\end{minipage}
\hspace{0.5cm}
\begin{minipage}[t]{0.30\textwidth}
\centering
\includegraphics[width=\textwidth]{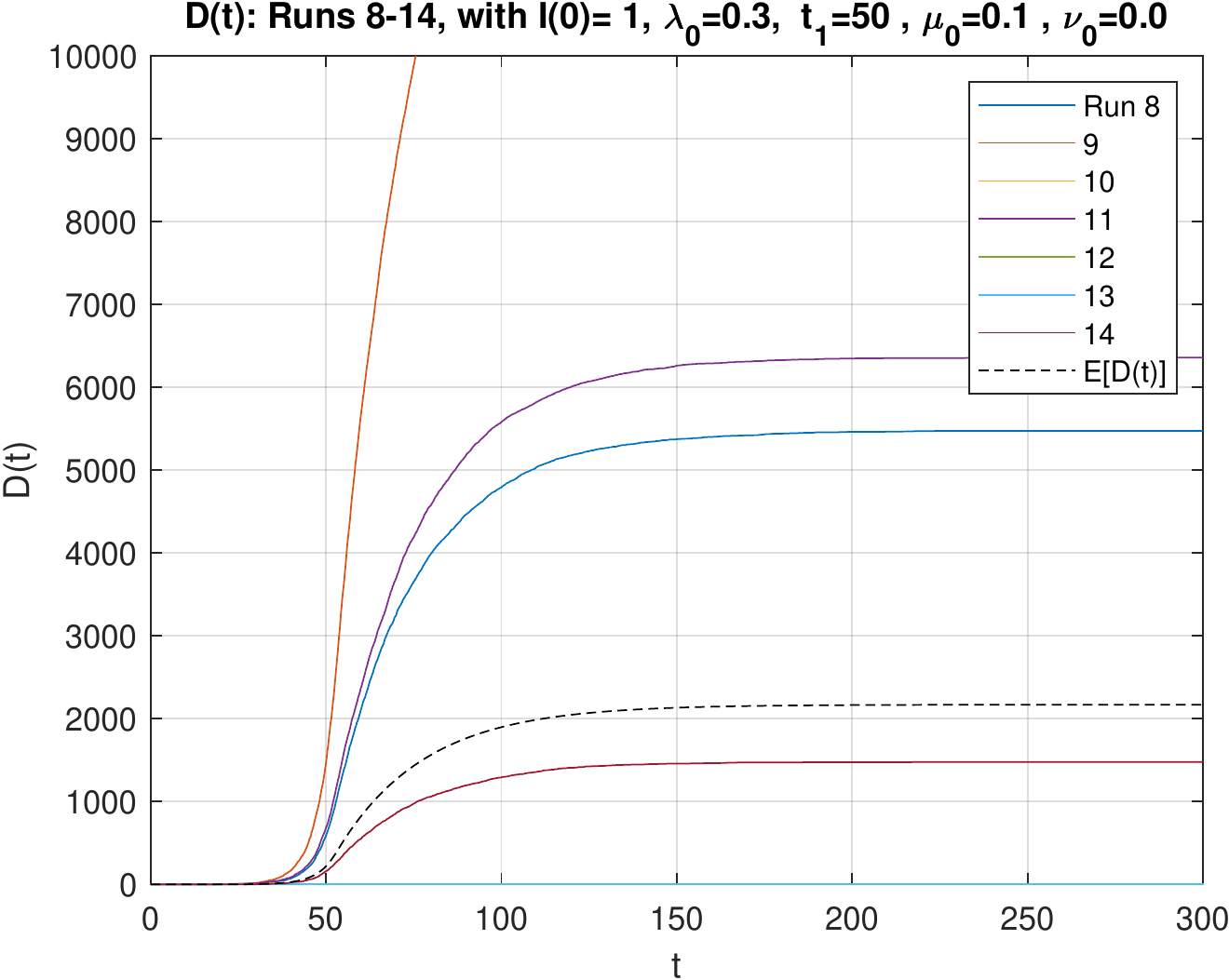}
\caption{\sf Runs 8-14: $D(t)$.}
\label{fig:BD-D(t)_Runs_8-14-I_0=1}
\end{minipage}
\hspace{0.5cm}
\begin{minipage}[t]{0.30\textwidth}
\centering
\includegraphics[width=\textwidth]{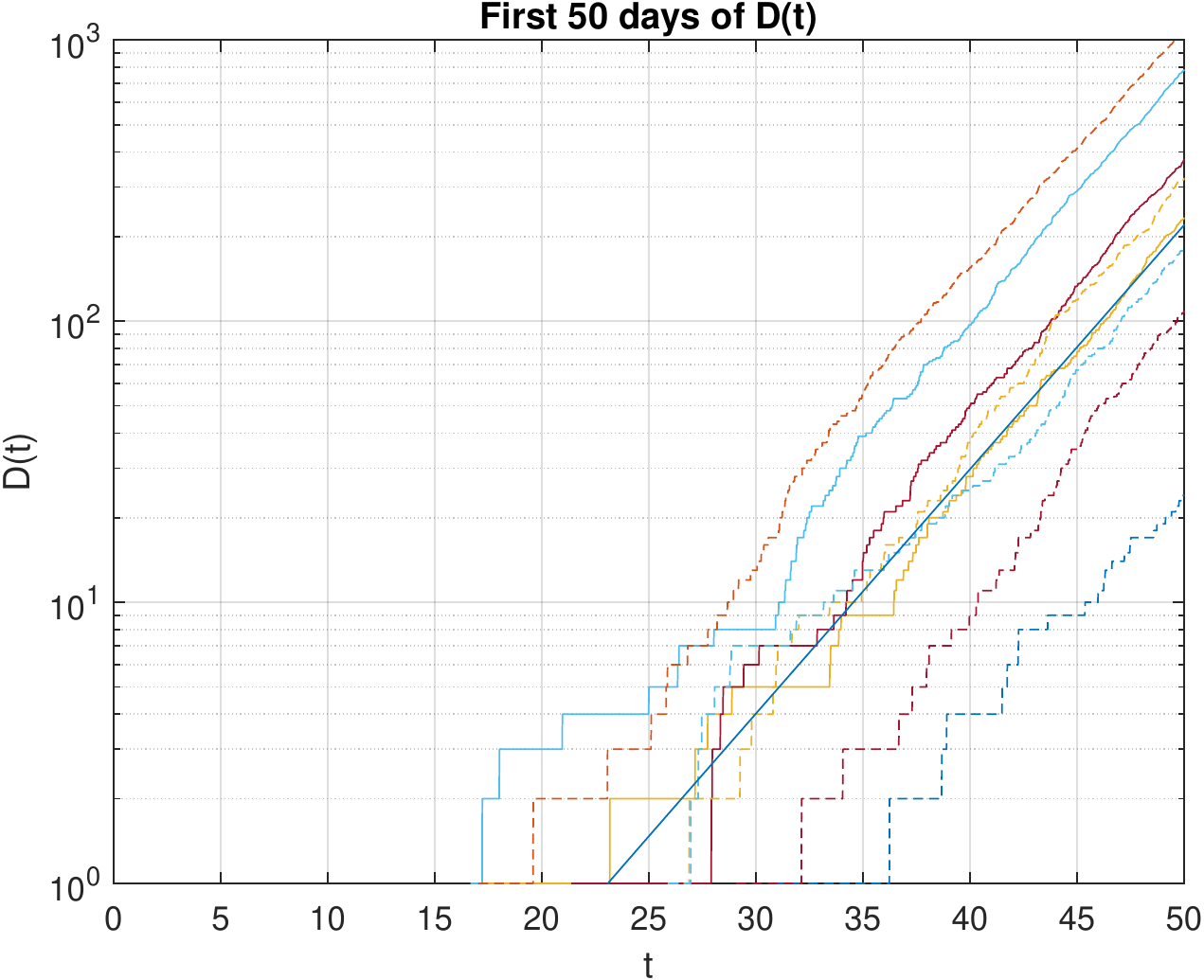}
\caption{\sf First 50 days of $D(t)$, Runs 1-14.}
\label{fig:First_50_days_D(t)}
\end{minipage}
\end{figure}

\clearpage

\section{Simulating a Time-Nonhomogeneous BDI Process\\
with \boldmath{$I_0=0$}, and \boldmath{$\nu(t)=r\lambda(t)$}}\label{sec:nonhomo-BDI-simulation}

We now report on simulation experiments of a time-varying BDI process.
The BDI process with the initial population $I_0$, $I_{BDI:I_0}(t)$, can be decomposed into  $I_{BD:I_0}(t)$ and $I_{ID:0}(t)$ as given in (\ref{BDI-two-processes}) of Proposition \ref{prop:nonhomo-BDI}.  The first component was studied in the previous section \footnote{A full analysis and simulation of a BD process with $I_0>1$ will be deferred to \cite{kobayashi:2021bb}.}, so we set $I_0=0$ here to focus on the behavior of $I_{BDI:0}(t)=I_{ID:0}(t)$.  

As Corollary \ref{coro:nonhomo-BDI-NBD} states, if we choose the immigrants' arrival rate $\nu(t)=r \lambda(t)$ with some positive real number $r$, then the PGF reduces to that of an NBD (negative binomial distributed) process with parameters $(r, \beta(t))$:
\begin{align}
G_{BDI:0}(z,t)= G_{ID:0}(z,t)=\left(\frac{1-\beta(t)}{1-\beta(t)z}\right)^r. \label{PGF-BDI:0}
\end{align}

We adopt the same $\lambda(t)$ and $\mu(t)$ as assumed in the numerical analysis in Section 1, and in the preceding section on simulating the BD process.  We consider again the case of $d=5$ [days].  The immigrants' arrival rate $\nu(t)=r\lambda(t)$ is assumed, with $r=\frac{\nu_0}{\lambda_0}=\frac{2}{3}$ (see Figure \ref{fig:nu(t)-raised-cosine} and \ref{fig:Detailed-nu_t-bent-at-t=50}.).

We shall discuss cases $r(t)\neq r$ in a subsequent report \cite{kobayashi:2021bb}, together with a comprehensive analysis.

\begin{figure}[thb]
\begin{minipage}[t]{0.45\textwidth}
\centering
\includegraphics[width=\textwidth]{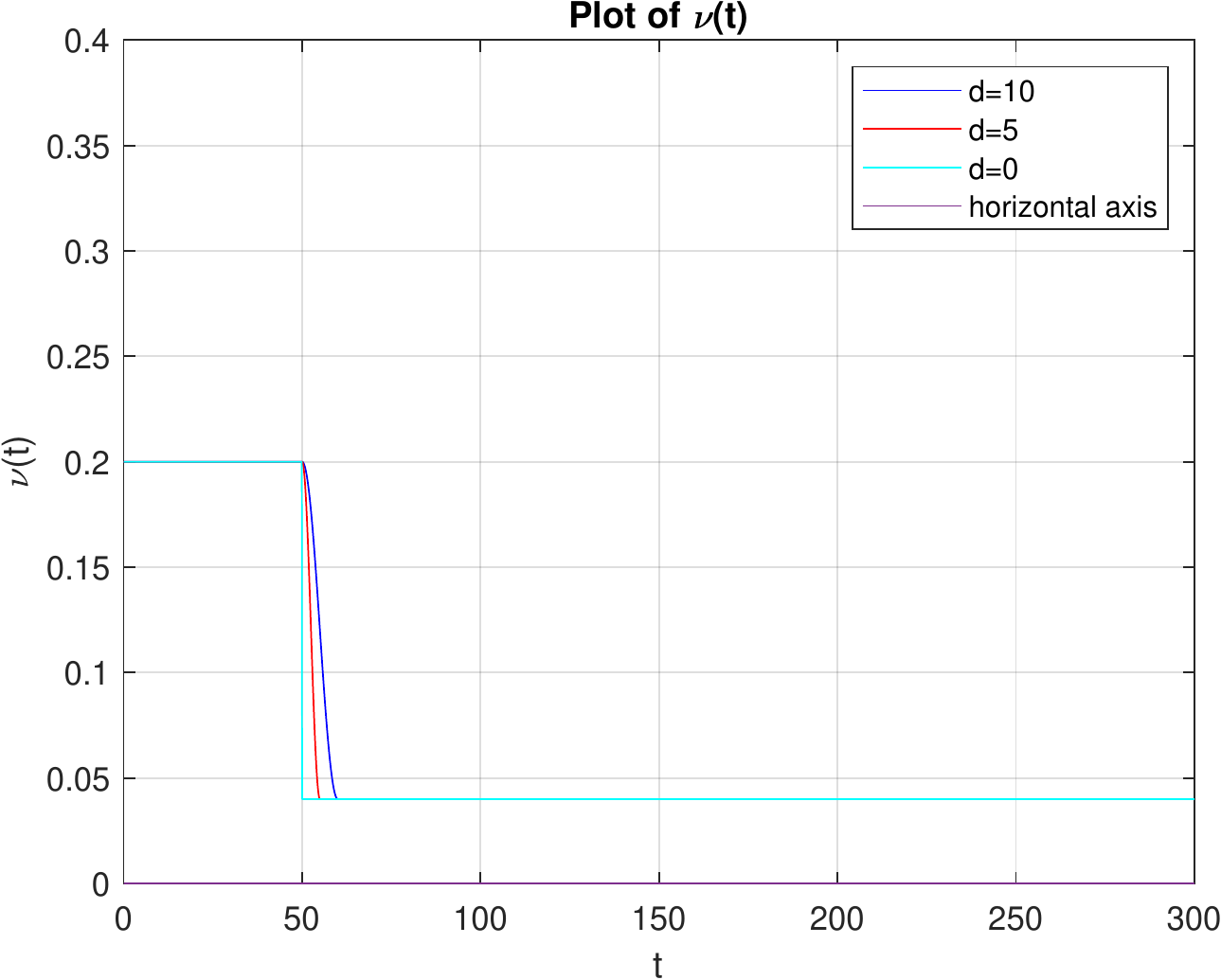}
\caption{\sf The function $\nu(t)$ makes a transition from $\nu_0=0.2$ down to $\nu_1=0.04$ over the interval
 $[50, 55)$ (i.e.,the red curve with $d=5$ [days] is adopted in the simulation.)}
\label{fig:nu(t)-raised-cosine}
\end{minipage} 
\qquad
\begin{minipage}[t]{0.45\textwidth}
\centering
\includegraphics[width=\textwidth]{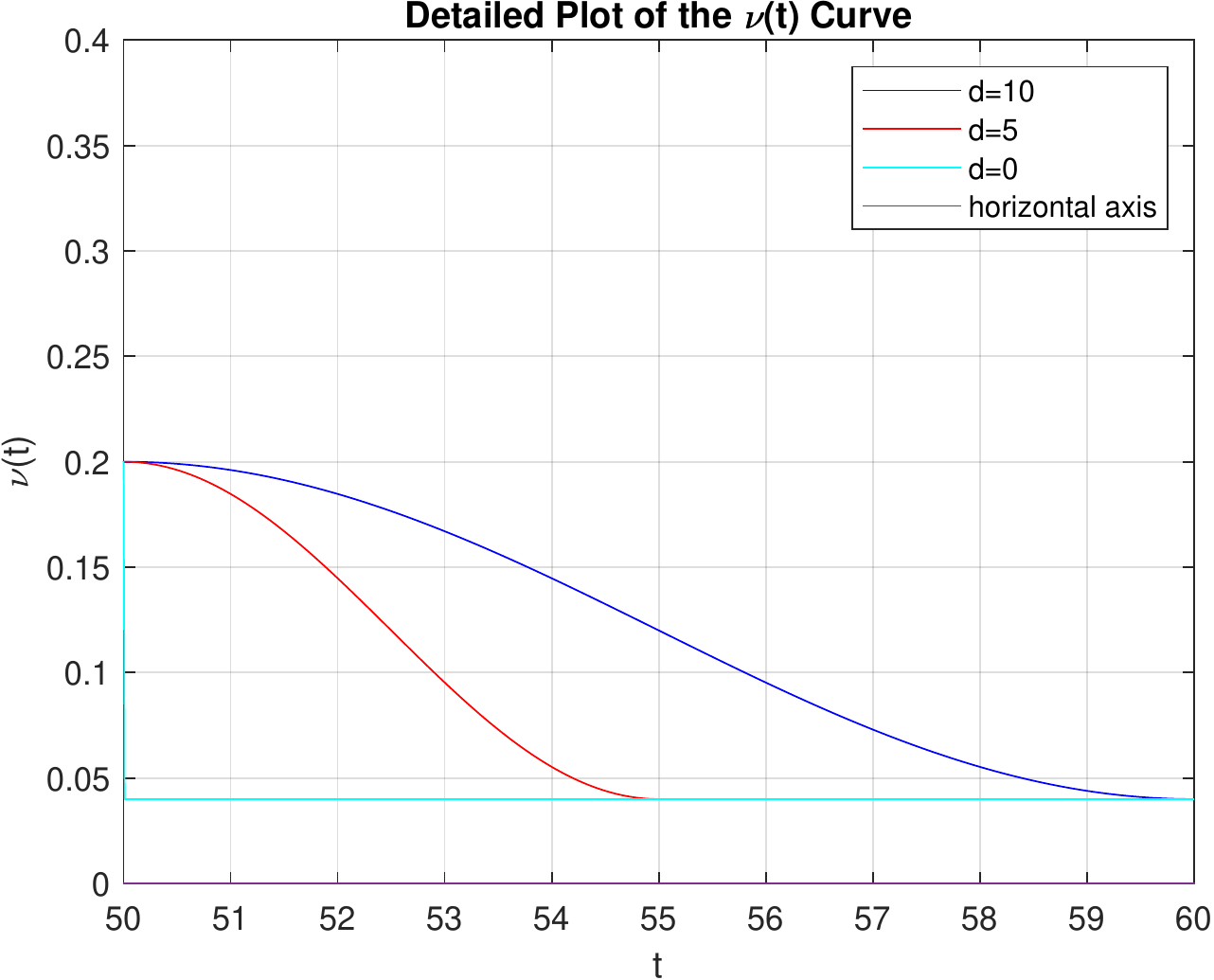}
\caption{\sf An expanded view of the transition, where the function $\nu(t)$ takes a smooth curve represented one half cycle of a cosine function, raised up properly .} 
\label{fig:Detailed-nu_t-bent-at-t=50}
\end{minipage}
\end{figure}

\subsection{The Processes $A(t)$, and $I_{BDI:0}(t)$} \label{subsec:I_BDI:0}
We provide below major findings of the simulation experiments of the $I_{BDI:0}(t)$ process.

\begin{enumerate}
\item  Since we assume zero population at $t=0$, i.e., $I_0=0$, all those who are present at time  $t$ are either the immigrants who have arrived prior to $t$ or their descendants who were born and alive. By comparing  Figures \ref{fig:A(t)_Runs_1-7} and \ref{fig:A(t)_Runs_8-14}, we see considerable differences in the arrival pattern of the first 7 runs and the second 7 runs.
Five out of the first 7 runs have their arrivals at faster pace than the average rate $\oA(t)$, whereas only three runs in the second group have their arrivals at faster rates  than  $\oA(t)$.  

\item  By looking at Figures \ref{fig:BDI_0(t)_Runs_1-7} and \ref{fig:BDI_0(t)_Runs_8-14}, we notice considerable differences in the behaviors of the process $I(t)$ between the first 7 and second 7 runs. In the first group, only Run 5 (shown in green) exceeds the $\oI(t)$, whereas in the second group both Run 10 (yellow) and Run 9 (blue) far exceed $\oI(t)$,  whereas Run 14 (dark red) and Run 13 (light blue) are hardly visible in Figure \ref{fig:BDI_0(t)_Runs_1-7}.  In the semi-log plots of Figure \ref{fig:Semilog_BDI_0(t)_Runs_8-14}, however, we clearly see these runs. As we already remarked in Part I \& II, the NBD with $r=\frac{\nu}{\lambda}<1$ exhibits a wider spread than the Poisson or other distributions we normally encounter.\footnote{In the running example, the first 50 days of this time-varying BDI process is, statistically speaking, exactly the same as the time-homogeneous case we studied in Parts I \& and II.}

\item At first glance, there does not seem to exist a significant relation between the $A(t)$ process and $I(t)$ in these simulation runs.  Run 5 (green) ranks at the bottom in Figure \ref{fig:A(t)_Runs_1-7}, whereas in terms of the $I(t)$ this run ranks nearly at the top: see Figures \ref{fig:BDI_0(t)_Runs_1-7}, \ref{fig:Semilog_BDI_0(t)_Runs_1-7} and \ref{fig:Semilog_BDI_0(t)_Runs_1-14}. By examining carefully Figure \ref{fig:BD-Semilog_BDI_0(t)_Runs_1-14_25-days} and reexamining Figure \ref{fig:A(t)_Runs_1-7}, however, we notice that Run 5 climbs up immediately after $t=0$.  There are a few quick arrivals and births and few deaths in the initial period, all of which helped this process grow steadily fast.  Run 9 (yellow) in the second group also exhibits a fast build-up in the initial period with quick arrivals and births, and few deaths deaths in the ten days, as you can see in Figure 
\ref{fig:BD-Semilog_BDI_0(t)_Runs_1-14_25-days}.

\item An opposite example is Run 1 (blue) that has the largest cumulative arrivals after around $t=50$. However, it is near the bottom in terms of $I(t)$.  By closely examining Figure \ref{fig:BD-Semilog_BDI_0(t)_Runs_1-14_25-days}, Run 1 does not have many arrivals or births in the initial period, seems plagued by deaths. It is not until 18th day that this process begins to grow beyond two.  A slow-start certainly hurts in building up the population.

\item Once the process $I(t)$ grows beyond about $1\sim 2$ hundreds, the law of large number seems to set in, making the future path of the process more predictable.  The same seems to apply when the $I(t)$ decreases when $a(t)=\lambda(t)-\mu(t) <0$.  The declining behavior seems predictable until $I(t)$ decreases less than $1\sim 2$ hundred
level.  In terms of $s(t)=\int_0^t a(u)\,du$, these numbers approximately translate to $s(t)\sim 4.6-5.2$ from the approximate formula $\oI(t)=\frac{\nu_0}{\lambda_0-\mu_0)}e^{s(t)}$.  For more on this rule of thumb, refer to the discussion in the next section.

\item In Figures \ref{fig:BDI-I_new[t]-Runs_1-7}, \ref{fig:BDI-I_new[t]-Runs_8-14} and \ref{fig:BDI-I_new[t]-Run_5}, present simulation results of familiar statistics, i.e, the daily statistics of newly infected from day 0 to day $t$.
The curves agree with the shape shown in the analysis as presented in Figure \ref{fig:New-daily-infections} (the red curve corresponds to $d=5$).  As we discussed concerning the similar curve in the BD process presented in Figure 11 of Section 2.2 in \cite{kobayashi:2021a} there is a sharp drop on the RHS, because this function is proportional to 
$\lambda(t)I(t)$ (see (41), ibid.).  Although the simulation results confirm the shape given by the analysis, the present author has not seen such sharp drops reported in the news reported in the Covid-19 epidemics. A plausible explanation is that the dormant period from the moment when an infection takes place until any symptom should appear (or a PCR-test showing positive) varies case by case, a decrease in the number of new infections may not be as dramatic as we might expect.  
\end{enumerate}

\begin{figure}[thb]
\begin{minipage}[t]{0.50\textwidth}
\centering
\includegraphics[width=\textwidth]{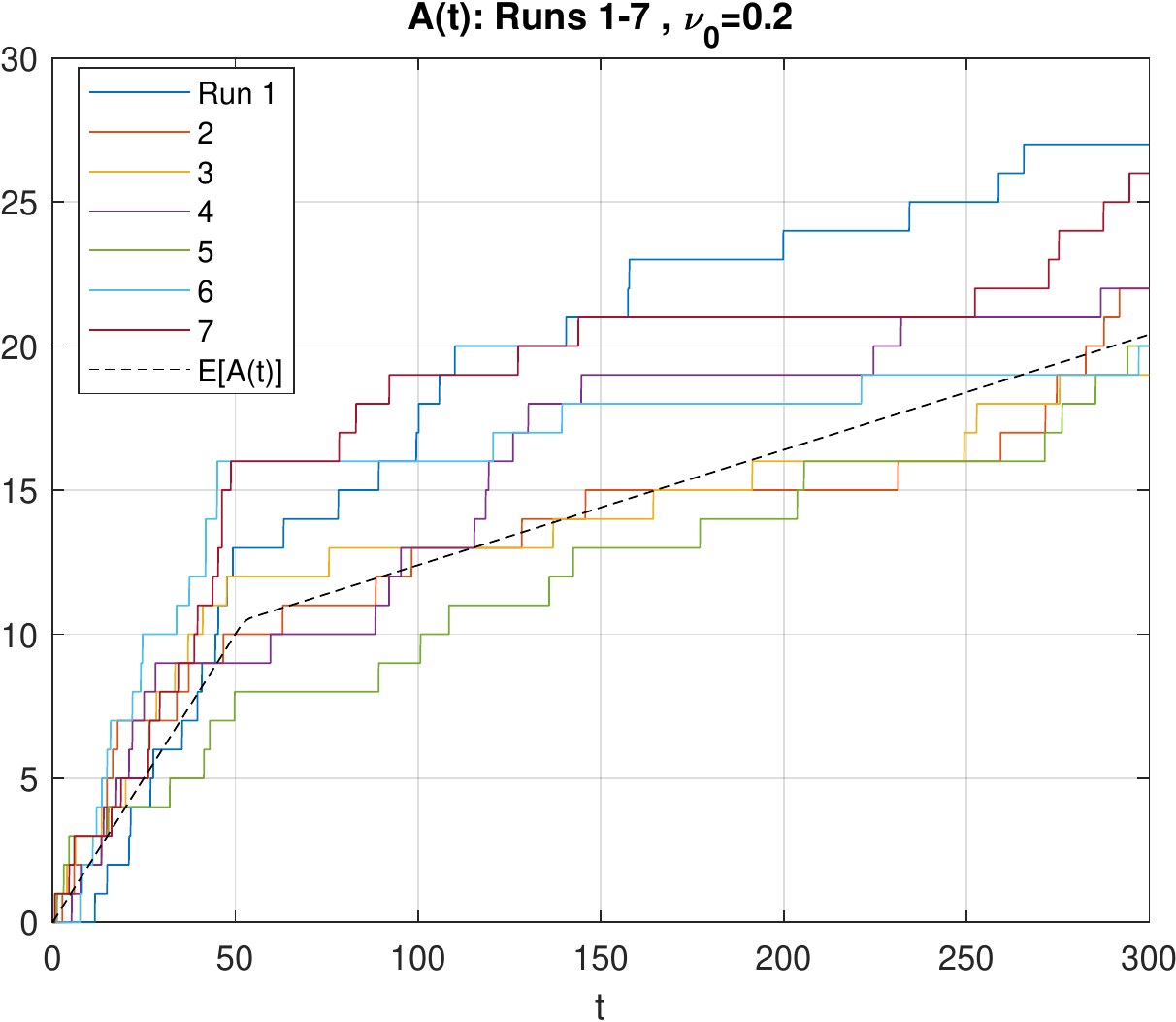}
\caption{\sf The cumulative count of arrivals $A(t)$, Runs 1-7.}
\label{fig:A(t)_Runs_1-7}
\end{minipage}
\hspace{0.5cm}
\begin{minipage}[t]{0.50\textwidth}
\centering
\includegraphics[width=\textwidth]{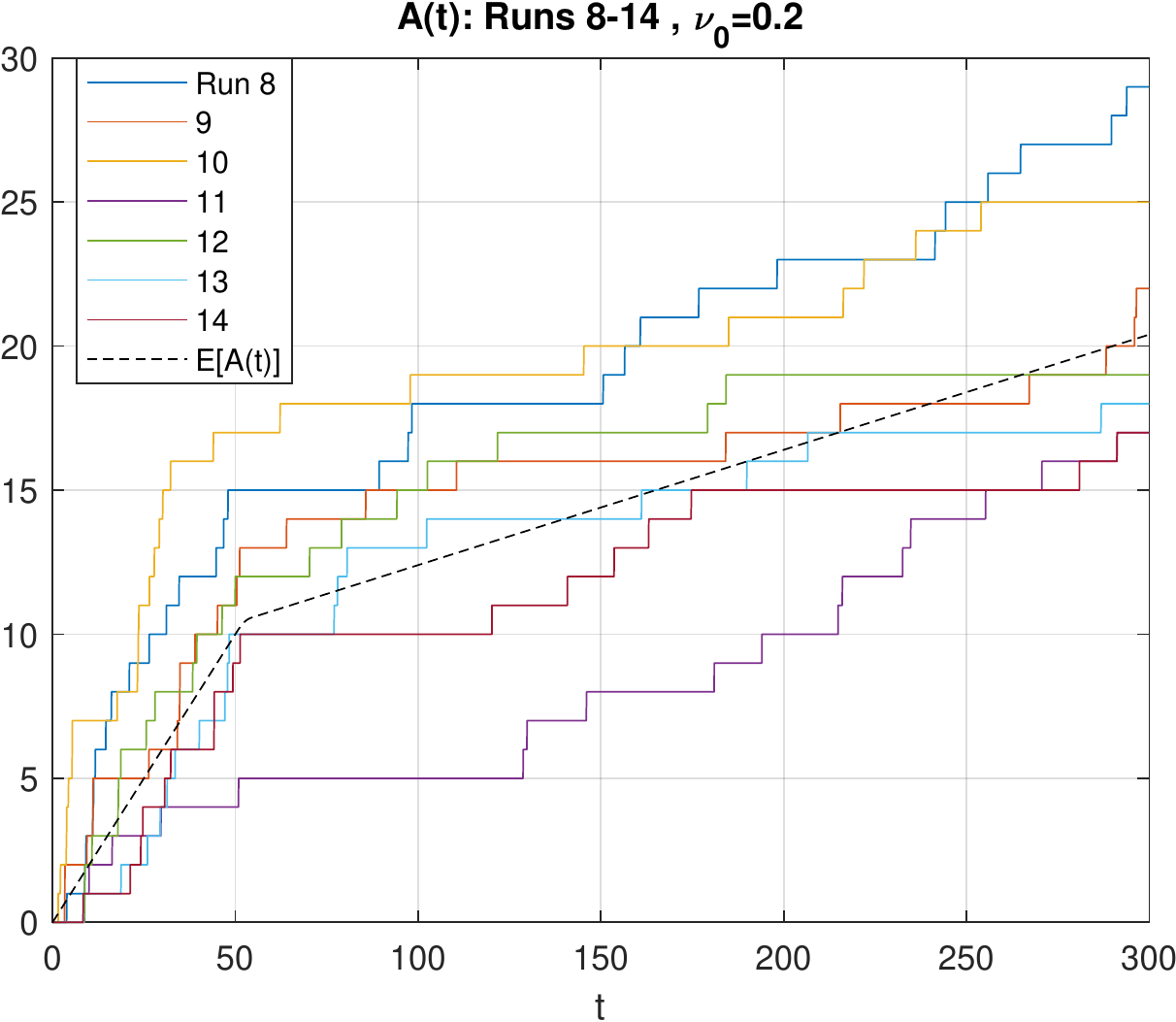}
\caption{\sf  The cumulative count of arrivals $A(t)$, Runs 8-14.}
\label{fig:A(t)_Runs_8-14}
\end{minipage}
\end{figure}
\begin{figure}[thb]
\begin{minipage}[t]{0.50\textwidth}
\centering
\includegraphics[width=\textwidth]{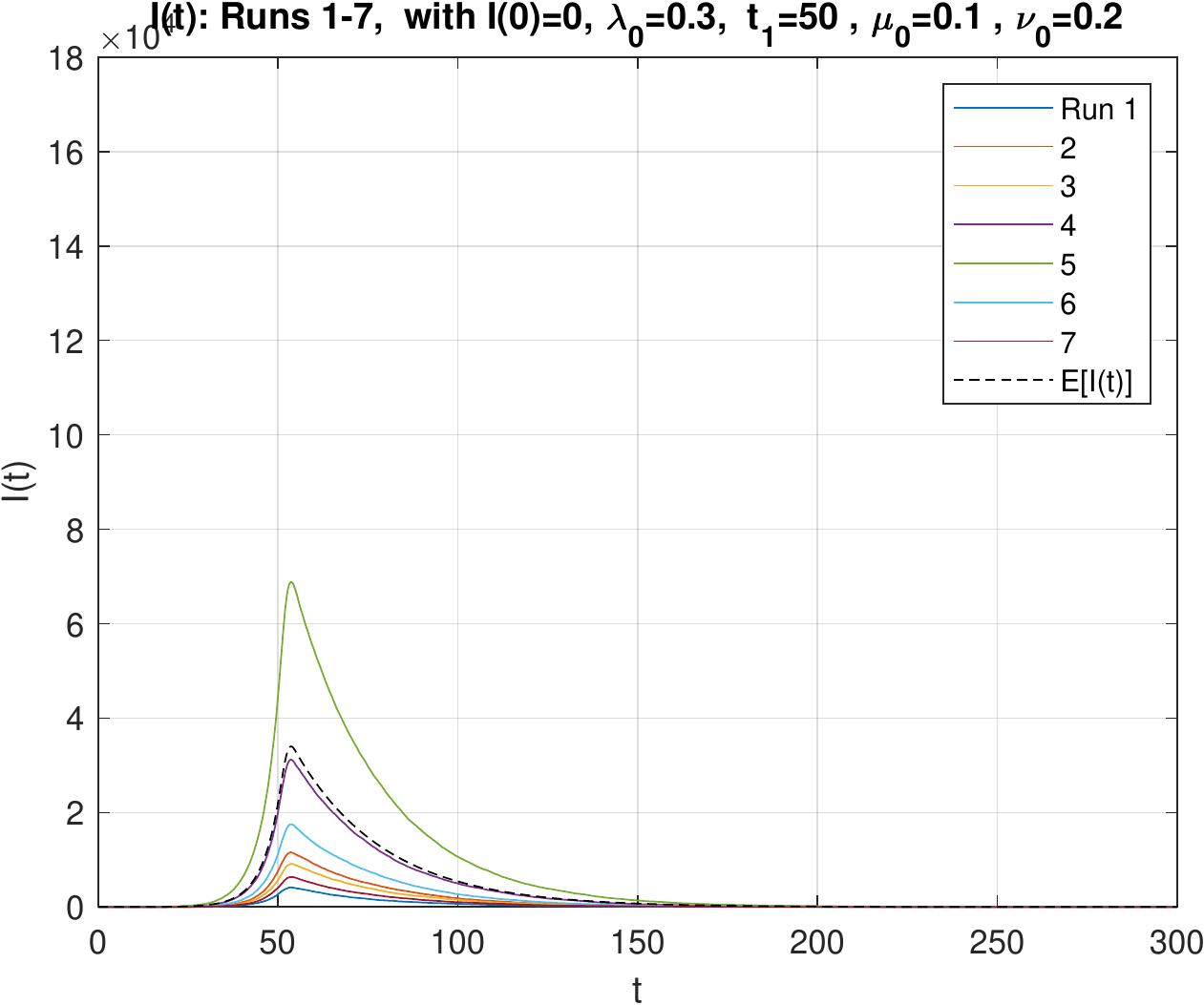}
\caption{\sf  The $I(t)$ process with $I_0=0$, Runs 1-7.}
\label{fig:BDI_0(t)_Runs_1-7}
\end{minipage}
\hspace{0.5cm}
\begin{minipage}[t]{0.50\textwidth}
\centering
\includegraphics[width=\textwidth]{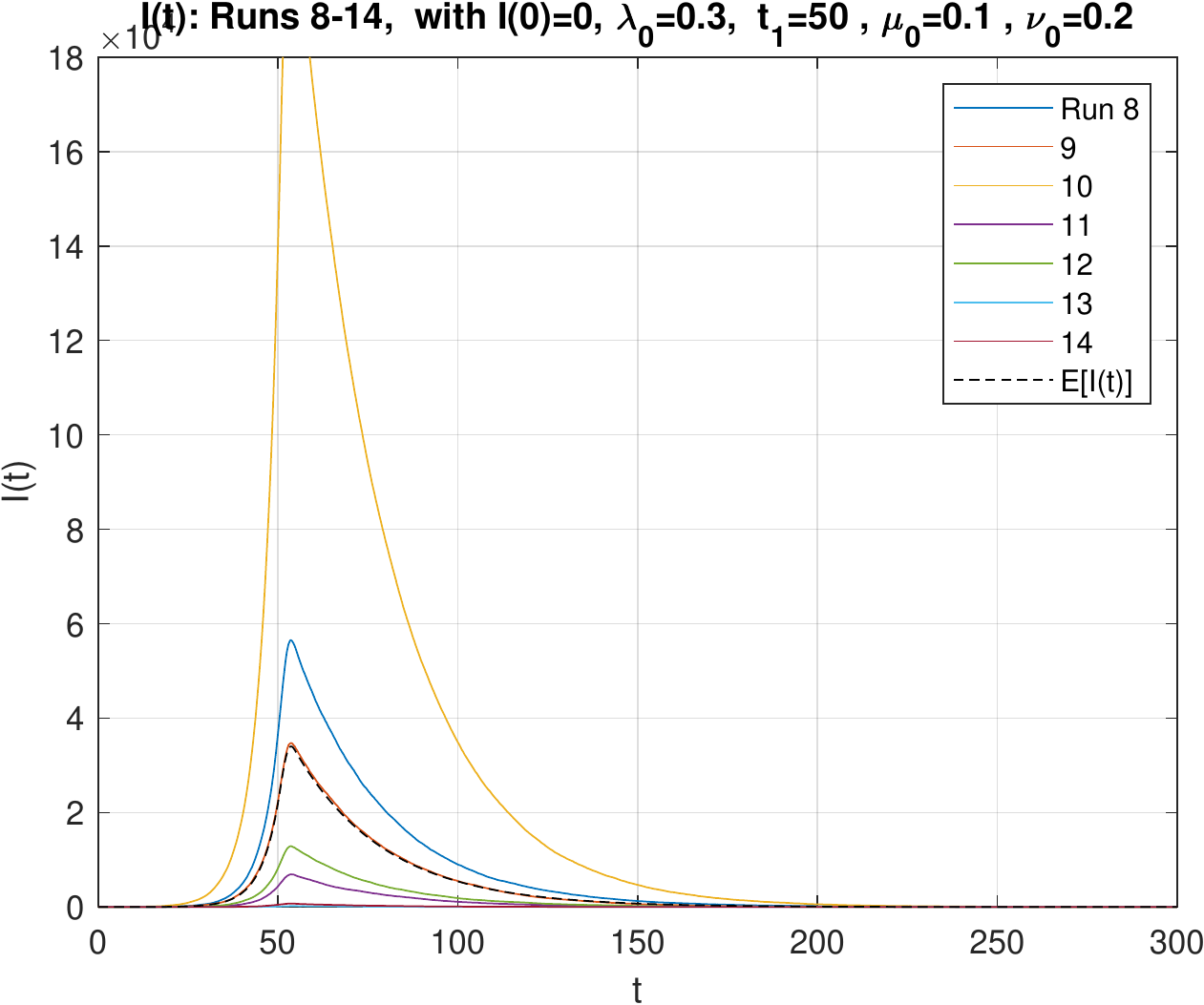}
\caption{\sf  The $I(t)$  process with $I_0=0$, Runs 8-14.}
\label{fig:BDI_0(t)_Runs_8-14}
\end{minipage}
\end{figure}
\begin{figure}[hbt]
\begin{minipage}[h]{0.50\textwidth}
\centering
\includegraphics[width=\textwidth]{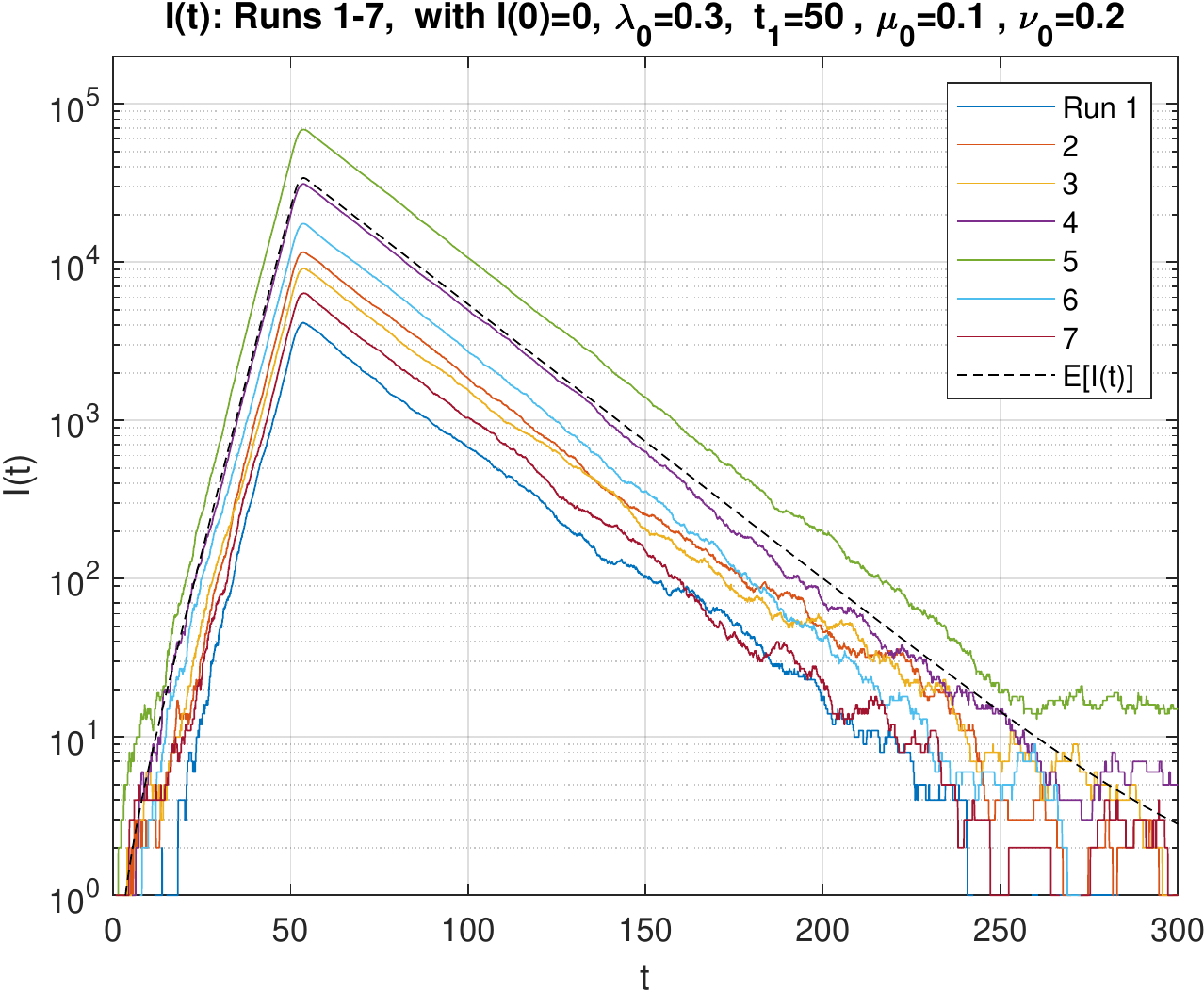}
\caption{\sf Semi-log plot of the $I(t)$ process, Runs 1-7.}
\label{fig:Semilog_BDI_0(t)_Runs_1-7}
\end{minipage}
\hspace{0.5cm}
\begin{minipage}[h]{0.50\textwidth}
\centering
\includegraphics[width=\textwidth]{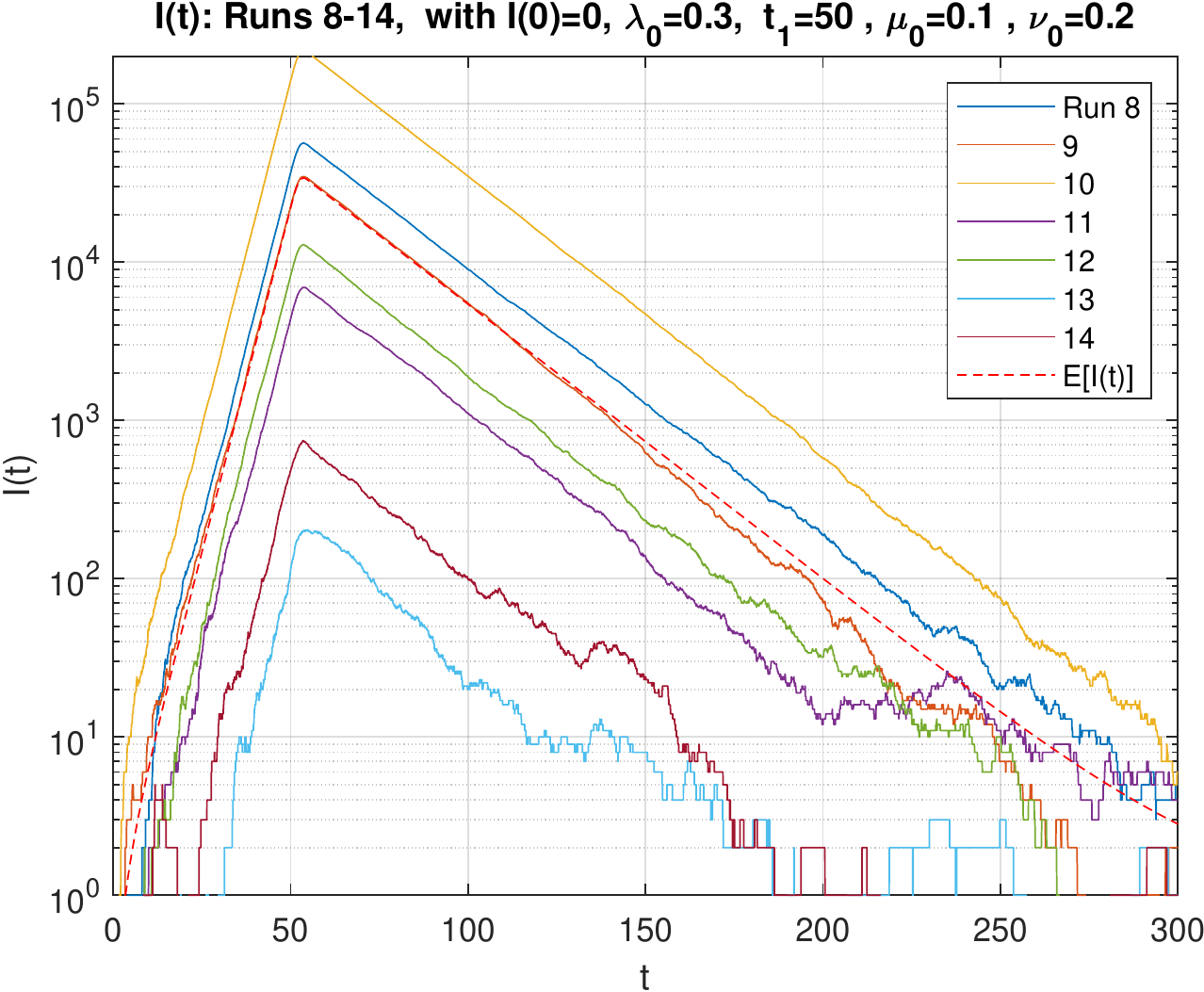}
\caption{\sf Semi-log plot of the $I(t)$ process, Runs 8-14.}
\label{fig:Semilog_BDI_0(t)_Runs_8-14}
\end{minipage}
\hspace{0.5cm}
\end{figure}
\begin{figure}[thb]
\begin{minipage}[t]{0.54\textwidth}
\centering
\includegraphics[width=\textwidth]{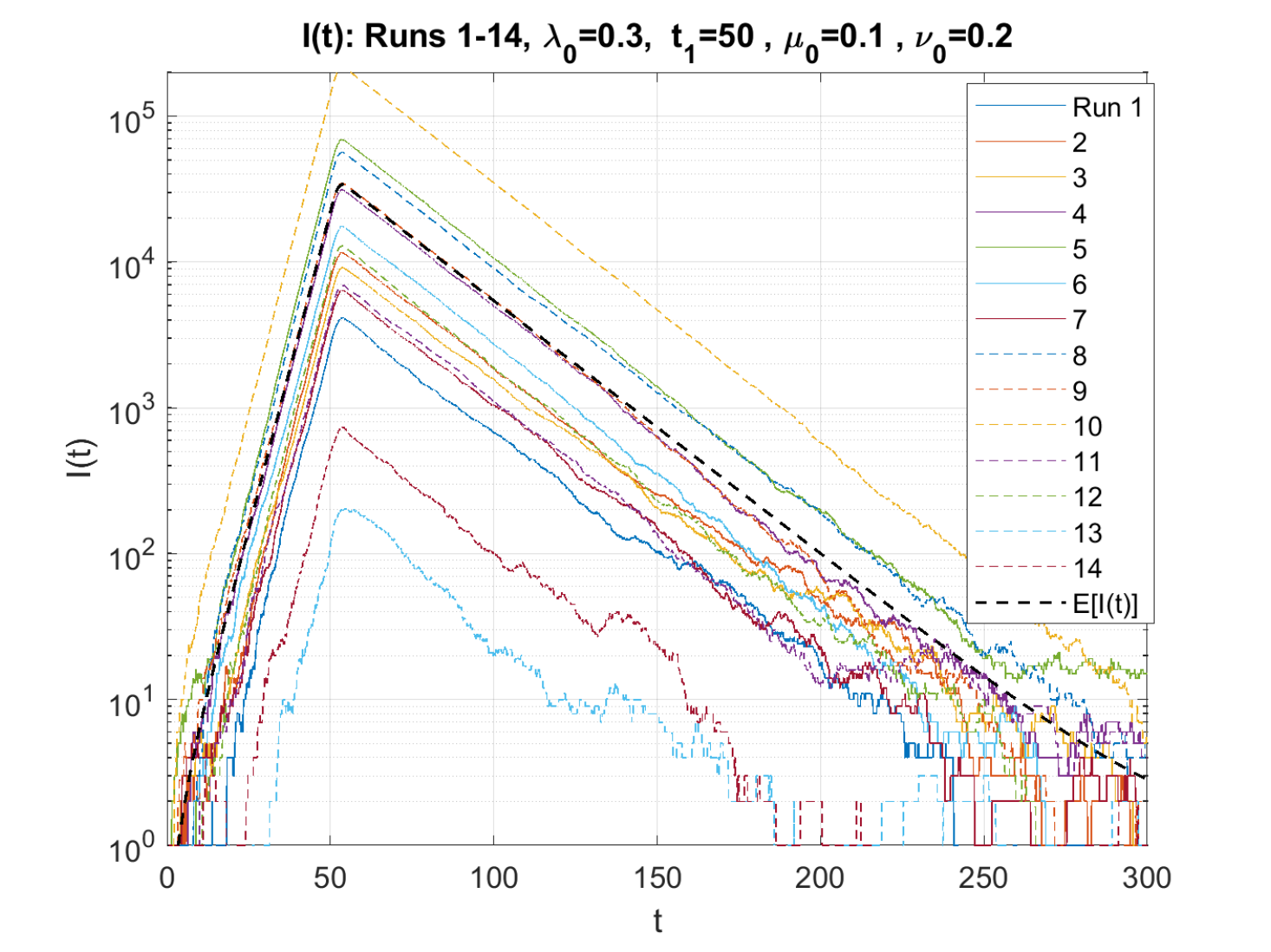}
\caption{\sf Semi-log plot of the $I(t)$ process, Runs 1-14.}
\label{fig:Semilog_BDI_0(t)_Runs_1-14}
\end{minipage}
\hspace{0.5cm}
\begin{minipage}[t]{0.50\textwidth}
\centering
\includegraphics[width=\textwidth]{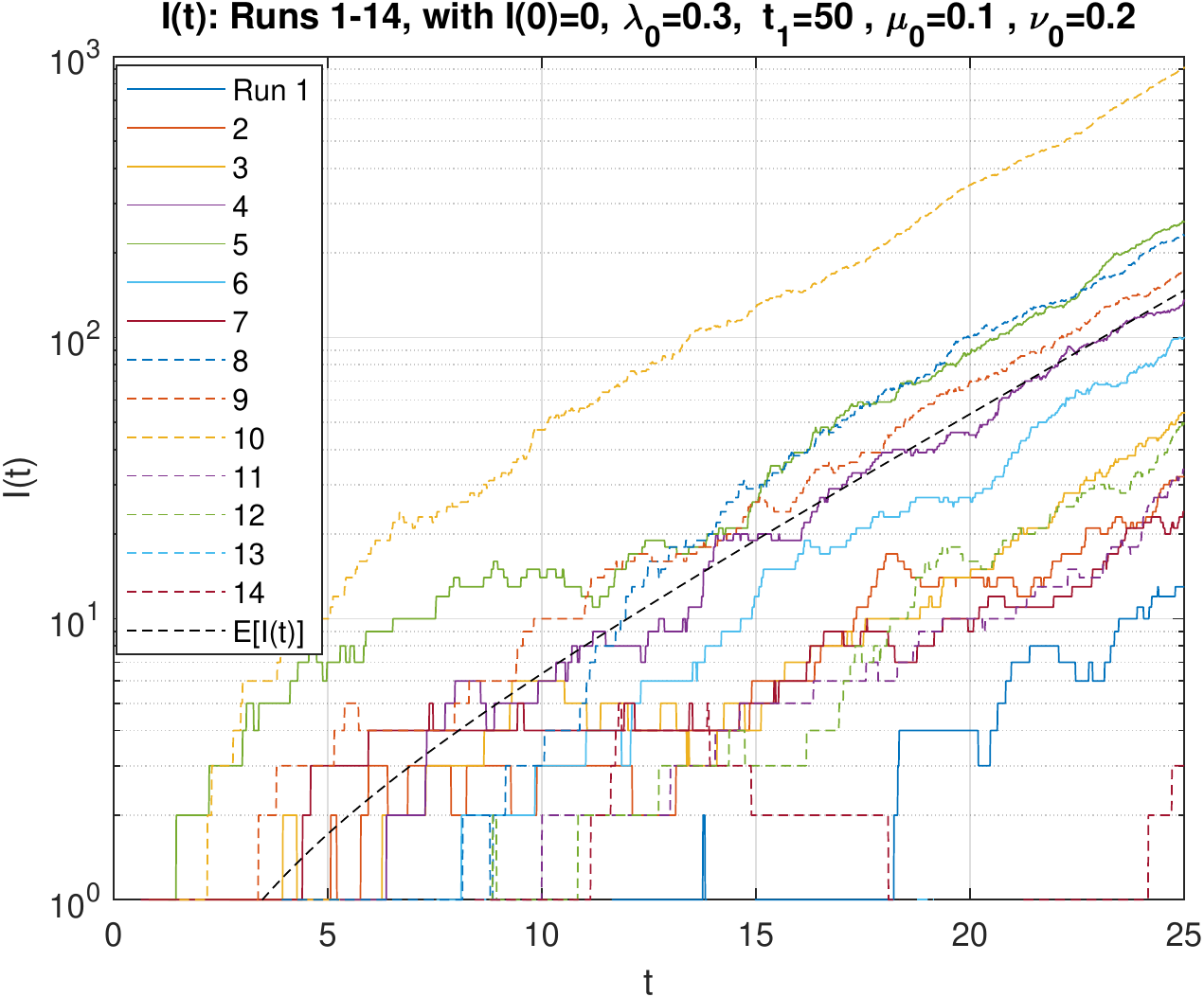}
\caption{\sf Initial 25 days of the $I(t)$ process, Runs 1-14.}
\label{fig:BD-Semilog_BDI_0(t)_Runs_1-14_25-days}
\end{minipage}
\end{figure}

\begin{figure}[thb]
\begin{minipage}[t]{0.50\textwidth}
\centering
\includegraphics[width=\textwidth]{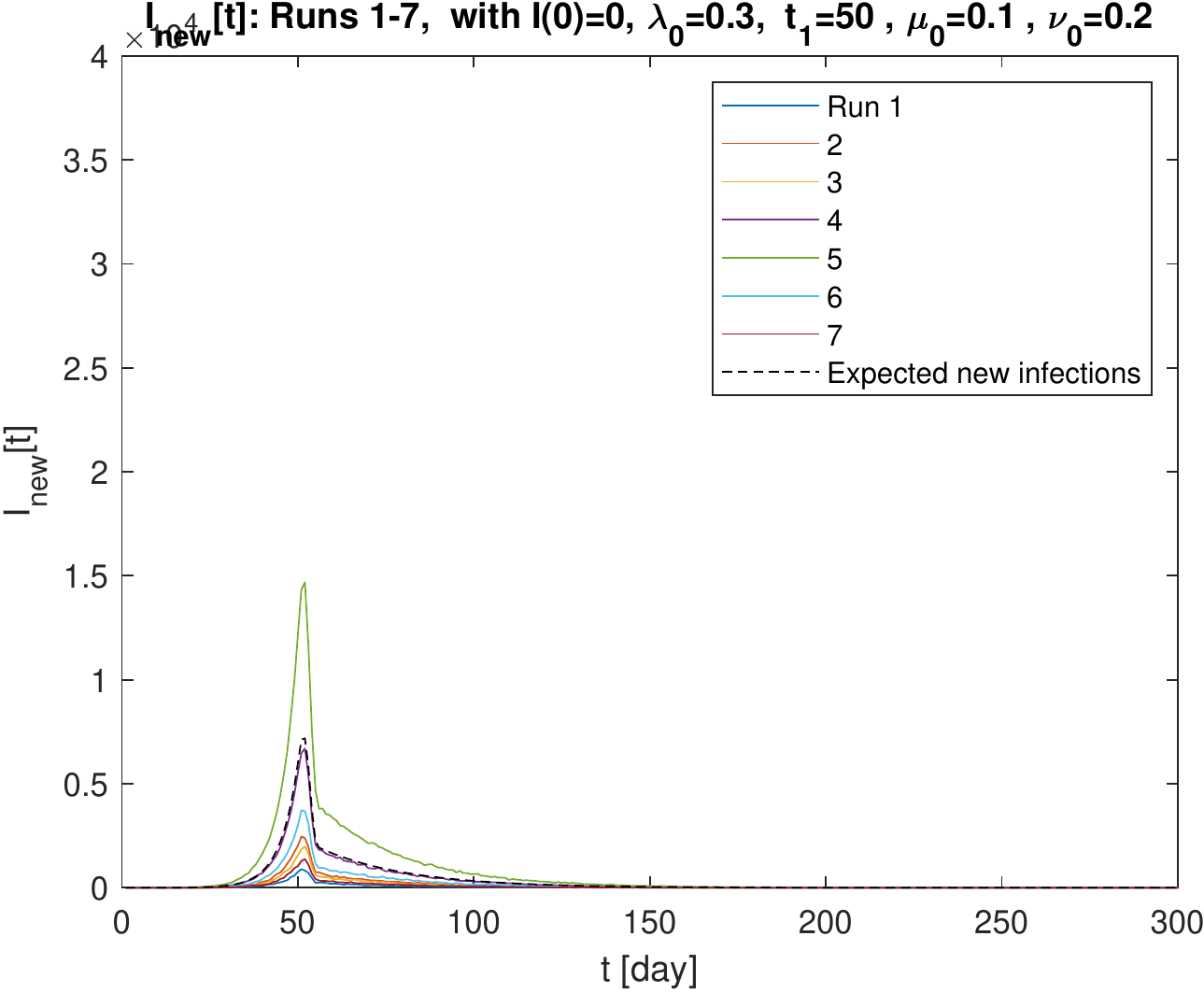}
\caption{\sf New daily infections $I_{new}[t]$, Runs 1-7.}
\label{fig:BDI-I_new[t]-Runs_1-7}
\end{minipage}
\hspace{0.5cm}
\begin{minipage}[t]{0.50\textwidth}
\centering
\includegraphics[width=\textwidth]{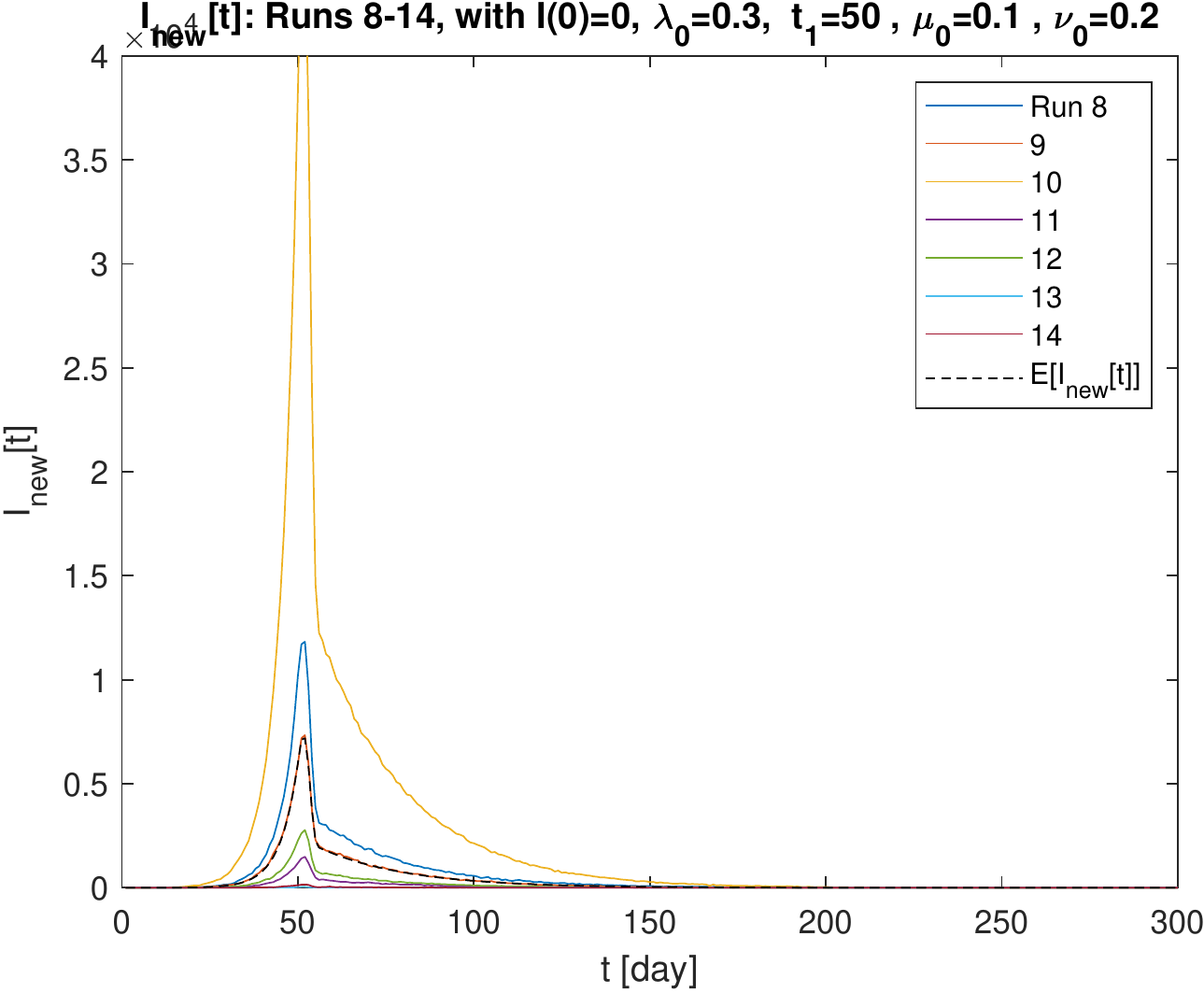}
\caption{\sf New daily infections $I_{new}[t]$, Runs 8-14.}
\label{fig:BDI-I_new[t]-Runs_8-14}
\end{minipage}
\end{figure}
\begin{figure}[hbt]
\begin{minipage}[h]{0.50\textwidth}
\centering
\includegraphics[width=\textwidth]{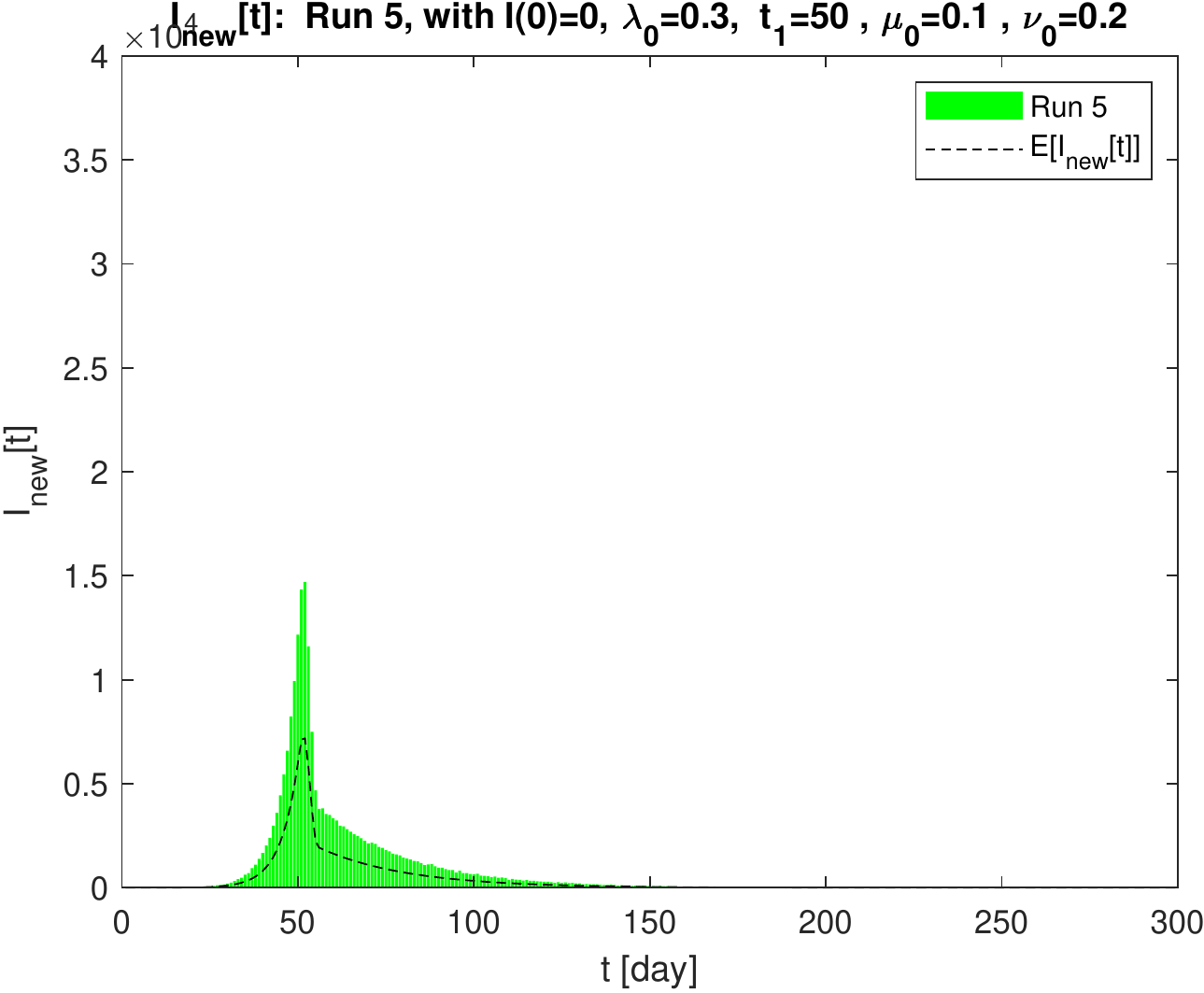}
\caption{\sf New daily infections $I_{new}[t]$, Run 5}
\label{fig:BDI-I_new[t]-Run_5}
\end{minipage}

\end{figure}

\clearpage

\subsection{The Processes $B_{BDI:0}(t)$ and $R_{BDI:0}(t))$}

\begin{enumerate}
\item The variability in the $B(t)$ process is as large as that of $I(t)$.  Among these 14 runs the largest $B(t)$ and the smallest differ by as much as a three order of magnitude (see Figures \ref{fig:BDI-log_B(t)-Runs_1-14}  and
\ref{fig:BDI-log_B(t)-Runs_1-14_25-days}.

\item The process $R(t)$ exhibits similarly large variations among the 14 runs. Although not presented here, our simulation results confirm also the shape of the number of new daily recoveries, denoted $R_{new}[t]$, shown in Figure \ref{fig:New-daily-recoveries}.  They are proportional to $I(t)$ as they should.  The cumulative count of deaths $D(t)$, which is a sub-process of the $R(t)$ behaves similar to $R(t)$.
\end{enumerate}

\begin{figure}[thb]
\begin{minipage}[t]{0.45\textwidth}
\centering
\includegraphics[width=\textwidth]{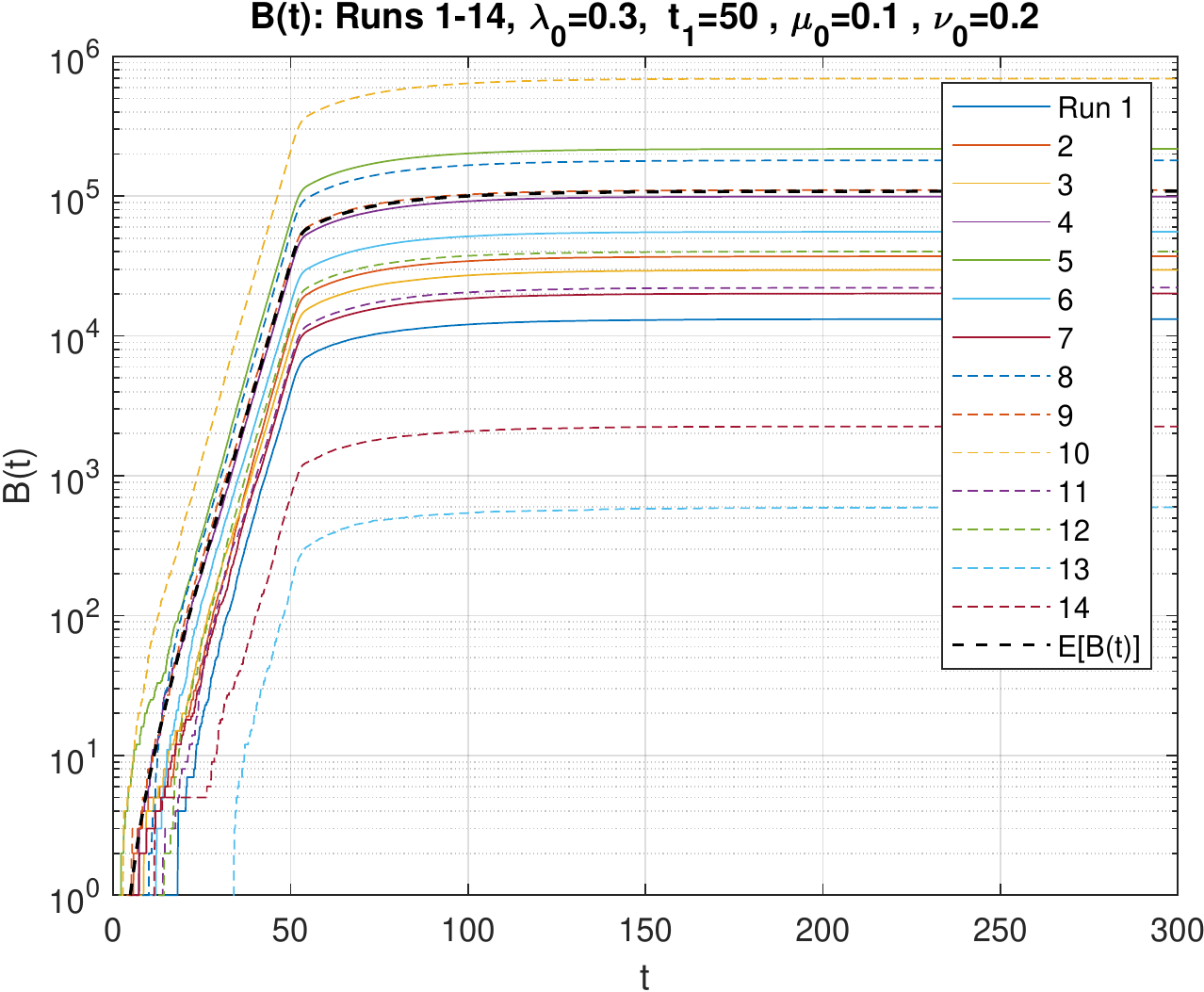}
\caption{\sf Semi-log plot of the cumulative infected $B(t)$, Runs 1-14.}
\label{fig:BDI-log_B(t)-Runs_1-14}
\end{minipage}
\hspace{0.5cm}
\begin{minipage}[t]{0.45\textwidth}
\centering
\includegraphics[width=\textwidth]{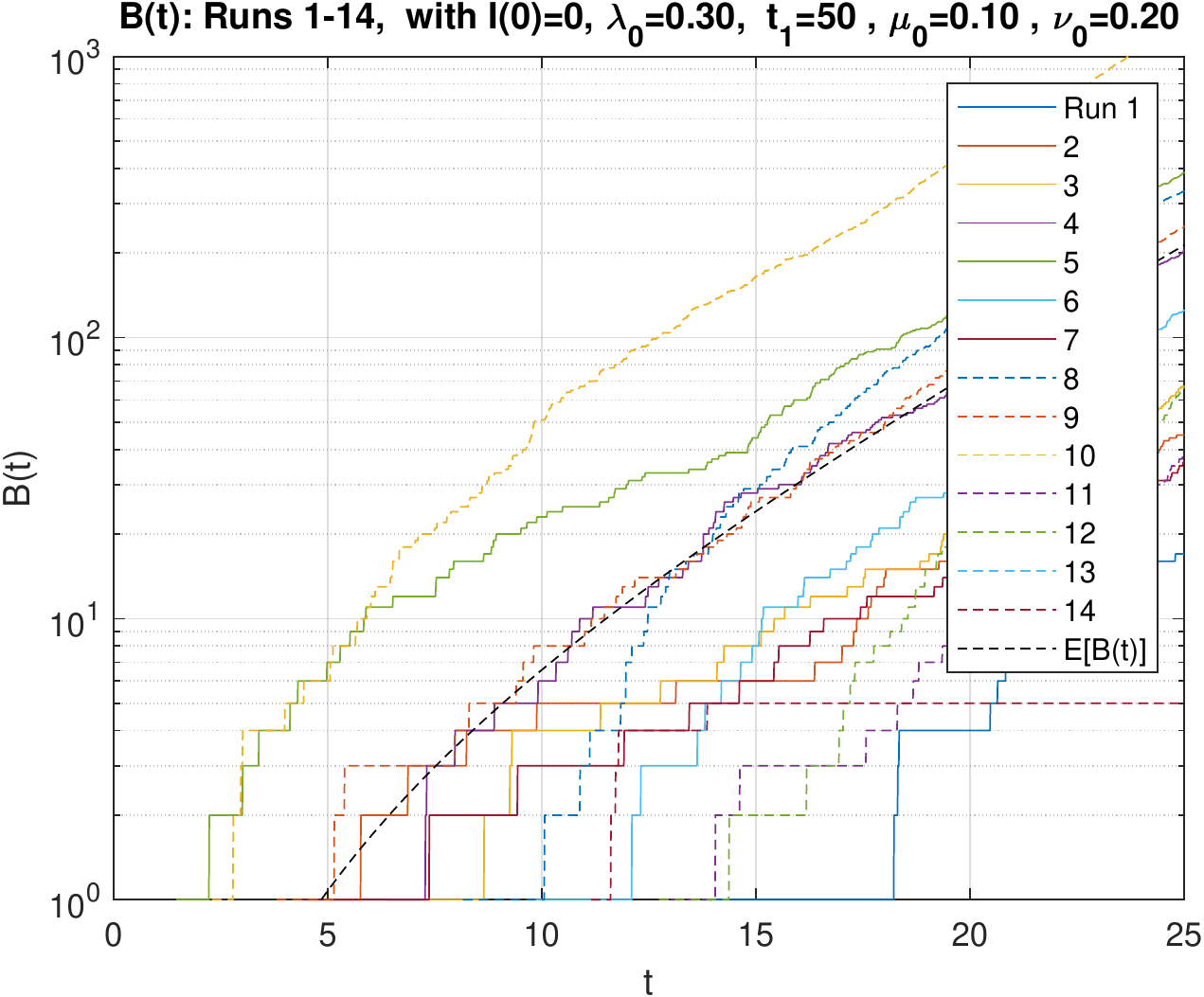}
\caption{\sf Initial 25 days of Semi-log plot of $B(t)$, Runs 1-14.}
\label{fig:BDI-log_B(t)-Runs_1-14_25-days}
\end{minipage}
\end{figure}
\begin{figure}[bht]
\begin{minipage}[b]{0.45\textwidth}
\centering
\includegraphics[width=\textwidth]{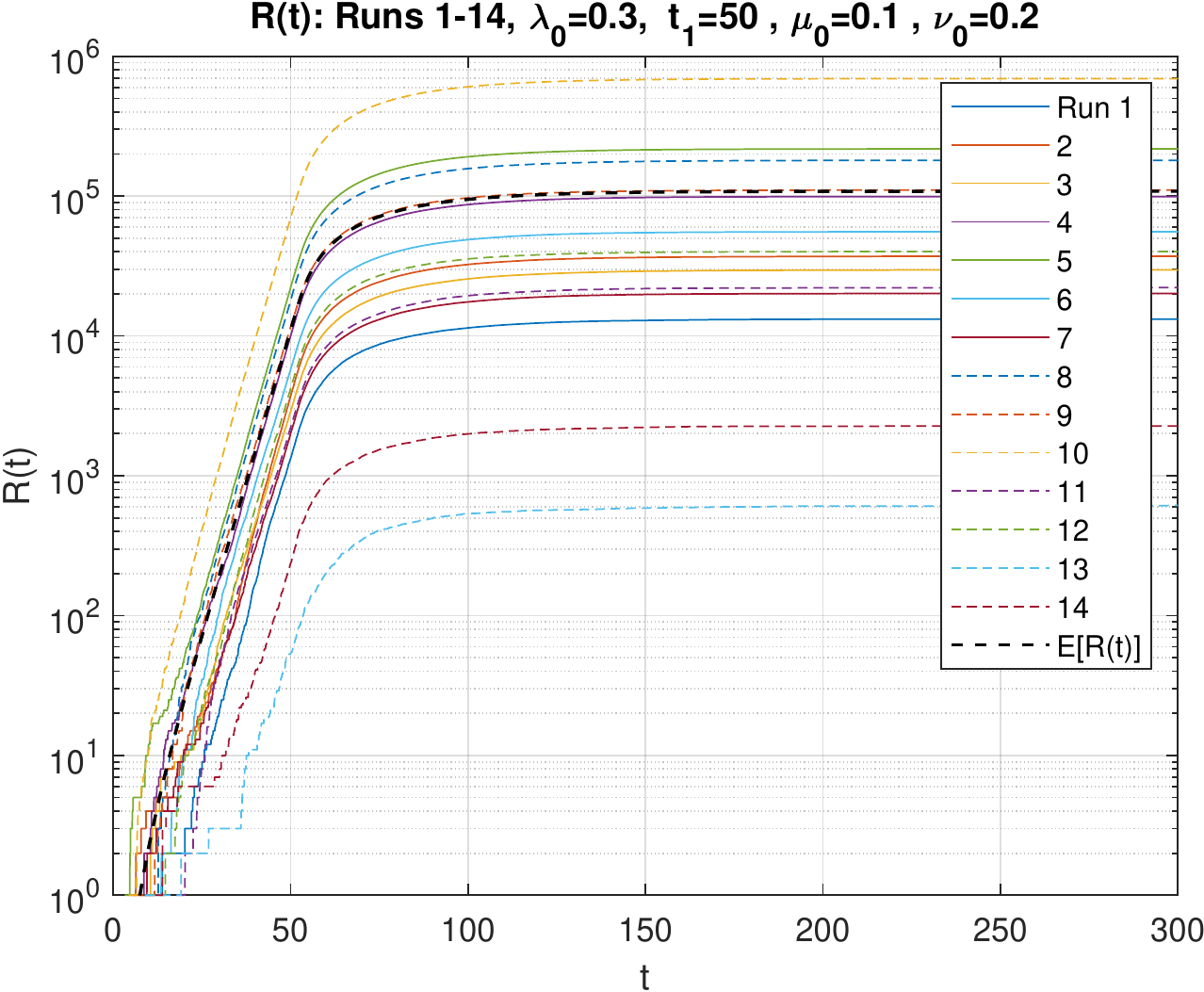}
\caption{\sf Semi-log plot of the cumulative recovered $R(t)$, Runs 1-14.}
\label{fig:BDI-logR(t)-Runs_1-14}
\end{minipage}
\hspace{0.5cm}
\begin{minipage}[b]{0.45\textwidth}
\centering
\includegraphics[width=\textwidth]{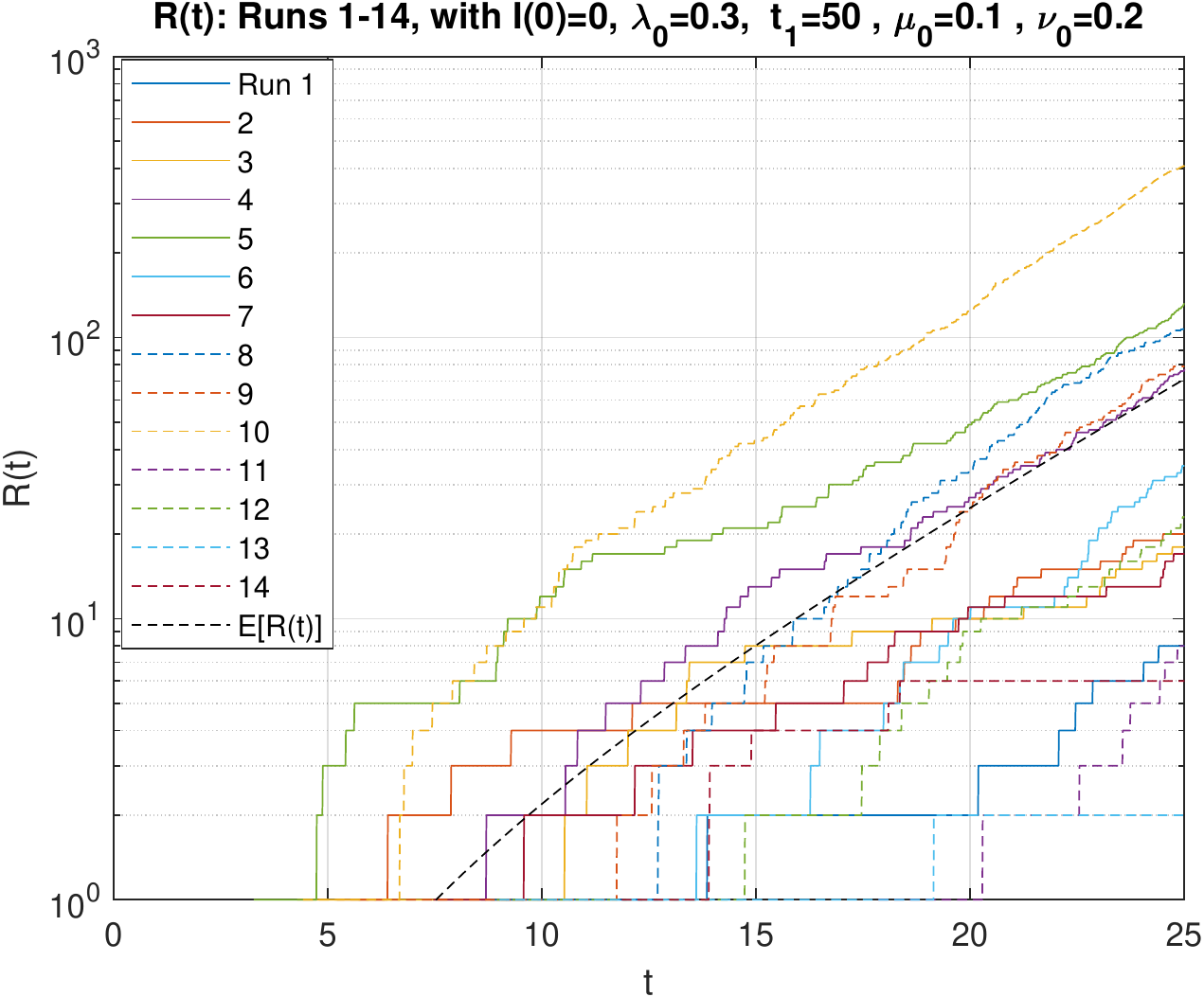}
\caption{\sf Initial 25 days of Semi-log plot of the recovered $R(t)$, Runs 1-14.}
\label{fig:BDI-logR(t)-Runs_1-14_25-days}
\end{minipage}
\end{figure}

\clearpage

\section{Discussion and Future Plans}\label{sec:discussion}

\begin{enumerate}

\item The main result of this report is that we have shown that the PGF of the BDI process with $I_0$  initial population can be represented as
$G_{BDI:I_0}(z,t)=G_{BD:I_0}(z, t)G_{ID:0}(z, t)$, where $G_{BD:I_0}(z,t)$ is the BD process with the initial size $I_0$, and $G_{ID:0}(z,t)$ is the contribution due to immigrants and their descendants. 

\item We have also shown that with the ratio $r(t)=\frac{\nu(t)}{\lambda(t)}$ being kept a constant $r$, the time-dependent PMF (probability mass function) is NBD for all $t$, which is a generalization of the result known heretofore only for the time-homogeneous case.

\item We pursued Bartlett-Bailey's heuristic approach to include the effect of immigration, and have successfully completed their approach by getting the above obtained  solution $G_{BDI:I_0}(z,t)$.

\item For the case where $r(t)\neq r$, we have derived the representation $G_{ID:0}(z,t)=G_{NB(r,\beta)}(z,t)G_c(z,t)$, where the first term is distributed according to NB($r,\beta(t)$), and $G_c(z, t)$ is what we term as a corrective   
process.  Thus, we have the following decomposition of the general BDI process:
\begin{align}
I_{BDI:I_0}(t)=I_{BD:I_0}(t)+I_{I:0}(t)= I_{BD:I_0}(t)+I_{NB(r,\beta)}(t)+I_c(t)
\end{align}
In a forthcoming article \cite{kobayashi:2021bb}, we will present a complete analysis of the process $I_c(t)$.

\item In Sections \ref{sec:nonhomo-BD-simulation}
 and \ref{sec:nonhomo-BDI-simulation}, we reported on simulation results of the BD process analyzed in \cite{kobayashi:2021a} and the BDI process analyzed in Section 1, respectively, and confirmed major findings in the analysis.

\item By extending our results of this article, we will derive the state transition probability, or equivalently the conditional state probability distribution function, i.e., $P^{(BDI)}_{j,k}(u,t)\triangleq \Pr[I_{BDI}(t)=k|I_{BDI}(u)=j]$ of the BDI process. With this information, we will be in a position to predict a future behavior of the process $I(t)$, given its value $I^*$ at an arbitrary instant $t^*\geq 0$. For a large value of $I^*$, however, we may have to resort to an approximation, because an exact computation of probability 
distributions may be computationally expensive.  A saddle-point integration based approximation to convert the PGF to a probability distribution will be pursued \cite{kobayashi:2021d}.  An approach to approximate the process $I(t)$ by a diffusion process or It\^{o} process will be also investigated.

\item Now that we have solved the general time-nonhomogeneous process, our model will be more powerful and useful than has heretofore been expected.  As was claimed earlier, our model has an advantage to the SIR and its variants in that it provides not only probabilistic information, but also is intrinsically linear. Because of the linear property of the model and elegant property that the NBD belongs to the \emph{infinitely divisible distributions} \cite{feller:1968}, we can easily incorporate multiple types of infectious diseases.  The recent development of several variants of COVID-19 should make our modeling approach promising in obtaining accurate analysis and reliable prediction of the behavior of an infectious disease.  Needless to say, the development of a useful estimation algorithm for the model parameters is the most critical step towards a successful application of our model to real situations.  In another forthcoming article \cite{kobayashi:2021c}, we plan to discuss the problem of estimating the model parameters $\lambda(t), \mu(t)$ and $\nu(t)$ from the $I(t)$ and other observable data.  The \emph{Eexpectation-Maximization (EM) algorithm}\footnote{See e.g., \cite{kobayashi-mark-turin:2012}, pp. 559-565.} will be investigated towards this goal.

\end{enumerate}

\appendix

\numberwithin{equation}{section}
\section{The Probability that an Epidemic Terminates}
Let us consider the case with no external arrivals of the infected, i.e., $\nu(t)$ for all $t$.  Then once $I(t)$ reaches zero at some point $T$, then $I(t)$ will be zero for any $t\geq T$. In the population model,   in which the BD process was originally studied, such $T$ is called the time of extinction (of the species under study).  In our context, $T$ is the time when the infection finally comes to an end, i.e.,
\begin{align}
I(t)\left\{\begin{array}{ll}>0&~~\mbox{for}~~ t<T;\\
                            =0&~~\mbox{for}~~ t\geq T.
                             \end{array}\right.
\end{align}
Then, it is not difficult to see that
\begin{align}
\Pr[T\leq t]=\Pr[I(t)=0]. \label{Prob-extinct-by-T}
\end{align}

In Part III-A \cite{kobayashi:2021a}, we showed that 
\begin{align}
\Pr[I(t)=0]\triangleq P_0(t)=\alpha(t)^{I_0}, \label{Prob-I(t)=0}
\end{align}
where $I_0$ is the initial value, i.e., $I(0)$, and the function $\alpha(t)$ was defined by (77), ibid. The numerical plot of $\alpha(t)$ is shown in Figure 14 for our running example, in which $\lambda(t)$ decreases from $\lambda_0=0.3$ to $\lambda_1=0.06$ in the interval from $t_1=50$ till $t_1+d=55$.

Recall that the functions $L(t)$ and $M(t)$ satisfy the following identity (cf. \cite{kobayashi:2021a}, Eqn.(61)).
\begin{align}
L(t)-M(t)+e^{-s(t)}=1,\label{Identity-L-M-s}
\end{align}
Thus, we find an alternative expression for $\alpha(t)$:
\begin{align}
\alpha(t)=1-\frac{1}{e^{s(t)}+L(t)}.  \label{Alt-def-alpha(t)}
\end{align}
The probability that the infection processes eventually terminates will be found by taking the limit $t\to\infty$ :
\begin{align}
\Pr[T<\infty] =\lim_{t\to\infty}P_0(t)=\lim_{t\to\infty}\left(\frac{M(t)}{1+ M(t)}\right)^{I_0}.\label{Prob-T-finite}
\end{align}
Thus, the above probability is one, if and only if
\begin{align}
\lim_{t\to\infty}M(t)=\infty. \label{M(t)-diverge}
\end{align}

As a special case, let us assume $I_0=1$. Since $T$ is a non-negative random variable, $\Ex[T]$ is the area above the area between $\Pr[T\leq t]=\alpha(t)$ and $y=1$ in Figure 4 of \cite{kobayashi:2021a}, i.e.,  
\begin{align}
\Ex[T]=\int_0^\infty \Pr[T>t]\,dt=\int_0^\infty(1-P_0(t))\,dt,
\end{align}
which can be expressed as 
\begin{align}
\Ex[T]=\int_0^\infty \frac{1}{e^{s(t)}+L(t)}\,dt=\int_0^\infty \frac{1}{1+M(t)}\,dt
\end{align}

Recall that the process $I(t)$ is, by definition, a \emph{continuous-time Markov chain} (CTMC)
\footnote{For discussions on Markov chains (or processes),  refer to advanced textbooks on random processes. See e.g., \cite{kobayashi-mark-turin:2012} Chapters 15 \& 16.}.  The non-negative integers 0, 1, 2, \ldots, can be viewed as states, and state 0 of the BD process without immigration is an \emph{absorbing state}.  The time until extinction $T$ discussed above is equivalent to the \emph{first passage time} from state $I_0$ to state 0.

Consider the time-homogeneous case, $\mu(t)=\mu$.  Then
\begin{align}
M(t)=\frac{\lambda}{a}\left(1-e^{-at}\right),~~a=\lambda-\mu.
\end{align}
Then,
\begin{align}
\lim_{t\to\infty}\alpha(t)=\lim_{t\to\infty}\frac{\mu(1-e^{-at})}{a+\mu(1-e^{-at})}
=\left\{\begin{array}{ll}&\frac{\mu}{\lambda},~~\mbox{if}~~a>0;\\
                         & 0,~~\mbox{if}~~a=0;\\
                         & 1,~~\mbox{if}~~a<0,\end{array}\right.
\end{align}
which we obtained in \cite{kobayashi:2021a}, Eqn.(84).

\section{Immigration-and-Death Process: The M(t)/M(t)/$\infty$ Queue}\label{subsec:immigration-death}

Let us assume $\lambda(t)=0$, i.e., no birth, i.e., no internal infections in our context. Since we cannot define $r(t)$ (because $\lambda(t)=0$), we go back to (\ref{G_BDI}). On setting $L(t)=L(u)=0$, we find the PGF of the \emph{immigration and death} process as\footnote{We used in Section 1 the subscript ``ID" to stand for ``immigrants and descendants.''  Its use for ``immigration and death" should be limited in this Appendix only.}
\begin{align}
G_{ID:I_0}(z,t)=\exp\left((z-1)e^{s(t)}\int_0^t\nu(u)e^{-s(u)}\,du\right)\cdot 
\left(1+(z-1)e^{s(t)}\right)^{I_0} \label{G_ID}
\end{align}
Note $s(t)$ is now a function of $\mu(t)$ only:
\begin{align}
e^{s(t)}=e^{-\int_0^t \mu(u)\,du}\triangleq \gamma(t).
\end{align}
Note that the identity formula $L(t)=M(t)+1-s^{-s(t)}$ (see (71) of \cite{kobayashi:2021a}) reduces to
\begin{align}
e^{s(t)}=\frac{1}{M(t)+1}.
\end{align}
This, together with the definition of $\alpha(t)$, leads to the following simple relation:
\begin{align}
\gamma(t)=1-\alpha(t), \label{gamma-alpha}
\end{align}

By defining
\begin{align}
m(t)\triangleq \gamma(t)\int_0^t\frac{\nu(u)}{\gamma(u)}\,du=(1-\alpha(t))\int_0^t\frac{\nu(u)}{1-\alpha(u)}\,du,
\end{align}
we find that the first term of the product form (\ref{G_ID}) corresponds to the contribution by the immigrants who arrived in $(0, t]$:
\begin{align}
G_{I:0}(z, t)=\exp\left((z-1)m(t)\right),
\end{align}
which is the PGF of Poisson distribution with mean $m(t)$.  This solution for the M(t)/M(t)/$\infty$ queue was reported by T. Collngs and C. Stoneman \cite{collings:1976}.

The second term of (\ref{G_ID}) represents the PGF of the \emph{pure death process} with the initial population $I_0$, each of which dies (i.e., departs) independently of each other at rate of $\mu(t)$.  
\begin{align}
G_{D:I_0}(z, t)=\left(1-\gamma(t)+ \gamma(t) z\right)^{I_0}
\end{align}
which gives the binomial distribution:
\begin{align}
P^{(D:I_0)}_k(t)\triangleq \Pr[I_{D:I_0}(t)=k]={I_0\choose k}\gamma(t)^k (1-\gamma(t))^{I_0-k},~~~k=0, 1, 2, \ldots, I_0.
\end{align}

If we assume a constant death rate (i.e., departure rate), $\mu(t)=\mu$, then $\gamma(t)=e^{-\mu t}$ and the mean becomes
\begin{align}
m(t)=e^{-\mu t}\int_0^t \nu(u)e^{\mu u}\,du,
\end{align}
which is found in, e.g., Saaty \cite{saaty:1961} in the analysis of an M(t)/M/$\infty$ queue..

If we assume, in addition, a constant immigration arrival rate $\nu(t)=\nu$, then $m(t)$ reduces to 
\begin{align}
m(t)=\rho\left(1-e^{-\mu t}\right),~~~~\mbox{where}~~\rho=\frac{\nu}{\mu}.
\end{align}
This M/M/$\infty$ queue is discussed in most books on queuing theory and random processes (see e.g., \cite{kobayashi-mark-turin:2012}, pages 422, 703-704, 732-733.

It should be instructive to note that the BDI process gives a negative binomial distribution, whereas its special case, the ID (immigration and death) process gives a Poisson distribution. This is because a negative binomial distribution converges to a Poisson distribution in some limit.\footnote{In the PGF $G(z)=\left(\frac{1-q}{1-qz}\right)^r$ of the negative binomial distribution NB$(r,q)$, let let $q\to 0$ and  $r\to\infty$, so that $rq\sim\nu$, where $\nu$ is a a positive constant.  In the BDI model, this is equivalent to $q\sim\lambda\to 0$.  Then $G(z)
=\left(\frac{1-\frac{\nu}{r}}{1-\frac{\nu z}{r}}\right)^r \to \frac{e^{-\nu}}{e^{-\nu z}}=e^{-\nu(1-z)}$, which is the PGF of a Poisson distribution with mean $\nu$. See e.g., \cite{feller:1968}, p. 281.}

\section*{Acknowledgments}
\addcontentsline{toc}{section}{Acknowledgments}
The author thanks Prof. Brian L. Mark of George Mason University for his advice in use of MATLAB.
He also thanks Dr. Linda Zeger for her careful review of Part III-A.  Her comments and questions helped the author improve the presentation of this report. The author is also thankful to Professor Andrew Viterbi for his careful reading and encouraging opinion of Part I of the report.

\bibliographystyle{ieeetr}
\bibliography{infections}

\end{document}